\documentclass[reqno,a4paper]{amsart}
\usepackage[colorlinks=true,linkcolor=blue,citecolor=blue,urlcolor=blue]{hyperref}
\usepackage[utf8]{inputenc}
\usepackage[UKenglish]{babel}
\usepackage{enumerate}
\usepackage[top=3.0cm, bottom=3.0cm, left=2.5cm, right=2.5cm]{geometry}
\usepackage{graphicx}
\usepackage{amsmath}
\usepackage{amsthm}	
\usepackage{amsfonts}
\usepackage{mathrsfs}	
\usepackage{amssymb}
\usepackage{dsfont}

\theoremstyle{plain}
\newtheorem{theorem}{Theorem}[section]
\newtheorem{lemma}[theorem]{Lemma}

\newtheorem{corollary}[theorem]{Corollary}
\theoremstyle{definition}
\newtheorem{definition}[theorem]{Definition}

\newtheorem{example}[theorem]{Example}
\newtheorem{assumption}[theorem]{Assumption}

\numberwithin{equation}{section}
\numberwithin{figure}{section}

\DeclareMathOperator*{\argmin}{arg\,min}

\date{\today}
\title{Gradient Flows for Regularized Stochastic Control Problems}
\author{David \v{S}i\v{s}ka}
\address{\href{https://www.maths.ed.ac.uk}{School of Mathematics, University of Edinburgh}}
\email{D.Siska@ed.ac.uk}

\author{{\L}ukasz Szpruch}
\address{\href{https://www.maths.ed.ac.uk}{School of Mathematics, University of Edinburgh} and \href{https://www.turing.ac.uk}{The Alan Turing Institute}}
\email{L.Szpruch@ed.ac.uk}

\begin{document}	


\maketitle

\begin{abstract}
This paper studies stochastic control problems with the action space taken to be probability measures, with the objective penalised by the relative entropy. 
We identify suitable metric space on which we construct a gradient flow for the measure-valued control process, in the set of admissible controls,  along which the cost functional is guaranteed to decrease. 
It is shown that any invariant measure of this gradient flow satisfies the Pontryagin optimality principle.
If the problem we work with is sufficiently convex, the gradient flow converges exponentially fast. 
Furthermore, the optimal measure-valued control process admits a Bayesian interpretation which means that one can incorporate prior knowledge when solving such stochastic control problems.   
This work is motivated by a desire to extend the theoretical underpinning for the convergence of stochastic gradient type algorithms widely employed in the reinforcement learning community to solve control problems. 
\end{abstract}


\section{Introduction}
\label{sec introduction}

Stochastic control problems are ubiquitous in technology and science and have been a very active area of research for over half-century \cite{krylov1980controlled,bertsekas1995dynamic,bensoussan2004stochastic,bensoussan2011applications,fleming2006controlled,carmona2018probabilistic}. 
Two classical approaches to tackle the (stochastic) control problem are dynamic programming principle and Pontryagin's optimality principle. 
Either can be used to establish existence (and possibly uniqueness) of solutions to the control problem (this, depending on the context, means either the value function or the optimal control). 
Either approach can form basis for the derivation of approximation methods.
Numerical approximations are almost always needed in practice and rarely scale well with the dimension of the problem at hand. 
Indeed, the term ``curse of dimensionality'' (computational effort growing exponentially with dimension) was coined by R. E. Bellman when considering problems in dynamic optimisation~\cite{bellman1966dynamic}.

This work provides a new perspective by establishing a connection between stochastic control problems and the theory of gradient flows on the space of probability measures and shows how they are fundamentally intertwined.
The connection is reminiscent of stochastic gradient type algorithms widely used in the reinforcement learning community to solve high-dimensional control problems~\cite{doya2000reinforcement, sutton2018reinforcement}. 
We also refer the reader to \cite{gobet2021newton,chassagneux2021learning,kerimkulov2020exponential,kerimkulov2021modified,ito2021neural} for recent work on iterative algorithms for stochastic control problems. 

Our work builds on~\cite{wang2019exploration,reisinger2020regularity} and \cite{ziebart2010modeling,geist2019theory} where entropy-regularised stochastic control problems have been considered in continuous and discrete time settings. 
Further results on entropy regularised control problems can be found in~\cite{tang2022exploratory} and~\cite{huang2022convergence}. 

\subsection{Problem Formulation}
\label{sec problem formulation}

Let $(\Omega^W, \mathcal F^W, \mathbb P^W)$ be a probability space and let $(\mathcal F^W_t)_{t\in[0,T]}$ be a filtration on this space satisfying the usual conditions. 
Let $W$ be a 
$d'$-dimensional Wiener process on this space which is also a martingale w.r.t. the filtration $(\mathcal F^W_t)_{t\in[0,T]}$.
The expectation with respect to the measure $\mathbb P^W$ will be denoted $\mathbb E^W$. 
Given some metric space $\mathcal X$ and $0 \leq q <\infty$, let $\mathcal P_q(\mathcal X)$ denote the set of probability measures defined on $\mathcal X$ with finite $q$-th moment.
Let $\mathcal P_0(\mathcal X)=\mathcal P(\mathcal X)$, the set of probability measures and let $\mathscr M(\mathcal X)$ denote the set of measures on $\mathcal X$. 
Let $\Lambda$ denote the Lebesgue measure.
Throughout the paper we will abuse notation and won't distinguish between a measure and its density (provided it exists). 
For $q\geq 2$ let
\[
 \mathcal M_q \!\!:= \!\!\left\{ \nu \in \mathscr M([0,T] \!\!\times\!\! \mathbb R^p) \!:\!  \text{$\Lambda([0,T])$-a.e.}
\, \exists \nu_t \in \mathcal P(\mathbb R^p)\,\,\text{s.t.}\,\,  \nu(dt,da)=\nu_t(a)\,da\,dt \,, \,\, 
\textstyle{\int_{0}^{T}\!\!\!\int} |a|^q\,\nu_t(da,dt) < \infty \right\}	\\
\]
and
\begin{equation}
\label{eq def Vq}
\begin{split}
\mathcal V_q & := \left\{ \nu : \Omega^W \to  \mathcal M_q : \mathbb E^W\textstyle{\int_{0}^{T}\int} |a|^q\,\nu_t(da,dt) < \infty\right\} \,,\\
\mathcal V_q^W &  := \left\{\nu \in \mathcal V_q :  \text{$(\nu_t)_{t\in[0,T]}$ is progressively measurable w.r.t. $(\mathcal F_t^W)_{t\in[0,T]}$}\right\}\,,\\
\end{split}
\end{equation}
where here and elsewhere, any integral without an explicitly stated domain of integration is over $\mathbb R^p$
 and where we say that $\nu \in \mathcal V_q$ is progressively measurable if for every Borel set $B$ we have that $(\omega,t) \mapsto\nu_t(\omega,B)$ is progressively measurable.
Note that we require that $\mathbb P^W \otimes \Lambda([0,T])$-almost everywhere $\nu\in\mathcal V_q$ has first marginal equal to the Lebesgue measure and the second marginal 
absolutely continuous with respect to the Lebesgue measure i.e. $\nu(\omega)(dt,da) = \nu_t(a)\,da\,dt$.
For $\xi \in \mathbb R^d$ and $\nu \in \mathcal V_2^W$, consider
the controlled process
\begin{equation}
\label{eq process} 
X_t(\nu) = \xi + \int_0^t \Phi_r(X_r(\nu), \nu_r)\,dr
+ \int_0^t \Gamma_r(X_r(\nu),\nu_r)\,dW_r \,,\,\,\, t\in [0,T]\,.
\end{equation}
Fix $m'\in \mathcal P(\mathbb R^p)$ which is absolutely continuous w.r.t the Lebesgue measure. 
We will use $\infty$ to denote the positive infinity.
Let us now define the relative entropy $R:\mathcal P(\mathbb R^p)\times \mathcal P(\mathbb R^p) \to \mathbb R\cup \{\infty\}$ as follows:
for $m \in \mathcal P(\mathbb R^p)$ that are not absolutely continuous with respect to the Lebesgue measure let $R(m|m') = \infty$ while for the  absolutely continuous ones (so that we can write $m(da) = m(a)\,da$)
let
\begin{equation}
\label{eq Ent}
R(m|m') := \int \left[ \log m (a) - \log m'(a)\right]m(a)\,da\,.	
\end{equation}
Let $F$ and $g$ be given.
Let $U=(U_t)_{t\in[0,T]}$ be an $\mathbb R^p$-valued progressively measurable process such that with $\gamma_t = e^{U_t}$ we have $\gamma = (\gamma_t)_{t\in [0,T]} \in \mathcal V^W_2$.
We define the entropy regularized objective functional
\begin{equation}
\label{eq objective bar J}   
\begin{split}
J^\sigma(\nu) & := \mathbb E^{W}\left[\int_0^T \bigg[ F_t(X_t(\nu),\nu_t) + \tfrac{\sigma^2}{2}R(\nu_t|\gamma_t)\bigg]\,dt + g(X_T(\nu)) \bigg| \nu \right]\,.
\end{split}
\end{equation}
Our aim is to minimize, for some fixed $\sigma > 0$ and $\xi$, the objective functional $J^\sigma$ over $\mathcal V_2^W$
subject to the controlled process $X(\nu)$ satisfying~\eqref{eq process}.
Note that if we permitted $\nu_t(\omega)$ singular w.r.t. Lebesgue measure for all $(t,\omega)\in B \in \mathcal B([0,T])\otimes \mathcal F^W$ with $B$ of non-zero measure in the set over which we are minimizing then for such $\nu$ we would have $J^\sigma = \infty$. 
In other words such ``singular'' $\nu$ will never be optimal for the regularized problem. 
So not allowing such controls leads to no loss of generality.
Our setup encompasses the setting of stochastic relaxed controls that dates back to the work of L.C. Young on generalised solutions of problems in the calculus of variations~\cite{young2000lectures}.  

Historically, in control theory, measure-valued controls have been used as a mathematical tool for proving the existence of solutions to the relaxed control problems associated with the original strict control problem.
On the other hand, in the theory of Markov Decision Processes (MDP) it is common to seek solutions within a class of probability measures, \cite{bertsekas2004stochastic,sutton2018reinforcement} as this often improves stability and efficiency of algorithms used to solve the MDP. 
It is only very recently that regularised relaxed control problems have been studied in the differential control setting, see~\cite{wang2019exploration} where the regularised relaxed linear-quadratic stochastic control problem is studied in great detail.   
In~\cite{reisinger2020regularity} it is proved that for sufficiently regularised stochastic control problems, the optimal Markov control function is smooth. 
This is in sharp contrast to unregularized control problems for which controls are often discontinuous (only measurable) functions.

Let us now introduce the Hamiltonians
\begin{equation}
\label{eq hamiltonian}
\begin{split}
H^0_t(x,y,z,m) & := \Phi_t(x,m)y + \text{tr}(\Gamma_t^\top(x,m) z) + F_t(x,m)\,, 	\\
H^\sigma_t(x,y,z,m,m') & := H^0_t(x,y,z,m) + \tfrac{\sigma^2}{2}R(m|m')\,.	\\
\end{split}
\end{equation}
 We work with Pontryagin's optimality principle and hence we use the adjoint processes
\begin{equation}
\label{eq adjoint proc}
\begin{split}
dY_t(\mu) & =-(\nabla_x H^0_t)(X_t(\mu),Y_t(\mu),Z_t(\mu),\mu_t)\,dt + Z_t(\mu)\,dW_t \,,\,\,\,t\in [0,T]\,,\,\,\,
Y_T(\mu) = (\nabla_x g)(X_T(\mu))\,
\end{split}
\end{equation}
and note that (trivially) $\nabla_x H^0 = \nabla_x H^\sigma$.
Under the assumptions we postulate in this work, for each $\xi \in \mathbb R^d$ and $\mu \in \mathcal V_2^W$ a unique solution $ Y(\mu)$ and $Z(\mu)$ exists, see Lemma \ref{lemma existence for fixed control}.

The necessary condition for optimality which we formulate precisely in Theorem \ref{thm necessary cond linear} says that if $\nu\in \mathcal V^W_{2}$ is (locally) optimal for $J^\sigma$, $X(\nu)$ and $Y(\nu)$, $Z(\nu)$ are
the associated optimally controlled state and adjoint processes respectively,
then for a.e. $(\omega,t)\in \Omega^W \times (0,T)$ $\nu_t$  locally minimises (point-wise) $H^\sigma (X_t(\nu),Y_t(\nu),Z_t(\nu),\nu,\gamma)$. 
In many cases the solution to this optimisation problem cannot be found explicitly. Motivated by the success of various gradient descent algorithms used in reinforcement learning to solve control problems~\cite{doya2000reinforcement, sutton2018reinforcement} our contribution here is to construct an appropriate gradient flow and study its convergence to the optimal solution of the stochastic control problem \eqref{eq process}--\eqref{eq objective bar J}.
Since the Hamiltonian $H^\sigma$ is optimised over (probability) measures we will require a notion of the linear functional derivative, see~\cite[Definition 5.43]{carmona2018probabilistic},
and will define, for $\mu \in \mathcal V_2^W$, $t\in [0,T]$, $a\in \mathbb R^p$ and $\xi \in \mathbb R^d$:
\begin{equation}
\label{eq fat h}
\tfrac{\delta \mathbf H_t^{0}}{\delta m}(\mu,a):=\tfrac{\delta H_t^{0}}{\delta m}(X_t(\mu), Y_t(\mu), Z_t(\mu), \mu_t, a)\,.
\end{equation}
Note that the stochastic process $(\tfrac{\delta \mathbf H_t^{0}}{\delta m}(\mu,a))_{t\in [0,T]}$ is $(\mathcal F_t^W)_{t\in[0,T]}$-adapted for every $\mu \in \mathcal V_2^W$ and $a\in \mathbb R^p$. 
Moreover, let 
\begin{equation}
\label{eq fat h sigma}
\tfrac{\delta \mathbf H_t^{\sigma}}{\delta m}(\mu,a) :=
\tfrac{\delta \mathbf H_t^{0}}{\delta m}(\mu,a) + \tfrac{\sigma^2}{2} \left( U_t(a) +  \log \mu_{t}(a)  \right)\,.
\end{equation}
Note that, in~\eqref{eq fat h sigma}, the right-hand side is not always the linear functional derivative as the entropy, being only lower semicontinuous, does not have a flat derivative on the entire $\mathcal P(\mathbb R^p)$. 
However this abuse of notation makes clear what role the left-hand side of~\eqref{eq fat h sigma} plays and, as long as it is only applied on the appropriate gradient flow, this notational choice is justified by Lemma \ref{lemma derivative of entropy along flow}.

\subsection{Heuristic derivation of the gradient flow}
 
Recall that $t\in [0,T]$ is the time associated with our control problem.
We introduce a new time $s\in [0,\infty)$ which will be the gradient flow time.
Our aim is to minimize $J^{\sigma}:\mathcal V^{W}_2 \rightarrow \mathbb R$  defined in~\eqref{eq objective bar J} using some gradient flow equation for the evolution of $\nu_{s,\cdot} \in \mathcal V_2^W$.  
More precisely, let $E:[0,\infty) \times \Omega^W \times [0,T] \times  \mathbb R^p \rightarrow \mathbb R^p$ be a vector field depending on the gradient flow time $s\in [0,\infty)$, the control problem time $t\in [0,T]$ and $\omega^W \in \Omega^W$ and which is such that $E_{s,t}(a)$ is $\mathcal F_t^W$-measurable so that $\nu_{s,\cdot}$ can be admissible. 
Dynamic representation of the Wasserstein distance due to Benamou and Brenier \cite{benamou2000computational}, suggest to consider the continuity equation 
\begin{equation}
\label{eq continuity eq}
\partial_s \nu_{s,t} = \nabla_a \cdot \left( E_{s,t} \nu_{s,t} \right)\,,\,\,\, s\in [0,\infty)\,,\,\,\, \nu_{0,t} 	\in \mathcal P_2(\mathbb R^p)\,.
\end{equation}
We wish to identify the vector field $E$ so that $J^{\sigma}(\nu_{s,\cdot})$ decreases as $s$ increases. 
We do this by following the method of ``free energy'' dissipation studied with Otto calculus as presented in~\cite[Section 3]{OV00} and~\cite[Chapter 15]{villani2008optimal}. 

For fixed $\varepsilon,\lambda>0$ let $\nu_s^{\lambda,\varepsilon}:=  \nu_s + \lambda(\nu_{s+\varepsilon} - \nu_s)$.
Using the notion of linear functional derivative we see that 
\[
\partial_s J^{\sigma}(\nu_{s,\cdot}) = \lim_{\varepsilon \rightarrow 0} \varepsilon^{-1}\bigg( J^{\sigma}(\nu_{s + \varepsilon,\cdot}) - J^{\sigma}(\nu_{s,\cdot} )\bigg) 
 =  \lim_{\varepsilon \rightarrow 0} \varepsilon^{-1}\bigg( \int_0^1 \mathbb E^W\int_0^T \int \tfrac{\delta J^{\sigma}}{\delta \nu}(\nu^{\lambda,\varepsilon}_{s,\cdot},a) (\nu_{s+\varepsilon,\cdot} - \nu_{s,\cdot} )  \,dt\, d\lambda\bigg)\,.
\]
Due to Lemma~\ref{lemma der of J as hamiltonian flat} and ignoring momentarily that entropy is only lower semi-continuous, one can deduce that
\[
\partial_s J^{\sigma}(\nu_{s,\cdot})
 =  \lim_{\varepsilon \rightarrow 0} \varepsilon^{-1}\bigg( \int^{1}_0 
\mathbb E^W \int_0^T \left[  \int  \tfrac{\delta \mathbf  H^{\sigma}}{\delta m}(\nu^{\lambda,\varepsilon}_{s,t},a) (\nu_{s+\varepsilon,t} - \nu_{s,t} ) (da)\right] dt d\lambda\bigg) \,.
\]
Continuity of $\tfrac{\delta \mathbf  H^{\sigma}}{\delta \nu}$ in measure,  and the fact that for all $\lambda>0$ we have $\nu_s^{\lambda,\varepsilon} \rightarrow \nu_s$  as $\varepsilon \rightarrow 0$, yields
\[
\begin{split}
&\partial_s J^{\sigma}(\nu_{s,\cdot}) 
 =  
\mathbb E^W \int_0^T \left[  \int  \tfrac{\delta \mathbf  H^{\sigma}}{\delta m}(\nu_{s,t},a) \partial_s \nu_{s,t}  (da)\right] dt 
= \mathbb E^W \int_0^T \left[  \int  \tfrac{\delta \mathbf  H^{\sigma}}{\delta m}(\nu_{s,t},a) \,\,  \nabla_a \cdot \left( E_{s,t} \nu_{s,t} \right) (da)\right] dt \,.
\end{split}
\]
Formal integration by parts implies that
\[
\begin{split}
&\partial_s J^{\sigma}(\nu_{s,\cdot}) 
 =  - \mathbb   E^W \int_0^T \left[  \int  (\nabla_ a \tfrac{\delta \mathbf  H^{\sigma}}{\delta m})(\nu_{s,t},a)  \left( E_{s,t} \nu_{s,t} \right) (da)\right] dt \,.
\end{split}
\]
Hence if we take $E_{s,t} := (\nabla_ a \tfrac{\delta \mathbf  H^{\sigma}}{\delta m})(\nu_{s,t},\cdot)$
then 
\[
\partial_s J^{\sigma}(\nu_{s,\cdot}) 
 =  - \mathbb   E^W \int_0^T \bigg[  \int  \bigg|\bigg(\nabla_ a \tfrac{\delta \mathbf  H^{\sigma}}{\delta m}\bigg)(\nu_{s,t},a)\bigg|^2 \nu_{s,t}(da)\bigg] dt \leq 0.
\]
From the definition of $\tfrac{\delta \mathbf  H^{\sigma}}{\delta \nu}$ in \eqref{eq fat h sigma}, the continuity equation~\eqref{eq continuity eq} can thus be written as
\begin{equation} \label{eq pde heuristic}
\partial_s \nu_{s,t} 
= \nabla_a \cdot \bigg(\left((\nabla_ a \tfrac{\delta \mathbf  H^{0}}{\delta m})(\nu_{s,t},\cdot) + \tfrac{\sigma^2}{2}(\nabla_a U_t) \right) \nu_{s,t} + \tfrac{\sigma^2}{2} \nabla_a \nu_{s,t} \bigg)  \,.	
\end{equation}
Thus we see that for each $(\omega^W, t)\in \Omega^W \times [0,T]$ the Hamiltonian~\eqref{eq hamiltonian} can be viewed as the potential function of the gradient flow~\eqref{eq pde heuristic}. 
If $\nu^\ast$ is a stationary solution to~\eqref{eq pde heuristic} then clearly $(\nabla_ a \tfrac{\delta \mathbf  H^{\sigma}}{\delta m})=0$ $\nu^\ast$-a.s. and hence $\nu^\ast$
satisfies the first order condition 
\begin{equation}
\label{eq foc intro}
\tfrac{\delta \mathbf  H^{\sigma}}{\delta m}(\nu^{\ast}_{s,t},a) \,\,\, \text{is a constant    } \nu^{\ast}-a.s\,.	
\end{equation}
This implies that $\nu^\ast$ is a Gibbs measure in the sense that it satisfies the equation 
\begin{equation} 
\label{eq nuast heuristic}
\nu_t^\ast(a) = \mathcal Z^{-1}_t e^{ -\frac{2}{\sigma^2} \tfrac{\delta \mathbf H_t^{0}}{\delta m}(\nu^\ast,a)}\gamma_t(a)\,,\quad \mathcal Z_t := \int e^{ -\frac{2}{\sigma^2} \tfrac{\delta \mathbf H_t^{0}}{\delta m}(\nu^\ast,a)}\gamma_t(a)da
\end{equation}
for a.e. $(\omega^W, t) \in \Omega^W \times (0,T)$.
Moreover such a stationary solution ought to be a local minimizer of $J^\sigma$ since $\partial_s J^\sigma(\nu^\ast) = 0$. 
The optimal control for \eqref{eq process}-\eqref{eq objective bar J} solves equation \eqref{eq nuast heuristic} and hence enjoys a Bayesian interpretation with $\gamma$ and $\nu^\ast$ being the prior and the posterior distributions, respectively.

In Theorem~\ref{thm unique minimiser} we show that 
if the gradient flow converges, then the limit point satisfies the first order condition~\eqref{eq nuast heuristic} and moreover we show that if the invariant measure is unique then $\nu^\ast$ is the unique global minimiser of the objective functional~\eqref{eq objective bar J}. 
These results are proved under mild assumptions: in essence those required for the necessary conditions for optimality arising from Pontryagin's maximum principle stated in Theorem~\ref{thm necessary cond linear}.
Theorem~\ref{thm conv to inv meas rate} tells us that for sufficiently convex stochastic control problems or if the entropic regularisation is sufficiently strong then the gradient flow converges exponentially to the unique invariant measure.
This work extends the analysis in~\cite{hu2019meanode} and~\cite{jabir2019mean}, that considered the gradient flow perspective for solving differential control problems with ODE dynamics and was motivated by the desire to develop a mathematical theory of deep learning. 

\subsection{Probabilistic representation}
Next we explain how the gradient flow \eqref{eq pde heuristic} relates to a familiar noisy gradient descent. For each fixed $t\in [0,T]$ and $\omega \in \Omega^W$ we know that the Kolmogorov--Fokker--Planck equation~\eqref{eq fpe} has a stochastic representation which we will introduce below.
Let there be $(\Omega^B, \mathcal F^B, \mathbb P^B)$ equipped with a $\mathbb R^p$-Brownian motion $B=(B_s)_{s\geq 0}$ and the filtration $\mathbb F^B = (\mathcal F_t^B)$ where
$\mathcal F_s^B := \sigma(B_u:0\leq u \leq s)$.
Let $(\Omega^0,\mathcal F^0,\mathbb P^0)$ be a probability space on which $\theta^0 = \theta^0_t(\omega^W,\cdot)$ and $\theta'^0 = \theta'^0_t(\omega^W,\cdot)$ are random variables for each $(\omega^W, t)$. 
Let $\Omega := \Omega^W \times \Omega^0 \times \Omega^B$, $\mathcal F := \mathcal F^W \otimes \mathcal F^0 \otimes \mathcal F^B$
and $\mathbb P := \mathbb P^W \otimes \mathbb P^0 \otimes \mathbb P^B$.

We will use $\mathcal L(\cdot |\mathcal F^W_t)$ to denote the conditional law of a random variable on $(\Omega, \mathcal F, \mathbb P)$ conditioned on $\mathcal F^W_t$.
Let $(\theta^0_t)_{t\in [0,T]}$ be an $(\mathcal F_t^W)$-adapted, $\mathbb R^p$-valued stochastic process on $(\Omega,\mathcal F, \mathbb P)$ s.t. $\mathbb E\int_0^T|\theta^0_t|^2\,dt < \infty $ so that $(\mathcal L(\theta_{t}^0 | \mathcal F_t^W))_{t\in [0,T]} \in \mathcal V_2^W$ and consider 
\begin{equation}
\label{eq mfsgd}
	d \theta_{s,t}  =  - \Big((\nabla_a \tfrac{\delta  H_t^{0}}{\delta m})(X_{s,t}, Y_{s,t}, Z_{s,t}, \nu_{s,t}, \theta_{s,t})
+ \tfrac{\sigma^2}{2}(\nabla_a U)(\theta_{s,t}) \Big)\,ds + \sigma dB_s\,, \quad s \geq 0\,,\,\,\,\theta_{0,t} = \theta_t^0\,,
	\end{equation}
where
\begin{equation}\label{eq mfsgd 2}
\begin{cases}
\nu_{s,t} & = \mathcal L(\theta_{s,t} | \mathcal F_t^W )\,, \\
	X_{s,t}(\nu) &= \xi + \int_0^t \Phi_r(X_{s,r}(\nu),\nu_{s,r})\,dr
+ \int_0^t  \Gamma_r(X_{s,r}(\nu),\nu_{s,r}(da))\,dW_r \,,\,\,\, t\in [0,T]\,, \\
d Y_{s,t}(\nu) & = - (\nabla_x H^{0}_t)(X_{s,t}(\nu),Y_{s,t}(\nu),Z_{s,t},\nu_{s,t})\,dt + Z_{s,t}(\nu)\,dW_t \,,\,\,\,\, 
Y_{s,T}(\nu)  = (\nabla_x g)(X_T(\nu))\,.
\end{cases}
\end{equation}
In what follows we write $(X(\nu),Y(\nu),Z(\nu))=(X,Y,Z)$ if the dependence on $\nu$ is clear from the context. 
Notice that \eqref{eq mfsgd 2} is a decoupled FBSDE system in the sense that at any $s\geq 0$ the forward process $(X_{s,t}(\nu))_{t\in [0,T]}$ doesn't depend on $(Y_{s,t}(\nu))_{t\in [0,T]}$ or $(Z_{s,t}(\nu))_{t\in [0,T]}$ directly, only through the control $\nu_{s,\cdot}$. 
The implication of this decoupling becomes clear when one considers an explicit time discretisation of~\eqref{eq mfsgd}, as we shall describe below, or an implicit time discretisation of~\eqref{eq mfsgd}, which would correspond to a modified method of successive approximations (MSA) algorithm, see e.g.~\cite{kerimkulov2021modified}.

The gradient flow, together with probabilistic numerical methods, provides a basis for a new class of algorithms for solving data-dependent stochastic control problems in high dimensions.  
Indeed, the system \eqref{eq mfsgd}--\eqref{eq mfsgd 2} can be used to design a complete algorithm as follows. First one could approximate \eqref{eq mfsgd} by a discrete time particle system, say $(\theta^i_{s_k,t})_k $ for $i=1,\ldots, N$.  
Given $\theta^i_{s_k,t} $ for $i=1,\ldots, N$ one computes $(\nu_{s_k,t}^N)_k$ where $\nu_{s_k,t}^N= \frac{1}{N}\sum_{i=1}^N\delta_{\theta^i_{s_k,t} }$. 
At the next step one would then solve forward process $(X_{s_k,t}(\nu_{s_k,\cdot}^N))_k$ on $[0,T]$ and then ``back-propagate'' $(Y_{s_k,t}(\nu_{s_k,\cdot}^N))_k$ on $[0,T]$. 
Finally one would update the gradient term. 
At this point one can obtain the next value of the particle system approximating the control: $\theta^i_{s_{k+1},t}$. 
In this work we do not study this numerical approximation but refer a reader to \cite{jabir2019mean} where this has been analysed for the deterministic control problem with gradient flow approximated by Euler scheme and usual interacting particle system. 
We refer the reader to~\cite{chassagneux2019weak,szpruch2019antithetic,delarue2021uniform} for recent progress on particle approximations for related McKean--Vlasov SDEs.

This paper is organised as follow: 
in Section~\ref{sec assumptions and results} we state the assumptions and announce the main results proved in this paper.
Section~\ref{sec proof of opt cond} is devoted to the proof of Pontryagin optimality condition in the setting of this paper. 
In Section~\ref{sec ex uniq and conve} we prove that the gradient flow system has unique solution and converges to invariant measure. 

\section{Assumptions and Results}
\label{sec assumptions and results}

There are two main sets of results. 
In the first, we use the necessary condition arising from the Pontryagin optimality principle to characterise the optimal control
and we show that if the gradient system~\eqref{eq mfsgd}-\eqref{eq mfsgd 2} converges to an invariant measure $\mu^\ast \in \mathcal V_2^W$ then this measure is an optimal control. 
In the second, we will present conditions for existence of solutions to the gradient system~\eqref{eq mfsgd}-\eqref{eq mfsgd 2} and conditions for its convergence to the invariant measure $\mu^\ast \in \mathcal V_2^W$.

\subsection{Characterisation of the optimal control}
We begin by formalising the definition of gradient flow PDEs.

\begin{definition}\label{def vect field flow def}
We will say that $b:[0,\infty)\times  [0,T] \times \Omega^W \times \mathbb R^p \to \mathbb R^p$ is a {\em permissible flow} if for a.e. $(\omega^W,t)$ we have
 $b_{\cdot,t}(\omega^W,\cdot) \in C^{0,1}([0,\infty)\times\mathbb R^p; \mathbb R^p)$ and for all $s$ and a.e. $(\omega^W,t)$ the function $a\mapsto b_{s,t}(\omega^W,a)$ is of linear growth and for any $s\geq 0$ and $a\in\mathbb R^p$ the random variable $b_{s,t}(a)$ is $\mathcal F_t^W$-measurable.  	
\end{definition}
We do not expect $b=b_{s,t}(\omega^W,a)$ to have any regularity in  $(\omega^W,t)$.
Indeed from the heuristic derivation in Section~\ref{sec introduction} it is clear that this term will involve the gradient of the flat derivative of the Hamiltonian for every $(\omega^W,t)$.

\begin{lemma}
\label{lem pde}
If $b$ is a permissible flow (cf. Definition~\ref{def vect field flow def}) then the linear PDE 	
\begin{equation}
\label{eq fpe}
\partial_s \nu_{s,t} = \nabla_a \cdot \left( b_{s,t} \nu_{s,t} + \tfrac{\sigma^2}{2}\nabla_a \nu_{s,t}\right)\,,\,\,\, s\in [0,\infty)\,,\,\,\, \nu_{0,t} 	\in \mathcal P_2(\mathbb R^p)\,
\end{equation}
has  unique solution $\nu_{\,\cdot,t} \in C^{1,\infty}((0,\infty)\times \mathbb R^p; \mathbb R)$ for each $t\in [0,T]$ and $\omega^W \in \Omega^W$.
Moreover for each $s > 0$, $t\in [0,T]$  and $\omega^W \in \Omega^W$ we have $\nu_{s,t}(a) > 0$ and $\nu_{s,t}(a)$ is $\mathcal F_t^W$-measurable. 
\end{lemma}

The proof of Lemma~\ref{lem pde} will be given in Section~\ref{sec proof of opt cond}.

\begin{assumption}[For characterisation of the optimal control] \label{as coefficients new}
Let $G:[0,T]\times \mathbb R^d \times \mathcal P_2(\mathbb R^p) \to \mathbb R^k$ stand for any of $\Phi$, $\Gamma$ or $F$ with $k=d$, $k=d\times d'$ or $k=1$ respectively. 
Let $K>0$ be given.  
\begin{enumerate}[i)]
\item For all $t\in [0,T]$ we have $|G_t(0,\delta_0)|\leq K$ and $|g(0)|\leq K$.
\item The function $G$ is differentiable in $x$ for every $(t,m) \in [0,T]\times \mathcal P_2(\mathbb R^p)$. 
The derivatives are jointly continuous.
\item For all $(t,x,m) \in [0,T]  \times \mathbb R^d \times \mathcal P_2(\mathbb R^p)$ we have   
$
|\nabla_x \Phi_t(x,m)| + |\nabla_x \Gamma_t(x,m)| \leq K\,.
$
\item For all $(t,x,x',m) \in [0,T]  \times \mathbb R^d \times \mathbb R^d \times \mathcal P_2(\mathbb R^p)$ we have   
\[
|\nabla_x \Phi_t(x,m)-\nabla_x \Phi_t(x',m)| + |\nabla_x \Gamma_t(x,m) - \nabla_x \Gamma_t(x',m)| \leq K|x-x'|\,.
\]

\item For each $(t,x,m) \in [0,T] \times \mathbb R^d \times \mathcal P_2(\mathbb R^p)$ the linear functional derivative $\tfrac{\delta G_t}{\delta m}$ exists and is jointly continous. 	
\item  For each $(t,x,a) \in [0,T] \times \mathbb R^d \times \mathbb R^p$ the linear functional derivatives $\tfrac{\delta^2 G_t}{\delta m^2}$ exist and for all $(t,x,a,a') \in [0,T] \times \mathbb R^d \times \mathbb R^p \times \mathbb R^p$ we have
$
\left|\tfrac{\delta^2 G_t}{\delta m^2}(x,m,a,a')\right|\leq K\,.
$
\item The function $g$ is twice differentiable and for all $x \in \mathbb R^d$ we have $| \nabla^2_x g(x)| \leq K$.
\item For all $(t,m,a)\in  [0,T]  \times \mathcal P_2(\mathbb R^p)\times \mathbb R^p $ the function $\tfrac{\delta F_t}{\delta m}$ is continuously differentiable in $x$ and for all $(t,x,m,a)\in  [0,T]  \times \mathbb R^d \times \mathcal P_2(\mathbb R^p)\times \mathbb R^p $ we have 
$
\left|\nabla_x\tfrac{\delta F_t}{\delta m}(x,m,a) \right| \leq K\,.
$

\end{enumerate}
	
\end{assumption}

Note that Assumption~\ref{as coefficients new}, i), ii), iii) and iv) imply that the coefficients of~\eqref{eq process} are Lipschitz continuous in $x$ uniformly in 
$(t,m) \in [0,T]\times \mathcal P_2(\mathbb R^p)$ and that they have linear growth in $x$ uniformly in 
$(t,m) \in [0,T]\times \mathcal P_2(\mathbb R^p)$.

\begin{lemma}
\label{lemma existence for fixed control}
Let Assumption~\ref{as coefficients new} hold.   
Then for any $q\geq 2$ and $\mu\in \mathcal V_q^W$ and $\xi \in \mathbb R^d$ the equation~\eqref{eq process} has a unique solution $X(\mu)$ which is adapted to $(\mathcal F_t^W)_{t\in [0,T]}$ and 
there is $c>0$ such that 
\[
\sup_{\mu\in \mathcal V_2^W}\mathbb E^W\bigg[\sup_{t\in[0,T]}|X_{t}(\mu)|^q \bigg] <  c(1 +|\xi|^q )\,.
\]
Moreover~\eqref{eq adjoint proc} has unique solution $(Y(\mu),Z(\mu))$ which is adapted to $(\mathcal F_t^W)_{t\in [0,T]}$ and  
 $Y(\mu)\in L^2((0,T)\times \Omega; \mathbb R^d)$ and $Z(\mu) \in L^2((0,T)\times \Omega; \mathbb R^{d\times d'})$.
\end{lemma}
We will not prove Lemma~\ref{lemma existence for fixed control} since the existence of a unique solution to~\eqref{eq process} and the moment bound state above is classical and 
can be found e.g. in Krylov~\cite{krylov1980controlled}. 
The adjoint equation~\eqref{eq adjoint proc} is affine, hence the coefficients are Lipschitz continuous in $y$ and $z$. 
Due to Assumption \ref{as coefficients new} and utilising the moment bound for $X$, we get existence, uniqueness and the stated integrability 
from e.g. from Zhang~\cite[Th 4.3.1]{zhang2017backward}.

Let us introduce the following metrics and spaces. First, for $\mu, \mu' \in \mathcal M_q$ let
\begin{equation}
\label{eq def of rho}
\mathcal W_q^T(\mu,\mu') := \bigg(\int_0^T \mathcal W_q(\mu_t,\mu_t')^q\,dt\bigg)^{1/q}\,,	
\end{equation}
where $\mathcal W_q$ denotes the usual $q$-Wasserstein metric in $\mathcal P_q(\mathbb R^p)$.
Note that $(\mathcal M_q, \mathcal W_q^T)$ is a complete metric space. 
For $\mu, \mu' \in  \mathcal V_q^W$ let
\[
\rho_q (\mu, \mu') = \left(\mathbb E^W\left[|\mathcal W_q^T(\mu,\mu')|^q\right]\right)^{1/q}\,. 
\]

The next result is reminiscent of the study of  ``dissipation of free energy'' along $\mathcal W_2$-gradient flow as in~\cite[Section 3]{OV00} and~\cite[Chapter 15]{villani2008optimal}, but in the setting of stochastic control with gradient flow \eqref{eq fpe} in the metric space $\big(\mathcal V_2^W, \rho_2 \big)$. 

\begin{theorem}
\label{th with measure flow}
Fix $\sigma \geq 0$ and let Assumption~\ref{as coefficients new} 
hold.
Let $b$ be a permissible flow (c.~f.~Definition~\ref{def vect field flow def})
such that
$a\mapsto |\nabla_a b_{s,t}(a)|$ is bounded uniformly in $s,t$ and $\omega^W \in \Omega^W$.
Let $\nu_{s,t}$ be the solution to~\eqref{eq fpe}.
Assume that $X_{s,\cdot}, Y_{s,\cdot}, Z_{s,\cdot}$ are the forward and backward processes arising from control $\nu_{s,\cdot} \in \mathcal V_2^W$ and $\xi\in \mathbb R^d$ given by~\eqref{eq process} and~\eqref{eq adjoint proc}.
Then
\begin{equation}
\label{eq der of J wrt flow}
\begin{split}
& \tfrac{d}{ds} J^\sigma(\nu_{s,\cdot})  = \\
& - \mathbb E^W \!\!\!\int_0^T \!\!\!\int \left[  \left(\nabla_a\tfrac{\delta \mathbf H^0}{\delta m}\right)(\nu_{s,t},\cdot) + \tfrac{\sigma^2}{2}\nabla_a U_t + \tfrac{\sigma^2}{2}\nabla_a \log \nu_{s,t}  \right] \cdot \bigg(b_{s,t}+\tfrac{\sigma^2}{2} \nabla_a \log \nu_{s,t}\bigg)\, \nu_{s,t}\,(da)   \,dt\,.			
\end{split}
\end{equation}
\end{theorem}

The proof of Theorem~\ref{th with measure flow} will come in Section~\ref{sec proof of opt cond}. 
But we first note that if one would like to choose a flow of measures such that $s\mapsto J(\nu_{s,\cdot})$ is (strictly) decreasing then the right hand side of~\eqref{eq der of J wrt flow} should be (strictly) negative.
Utilizing Theorem~\ref{th with measure flow}, we see that this can be achieved by taking $b_{s,t} = (\nabla_a \tfrac{\delta \mathbf H_t^{0}}{\delta m})(\nu_s,\cdot) + \tfrac{\sigma^2}{2}(\nabla_a U_t)$.
Note that if $a\mapsto \nabla_a U_t(a)$ is Lipschitz continuous $\mathbb P^W \otimes \Lambda([0,T])$-a.e and if Assumption~\ref{as coefficients new} holds then this satisfies the conditions of Definition~\ref{def vect field flow def}.
With this choice~\eqref{eq der of J wrt flow} becomes
\[
\tfrac{d}{ds} J(\nu_{s,\cdot}) = - \mathbb E^W \int_0^T  \bigg[\int  \big|(\nabla_a \tfrac{\delta \mathbf H_t^{\sigma}}{\delta m})(\nu_{s,t},\cdot) \big|^2 \,\nu_{s,t}(da)   \bigg]\,dt \leq 0		
\]
with equality only if $\nu_{s,\cdot}$ satisfies~\eqref{eq foc intro}.
From Theorem~\ref{th with measure flow} one expects convergence of the gradient flow to a control satisfying Pontryagin optimality (typically a local minimiser of $J$ but not necessarily).

In case the  system~\eqref{eq mfsgd}-\eqref{eq mfsgd 2} has unique solution (cf. Lemma~\ref{thm:WellposednessMeanField}) we
let $P_{s,t} \mu^0 := \mathcal L\big( \theta_{s,t} | \mathcal F_t^W \big)$,
where $(\theta_{s,t})_{s\geq 0, t\in [0,T]}$ is the unique solution to the  system~\eqref{eq mfsgd}-\eqref{eq mfsgd 2} started with the initial condition $(\theta^0_t)_{t\in [0,T]}$ such that $\mathcal L(\theta^0_t(\omega^W)) = \mathbb P^B \circ \theta^0_t(\omega^W)^{-1} = \mu^0(\omega^W)$ for a.e. $\omega^W \in \Omega^W$. 
Moreover we define a semigroup $P_s \mu^0 := (P_{s,t} \mu^0)_{t\in [0,T]}$ and note that $P_s \mu^0 \in \mathcal V_{2}^W$ for any $s\geq 0$ and
 we have $P_{s+s'} \mu^0 = P_s\big(P_{s'}\mu^0\big)$ for any $s'\geq 0$.

\begin{theorem}
\label{thm unique minimiser}
Let Assumptions \ref{as coefficients new} hold. 
Let
\begin{equation}
\label{eq foc}
\begin{split}
\mathcal I^\sigma := \bigg\{ \nu \in \mathcal V_q^W : \,\, & a\mapsto \tfrac{\delta \mathbf H_t^{\sigma}}{\delta m}(\nu,a)\,\, \text{is constant for a.e. $a\in \mathbb R^p$, a.e. $(t,\omega^W) \in (0,T)\times \Omega^W$}\bigg\}\,.	
\end{split}
\end{equation}
Then,
\begin{enumerate}[i)]
\item any solution of~\eqref{eq mfsgd} which satisfies $P_s \mu^\ast = \mu^\ast$ (i.e. an invariant measure) lies in $\mathcal I^\sigma$ and moreover
\item 
if there is a unique invariant measure $\mu^\ast \in \mathcal V^W_{2}$ and if for any $\mu^0 \in \mathcal V^W_2$ the system~\eqref{eq mfsgd}-\eqref{eq mfsgd 2} has solution given by $P_s \mu^0$ such that $\lim_{s\to\infty} \rho_q(P_s \mu^0, \mu^\ast) = 0$ then $\mathcal I^\sigma   = \{\mu^\ast\}$, i.e. $\mu^\ast$ is the only control which satisfies the first order condition~\eqref{eq foc}, and for any $\mu^0 \in \mathcal V^W_2$ we have $J^\sigma(\mu^\ast) \leq J^\sigma(\mu^0)$. 
\end{enumerate}
\end{theorem}
Theorem~\ref{thm unique minimiser} will be proved at the end of Section~\ref{sec proof of opt cond}. 
To prove Theorem~\ref{thm unique minimiser} we will need the following necessary condition for optimality, known as the Pontryagin optimality principle. 

\begin{theorem}[Necessary condition for optimality]
\label{thm necessary cond linear}
Fix $\sigma > 0$. 
Fix $q>2$.
Let the Assumptions~\ref{as coefficients new} hold. 
If $\nu\in \mathcal V^W_{2}$ is (locally) optimal for $J^\sigma$ given by~\eqref{eq objective bar J}, $X(\nu)$ and $Y(\nu)$, $Z(\nu)$ are
the associated optimally controlled state and adjoint processes given by~\eqref{eq process} and~\eqref{eq adjoint proc} respectively,
then for any other $\mu \in \mathcal V^W_{2}$ it holds that
\[
\int  \left[\tfrac{\delta H^0}{\delta m}(X_t, Y_t, Z_t,\nu_t,a) + \tfrac{\sigma^2}2(\log \nu_t(a)  - \log \gamma_t(a))\right] \,(\mu_t-\nu_t)(da)  
\geq 0 \,\,\, \text{for a.e. $(\omega,t)\in \Omega^W \times (0,T)$}\,.
\]
\end{theorem}

The proof of Theorem~\ref{thm necessary cond linear} is given in Section~\ref{sec proof of opt cond}. 

\subsection{Existence and uniqueness for~\eqref{eq mfsgd}-\eqref{eq mfsgd 2} and its convergence to invariant measure}

\begin{assumption}
\label{ass dissipativity of U}
Let $\nabla_aU$ be Lipschitz continuous in $a$ with the constant uniform in $(t,\omega)$, let $\nabla_a U_t(0) = 0$ for a.e. $(t,\omega)$ and moreover let there be $\kappa_u>0$ such that for a.e. $(t,\omega)$:
\[
\left(\nabla_aU_t(a')-\nabla_aU_t(a)\right)\cdot\left(a'-a\right)\geq \kappa_u|a'-a|^2,\,a,a'\in \mathbb R^p\,.
\]
Additionally assume that there exists $\alpha$ which is an $\mathbb R^p$-valued $(\mathcal F_t^W)_{t\in[0,T]}$-progressively measurable process such that $\mathbb E\int_0^T |\alpha_t|^2\,dt \leq K$ such that for a.e. $(t,\omega)$ we have
\[
\nabla_a U_t(a)\cdot a \geq \kappa_u |a|^2 - \alpha_t\,.
\]
In applications, Assumption~\ref{ass dissipativity of U} has natural interpretation. 
Imagine that one has already solved a related control problem, potentially one where a closed form solution exists, with strict open loop controls where we denote the optimal solution by $\alpha$. 
It is then natural to take $\gamma_t(a) \sim e^{-U_t} \sim e^{-\frac{\kappa_u}2|a-\alpha_t|^2}$.

\end{assumption}

We will assume that the running reward can be decomposed into a Lipschitz part and a convex part: $F=F^L + F^C$.
As an example consider
$
F^C_t(x,m) = x + \int \frac12|a|^2\,m(da)\,.
$
Then $\tfrac{\delta F^C}{\delta m}(x,m,a) = x + \frac{1}{2}|a|^2$, so that $\nabla_a \tfrac{\delta F^C}{\delta m}(x,m,a) = a$. 
This will satisfy Assumption~\ref{ass dissipativity of F} below with $\kappa_f = 1$. 

\begin{assumption}
\label{ass dissipativity of F}
Let $\nabla_a \tfrac{\delta F^C}{\delta m}$ exist on $[0,T]\times \mathbb R^d \times \mathcal P(\mathbb R^p) \times \mathbb R^p$  and there be $\kappa_f>0$ such that for any $x \in \mathbb R^d$, $m\in \mathcal P_2(\mathbb R^p)$ we have
\[
2\left(\nabla_a \tfrac{\delta F^C}{\delta m}(x,m,a') - \nabla_a \tfrac{\delta F^C}{\delta m}(x,m,a)\right)\cdot\left(a'-a\right)\geq \kappa_f|a'-a|^2,\,a,a'\in \mathbb R^p\,.
\]
\end{assumption}

\begin{assumption}
\label{ass for bsde estimates}
There is $K>0$ such that for all $t\in[0,T$], for all $x,x' \in \mathbb R^d$, for all $m,m' \in \mathcal P_2(\mathbb R^p)$ and for all $a\in \mathbb R^p$ we have:
\begin{enumerate}[i)]
\item 
$
|\Phi_t(x,m) - \Phi_t(x',m')| + |\Gamma_t(x,m) - \Gamma_t(x',m')| \leq K \Big( |x-x'| + \mathcal W_2(m,m')\Big)\,.
$
\item 
$
|\nabla_x \Phi_t(x,m)| + \sum_{i=1}^d\sum_{j=1}^{d'}|\nabla_x \Gamma_t^{ij}(x,m)| + |\nabla_x g(x)| + |\nabla_x F_t(x,m)|\leq K\,.
$
\item 
$
|\nabla_x \Phi_t(x,m) - \nabla_x \Phi_t(x',m')| \leq K |x-x'| + \mathcal W_2(m,m')$ and $|\nabla_x \Gamma_t(x,m) - \nabla_x \Gamma_t(x',m')| \leq K |x-x'|\,.
$
\item 
$
|\nabla_x F_t(x,m) - \nabla_x F_t(x',m')| \leq K |x-x'| + \mathcal W_2(m,m')
\,\,\,\text{and}\,\,\,
|\nabla_x g(x) - \nabla_x g(x')| \leq K|x-x'|\,.
$
\end{enumerate}	
	
\end{assumption}

\begin{assumption}
\label{ass for gradient system}
There is $K>0$ such that for all $t\in[0,T$], for all $x,x' \in \mathbb R^d$, for all $m,m' \in \mathcal P_2(\mathbb R^p)$ and for all $a,a'\in \mathbb R^p$ we have:
\begin{enumerate}[i)]
\item 
$
\left|(\nabla_a \tfrac{\delta \Phi_t}{\delta m})(x,m,a)\right| + \left|(\nabla_a \tfrac{\delta \Gamma_t}{\delta m})(x,m,a)\right| \leq K
$ and $\nabla_a^2 \frac{\delta \Gamma}{\delta m} = 0$.
\item With $G$ standing in for either $\Phi$ or $F^L$: 
\[
\left|(\nabla_a \tfrac{\delta G_t}{\delta m})(x,m,a) - (\nabla_a \tfrac{\delta G_t}{\delta m})(x',m',a')\right| \leq K(|x-x'|+\mathcal W_2(m,m') + |a-a'|)\,.
\]
\end{enumerate}	
\end{assumption}

\begin{lemma}[Existence and uniqueness]\label{thm:WellposednessMeanField}
\label{lemma existence and uniqueness}
Let Assumptions~\ref{ass dissipativity of U}, \ref{ass dissipativity of F}, \ref{ass for bsde estimates} and~\ref{ass for gradient system} hold.
Let $\hat c>0 $ be the constant arising in Lemma~\ref{lem sys diss} and let $L>0$ be the constant arising in Lemma~\ref{lem sys mono}.
If $\sigma^2\kappa_U + \kappa_f - L \geq 0$
then there is a unique solution to~\eqref{eq mfsgd}-\eqref{eq mfsgd 2} for all $s\geq 0$.
Moreover if $\lambda := \sigma^2\kappa_U + \kappa_f - \hat c > 0$ then for any $s\geq 0$ we have  
\begin{equation}\label{eq:UnifBound}
\begin{aligned}
&\int_{0}^{T}\mathbb E[|\theta_{s,t}|^2]\,dt
	\leq e^{-\lambda s} \int_0^T \mathbb E |\theta_{0,t}|^2 \,dt + \frac1\lambda(p\sigma^2 T + \hat c(1+|\xi|^2)\,.
\end{aligned}
\end{equation}
\end{lemma}

\begin{theorem} \label{thm conv to inv meas rate}
Let Assumptions~\ref{ass dissipativity of U}, \ref{ass dissipativity of F}, \ref{ass for bsde estimates} and~\ref{ass for gradient system} hold.
Moreover, assume that  $\lambda := \sigma^2\kappa_U + \kappa_f - 2L > 0$.
Then there is $\mu^\ast \in \mathcal V_q^W$ such that for any $s\geq 0$ we have $P_s \mu^\ast = \mu^\ast$ and $\mu^\ast$ is unique.
For any 
$\mu^0 \in \mathcal V_2^W$ we have that 
\begin{equation}
\label{eq exp conv to inv meas}	
\rho_2(P_s \mu^0, \mu^\ast) \leq e^{-\frac12 \lambda s}\rho_2(\mu^0,\mu^\ast)\,.
\end{equation}

\end{theorem}
Theorem~\ref{thm conv to inv meas rate} will be proved in Section~\ref{sec ex uniq and conve}.
Let us now present an example.
\begin{example}

Consider the controlled SDE 
\[
dX_t(\nu) = b(X_t(\nu))\,dt + \bigg(\int \phi(a) \, \nu_t(da)\bigg)\,dt + \sigma(X_t(\nu))\,dW_t + \bigg(\int a \,\nu_t(da)\bigg)\,dW_t\,,
\]
where $b$ and $\sigma$ are differentiable, Lipschitz continuous functions with bounded and Lipschitz continuous derivatives. Moreover $\phi$ is differentiable with bounded derivatives. 

To define the objective let $\zeta(x) := \frac12|x|^2$ for $|x| \leq 1$ and $\zeta(x) = \frac12|x|$ for $|x|>1$. 
This function is differentiable with bounded derivative. 
Our objective functional is
\begin{equation*}
\begin{split}
J^\sigma(\nu) & := \mathbb E^{W}\left[\int_0^T \bigg[ \zeta(X_t(\nu)) +  \frac{\kappa_f}{2}\int |a|^2\,\nu_t(da) + \tfrac{\sigma^2}{2}R(\nu_t|\gamma_t)\bigg]\,dt + \zeta(X_T(\nu)) \Big| X_0(\nu) = \xi\right]\,.
\end{split}
\end{equation*}

We can see that in this setting Assumptions~\ref{as coefficients new}, \ref{ass dissipativity of F}, \ref{ass for bsde estimates} and~\ref{ass for gradient system} and all hold. We are free to choose any prior which will satisfy Assumption~\ref{ass dissipativity of U}.  
A possible prior would be of the form $e^{-\tfrac12|a-\alpha_t|^2}$ with $\alpha$ the solution of an associated linear-quadratic control problem.
\end{example}

\section{Pontryagin Optimality for Entropy-Regularized  Stochastic Control}
\label{sec proof of opt cond}
We start by giving the proof of Lemma~\ref{lem pde} as this will be independent of all the results concerning Pontryagin's optimality principle. 
\begin{proof}[Proof of Lemma~\ref{lem pde}]
The linear PDE~\eqref{eq fpe} has unique solution $\nu_{\,\cdot,t} \in C^{1,\infty}((0,\infty)\times \mathbb R^p; \mathbb R)$ for each $t\in [0,T]$ and $\omega^W \in \Omega^W$ due to e.g Ladyzenskaja, Solonnikov and Ural'ceva~\cite[Chapter IV]{LSU68}.
The $\mathcal F_t^W$-measurability of $\nu_{s,t}(a)$ is a question of measurability of an explicitly defined function and this is proved e.g. in~\cite[Lemma~3.2]{emmrich2017nonlinear}. 
Consider, for each $t\in [0,T]$, the stochastic process $(\theta_{s,t})_{s\geq 0}$, solving
\[
d\theta_{s,t} = - b_{s,t}(\theta_{s,t})\,ds + \sigma \,dB_s\,.
\]	
Let $\mu_{s,t}$ denote the law of $\theta_{s,t}$ given $\omega \in \Omega^W$
From Girsanov's theorem we see that the $\mu_{s,t}$ has, for each $s>0$ and $t\in [0,T]$, smooth density and moreover $\mu_{s,t}(a) > 0$.
Applying It\^o's formula to $\varphi \in C^2_b(\mathbb R^p)$, taking expectation (over $\Omega^B$) and using $\mu_{s,t}$ to denote the law of $\theta_{s,t}$ given $\omega \in \Omega^W$ we can check that $\mu_{s,t}$ satisfies 
\[
\int \varphi(a) \mu_{s,t}(da) = \int \varphi(a)\mu_{0,t}(da) + \int_0^s \Big[ \int - b_{s,t}(a)\nabla_a \varphi(a)\mu_{r,t}(da) + \int \frac12\sigma^2 \Delta_a \varphi(a) \mu_{r,t}(da)  \Big]\,ds\,.
\]
Integrating by parts we see that is a solution to~\eqref{eq fpe}. 
As the solutions are unique we conclude that $\mu_{s,\cdot} = \nu_{s,\cdot}$ and so $\nu_{s,t}(a) > 0$ for all $s>0$, $t\in [0,T]$ and $\Omega^W$-a.s.
\end{proof}

We know that relative entropy is only lower semi-continous on $\mathcal P_2(\mathbb R^p)$ and thus we wouldn't expect even the directional derivative to exists (in the sense that the limit of the difference quotient is a finite number) everywhere on $\mathcal P_2(\mathbb R^p)$.  
The following lemma gives two useful estimates on the difference quotient.  
\begin{lemma}[Difference quotient estimates for relative entropy]
\label{lemma diff of Ent flat}
Let $\nu_t,\mu_t \in \mathcal V_2^W$ and let $\nu^\varepsilon = \nu + \varepsilon(\mu-\nu)$. 
Then a.s. 
\begin{enumerate}[i)]
\item for any $\varepsilon \in (0,1)$ we have
\[
\frac1\varepsilon \int_0^T \left[R(\nu^\varepsilon_t|\gamma_t) - R(\nu_t|\gamma_t)\right]\,dt \geq \int_0^T \int [\log \nu_t(a)  - \log \gamma_t(a) ](\mu_t - \nu_t)(da)\,dt\,,
\]
\item 
\[
\limsup_{\varepsilon \to 0} \frac1\varepsilon \int_0^T \left[R(\nu^\varepsilon_t|\gamma_t) - R(\nu_t|\gamma_t)\right]\,dt 
\leq \int_0^T \int [\log \nu_t(a)  - \log \gamma_t(a) ](\mu_t - \nu_t)(da)\,dt\,.
\]
\end{enumerate}
\end{lemma}
\begin{proof}
This proof is similar to the one given in~\cite[Proposition 2.4]{hu2019Mean} but extended to the setting of this paper.	
For i) we begin by observing that
\[
\begin{split}
& \frac1\varepsilon \left(R(\nu^\varepsilon_t|\gamma_t) - R(\nu_t|\gamma_t)\right) = \frac1\varepsilon\int \bigg[\log \frac{\nu^\varepsilon_t(a)}{\gamma_t(a)} \nu^\varepsilon_t(a)  - \log \frac{\nu_t(a)}{\gamma_t(a)}  \nu_t(a) \bigg]\,da\\
& =  \frac1\varepsilon \int\big(\nu^\varepsilon_t(a) - \nu_t(a)\big) \log \frac{\nu_t(a)}{\gamma_t(a)} \,da + \frac1\varepsilon  \int\nu^\varepsilon_t(a) \bigg[\log \frac{\nu^\varepsilon_t(a)}{\gamma_t(a)} - \log \frac{\nu_t(a)}{\gamma_t(a)} \bigg]\,da\\
& =  \int\big(\mu_t(a) - \nu_t(a)\big) \log \frac{\nu_t(a)}{\gamma_t(a)} \,da + \frac1\varepsilon  \int\nu^\varepsilon_t(a) \log \frac{\nu^\varepsilon_t(a)}{\nu_t(a)}  \,da\\
& =  \int [\log \nu_t(a) - \log\gamma_t(a)](\mu_t - \nu_t)(da) + \frac1\varepsilon  \int\frac{\nu^\varepsilon_t(a)}{\nu_t(a)} \log \frac{\nu^\varepsilon_t(a)}{\nu_t(a)} \nu_t(a) \,da\,.
\end{split}
\]
Since $x\log x \geq x - 1$ for $x\in (0,\infty)$ we get
\[
\begin{split}
& \frac1\varepsilon  \int\frac{\nu^\varepsilon_t(a)}{\nu_t(a)} \log \frac{\nu^\varepsilon_t(a)}{\nu_t(a)} \nu_t(a) \,da 
\geq \frac1\varepsilon  \int \bigg[\frac{\nu^\varepsilon_t(a)}{\nu_t(a)}  -1 \bigg] \nu_t(a) \,da = \frac1\varepsilon  \int \big[\nu^\varepsilon_t(a)  - \nu_t(a)\big] \,da = 0\,.
\end{split}
\]
Hence
\[
\frac1\varepsilon \left(R(\nu^\varepsilon_t|\gamma_t) - R(\nu_t|\gamma_t)\right) 
\geq \int [\log \nu_t(a) - \log\gamma_t(a)](\mu_t - \nu_t)(da)\,,
\]
which proves i).

To prove ii) start by noting that one may write
\[
\begin{split}
& \frac1\varepsilon \left(R(\nu^\varepsilon_t|\gamma_t) - R(\nu_t|\gamma_t)\right) = \frac1\varepsilon\int \big[\big(\log \nu^\varepsilon_t(a) - \log \gamma_t(a)\big) \nu^\varepsilon_t(a)  - \big(\log \nu_t(a) - \log \gamma_t(a) \big) \nu_t(a) \big]\,da\\
& = \int \frac1\varepsilon\big[\nu^\varepsilon_t(a) \log \nu^\varepsilon_t(a) - \nu_t(a)\log \nu_t(a) - \log \gamma_t(a)\big( \nu^\varepsilon_t(a) - \nu_t(a)\big)    \big]\,da\\
\end{split}
\]
Now 
\[
- \frac1\varepsilon \log \gamma_t(a) (\nu^\varepsilon_t(a) - \nu_t(a)) = -\log \gamma_t(a) (\mu_t(a) - \nu_t(a))\,.
\]
Moreover, since the map $x\mapsto x\log x$ is convex for $x > 0$ and by definition of $\nu^\varepsilon$, we have
\[
\frac1\varepsilon\left[
\nu^\varepsilon_t(a)\log(\nu^\varepsilon_t(a)) - \nu_t(a)\log(\nu_t(a))
\right] \leq \mu(a)\log \mu(a) - \nu(a)\log \nu(a)\,.
\]
Hence 
\[
\frac1\varepsilon \left(R(\nu^\varepsilon_t|\gamma_t) - R(\nu_t|\gamma_t)\right) \leq R(\mu_t|\gamma_t) - R(\nu_t|\gamma_t)\,.
\]
Since $\mu, \nu \in \mathcal V_2^W$ the right hand side is finite. 
Finally, by the reverse Fatou's lemma,
\[
\begin{split}
& \limsup_{\varepsilon \to 0} \frac1\varepsilon \left[R(\nu^\varepsilon_t|\gamma_t) - R(\nu_t|\gamma_t)\right] \\
& \leq \int \limsup_{\varepsilon \to 0} \frac1\varepsilon\big[\nu^\varepsilon_t(a) \log \nu^\varepsilon_t(a) - \nu_t(a)\log \nu_t(a) - \log \gamma_t(a)\big( \nu^\varepsilon_t(a) - \nu_t(a)\big)    \big]\,da\,.\\
\end{split}
\]
Calculating the derivative of $x\mapsto x \log x$ for $x>0$ leads to 
\[
\begin{split}
& \limsup_{\varepsilon \to 0} \frac1\varepsilon\big[\nu^\varepsilon_t(a) \log \nu^\varepsilon_t(a) - \nu_t(a)\log \nu_t(a) - \log \gamma_t(a)\big( \nu^\varepsilon_t(a) - \nu_t(a)\big)    \big] \\
& = (1 + \log \nu_t(a))(\mu_t(a)-\nu_t(a)) - \log \gamma_t(a) (\mu_t(a) -\nu_t(a))\,.
\end{split}
\]
Hence
\[
\limsup_{\varepsilon \to 0} \frac1\varepsilon \left[R(\nu^\varepsilon_t|\gamma_t) - R(\nu_t|\gamma_t)\right] 
\leq \int [\log \nu_t(a)  - \log \gamma_t(a) ](\mu_t - \nu_t)(da)\,.
\]
This completes the proof.
\end{proof}

The next Lemma is proved in Lemmas 6.1, 6.2 and 6.3 in~\cite{hu2019meanode}.
\begin{lemma}[Properties of the Gradient Flow]
\label{lemma prop grad flow}
Let $b$ be a permissible flow such that $a\mapsto |\nabla_a b_{s,t}(a)|$ is bounded uniformly in $s>0$, $t\in [0,T]$, $\omega^W \in \Omega^W$. Then
\begin{enumerate}[i)]
\item For all $s>0$, $t\in [0,T]$, $\omega^W \in \Omega^W$ and $a\in \mathbb R^p$ we have $\nu_{s,t}(a) > 0$	and $R(\nu_{s,t}|\gamma_t) < \infty$.
\item For all $s>0$, $t\in [0,T]$ and $\omega^W \in \Omega^W$ we have $\int |\nabla_a \log \nu_{s,t}(a)|^2 \nu_{s,t}(a)(da) < \infty$.
\item For all $s>0$, $t\in [0,T]$ and $\omega^W \in \Omega^W$ we have
\[
\int |\nabla_a \nu_{s,t}(a)|\,da + \int |a\cdot \nabla_a \nu_{s,t}(a)|\,da + \int |\Delta_a \nu_{s,t}(a)|\,da < \infty\,.
\]
\end{enumerate}
\end{lemma}

The following lemma proves that along the gradient flow~\eqref{eq fpe} the function $s\mapsto R(\nu_{s,t}|\gamma_t)$ is in fact differentiable and we have an explicit expression for the derivative. 
\begin{lemma}[Derivative of entropy along a gradient flow]
\label{lemma derivative of entropy along flow}
Fix $\sigma \geq 0$ and let Assumption~\ref{ass dissipativity of U}
hold.
Let $b$ be a permissible flow (c.~f.~Definition~\ref{def vect field flow def})
such that
$a\mapsto |\nabla_a b_{s,t}(a)|$ is bounded uniformly in $s,t$ and $\omega^W \in \Omega^W$.
Let $\nu_{s,t}$ be the solution to~\eqref{eq fpe}.
Then
\[
\begin{split}
d R(\nu_{s,t}|\gamma_t) 
& =  -\int\bigg( \nabla_a \log \nu_{s,t} + \nabla_a U \bigg)\cdot \left(  b_{s,t}+\tfrac{\sigma^2}{2} \nabla_a \log \nu_{s,t} \right)\,\nu_{s,t}(da)\,ds\,.
\end{split}
\]
	
\end{lemma}
\begin{proof}
Let 
\[
d\theta_{s,t} = - b_{s,t}(\theta_{s,t})\,ds + \sigma \, dB_s\,, 
\]
where $B$ is an $\mathbb R^p$-Wiener process on some $(\Omega^B, \mathcal F^B, \mathbb P^B)$, where the filtration $\mathbb F^B = (\mathcal F_t^B)$ where
$\mathcal F_s^B := \sigma(B_u:0\leq u \leq s)$.
Let $\Omega := \Omega^W \times \Omega^B$, $\mathcal F := \mathcal F^W \otimes \mathcal F^B$
and $\mathbb P := \mathbb P^W \otimes \mathbb P^B$.
Let $\theta_{0,t}$ be s.t. $\mathcal L(\theta_{0,t}|\mathcal F_t^W) = \nu_{0,t}$. 
It is easy to check, as in the proof of Lemma~\ref{lem pde}, that if $b$ a permissible flow (cf. Definition~\ref{def vect field flow def}) then $\nu_{s,t} = \mathcal L(\theta_{s,t}|\mathcal F_t^W)$.
From Lemma~\ref{lemma prop grad flow} we know that $\nu_{s,t} > 0$ and so from It\^o's formula we get that
\[
\begin{split}
 d\Big( \log(\nu_{s,t}(\theta_{s,t})) + U_t((\theta_{s,t})\Big)	
= & \bigg( \frac{\partial_s \nu_{s,t}}{\nu_{s,t}}(\theta_{s,t}) - \frac{\nabla_a \nu_{s,t}}{\nu_{s,t}}(\theta_{s,t})\cdot b_{s,t}(\theta_{s,t}) - \tfrac{\sigma^2}{2} \bigg|\frac{\nabla_a\nu_{s,t}}{\nu_{s,t}}\bigg|^2 + \tfrac{\sigma^2}{2}\tfrac{\delta_a \nu_{s,t}}{\nu_{s,t}}(\theta_{s,t})\\
& - b_{s,t}(\theta_{s,t})\cdot \nabla_a U_t(\theta_{s,t}) + \tfrac{\sigma^2}{2}\Delta_a U_t(\theta_{s,t}) \bigg)\,ds + dM^B_s\,,
\end{split}
\]
where $M^B_s = \sigma \nabla_a \log(\nu_{s,t}(\theta_{s,t})) dB_s$ is a $\mathcal F^B$-martingale starting from $0$ due to Lemma~\ref{lemma prop grad flow}.
From~\eqref{eq fpe} we get that $\partial_s \nu_{s,t} - b_{s,t} \nabla_a \nu_{s,t} = \nu \nabla_a \cdot b_{s,t} + \tfrac{\sigma^2}{2}\Delta_a \nu_{s,t}$ and so 
\[
\begin{split}
d\Big( \log(\nu_{s,t}(\theta_{s,t})) + U_t((\theta_{s,t})\Big) 
= & \bigg( \nabla_a \cdot b_{s,t}(\theta_{s,t}) - \tfrac{\sigma^2}{2} \bigg|\frac{\nabla_a\nu_{s,t}}{\nu_{s,t}}\bigg|^2 + \sigma^2\tfrac{\delta_a \nu_{s,t}}{\nu_{s,t}}(\theta_{s,t}) \\
&- b_{s,t}(\theta_{s,t})\cdot \nabla_a U_t(\theta_{s,t}) + \tfrac{\sigma^2}{2}\Delta_a U_t(\theta_{s,t}) \bigg)\,ds + dM^B_s\,.
\end{split}
\]
Now observe that $\sigma^2 \tfrac{\delta_a \nu_{s,t}}{\nu_{s,t}} = \sigma^2 \nabla_a\cdot \big (\frac{\nabla_a \nu_{s,t}}{\nu_{s,t}}\big) + \sigma^2 \big| \frac{\nabla_a \nu_{s,t}}{\nu_{s,t}}\big|^2$.
Hence 
\[
\begin{split}
d\Big( \log(\nu_{s,t}(\theta_{s,t})) + U_t((\theta_{s,t})\Big) 
= & \bigg( \nabla_a \cdot b_{s,t}(\theta_{s,t}) + \tfrac{\sigma^2}{2} \bigg|\frac{\nabla_a\nu_{s,t}}{\nu_{s,t}}\bigg|^2 + \sigma^2 \nabla_a \cdot \bigg(\frac{\nabla_a \nu_{s,t}}{\nu_{s,t}}(\theta_{s,t})\bigg) \\
&- b_{s,t}(\theta_{s,t})\cdot \nabla_a U_t(\theta_{s,t}) + \tfrac{\sigma^2}{2}\Delta_a U_t(\theta_{s,t}) \bigg)\,ds + dM^B_s\,.
\end{split}
\]
Taking expectation w.r.t. $\mathbb P^B$ we get 
\[
\begin{split}
d R(\nu_{s,t}|\gamma_t) & = d\mathbb E^B\Big( \log(\nu_{s,t}(\theta_{s,t})) + U_t((\theta_{s,t})\Big) \\
& =  \int\bigg( \nabla_a \cdot b_{s,t} + \tfrac{\sigma^2}{2} \bigg|\frac{\nabla_a\nu_{s,t}}{\nu_{s,t}}\bigg|^2 + \sigma^2 \nabla_a \cdot \bigg(\frac{\nabla_a \nu_{s,t}}{\nu_{s,t}}\bigg) - b_{s,t}\cdot \nabla_a U + \tfrac{\sigma^2}{2}\Delta_a U \bigg)\nu_{s,t}(da)\,ds\,.
\end{split}
\]
Integrating by parts (see Lemma~\ref{lemma prop grad flow}) and noting  that $\int \nabla_a \cdot \big(\frac{\nabla_a \nu_{s,t}}{\nu_{s,t}}\big)\nu_{s,t}\,da = -\int \big|\frac{\nabla_a \nu_{s,t}}{\nu_{s,t}}\big|^2\nu_{s,t}\,da$ this becomes 
\[
\begin{split}
d R(\nu_{s,t}|\gamma_t) 
& =  \int\bigg( - b_{s,t}\cdot \frac{\nabla_a \nu_{s,t}}{\nu_{s,t}} - \tfrac{\sigma^2}{2} \bigg|\frac{\nabla_a\nu_{s,t}}{\nu_{s,t}}\bigg|^2  - b_{s,t}\cdot \nabla_a U - \tfrac{\sigma^2}{2}\nabla_a U \frac{\nabla_a \nu_{s,t}}{\nu_{s,t}}\bigg)\nu_{s,t}(da)\,ds\,.
\end{split}
\]
This completes the proof.
\end{proof}

Next, start working towards the proof of Theorem~\ref{thm necessary cond linear}, which is the necessary part of Pontryagin's optimality principle in the setting of this paper.
To that end we need an expression for the directional derivative of $J^0$ and an estimate for the directional derivative of $J^\sigma$.
We will write $(X_t)_{t\in[0,T]}$ for the solution of~\eqref{eq process} driven by $\nu \in \mathcal V_{2}^W$.
We will work with an additional control $\mu\in \mathcal V^W_{2}$ and define $\nu^\varepsilon_t := \nu_t + \varepsilon(\mu_t - \nu_t)$
 and $(X^{\varepsilon}_t)_{t\in[0,T]}$ for the solution of~\eqref{eq process} driven by $\nu^\varepsilon$.
First, however we need a directional derivative for the forward process~\eqref{eq process}.
Let $V_0=0$ be fixed and consider
\begin{equation}
\label{eq V proc linear}
\begin{split}
dV_t  = & \left[(\nabla_x \Phi_t)(X_t, \nu_t)V_t + \int \tfrac{\delta \Phi_t}{\delta m}(X_t,\nu_t, a)(\mu_t-\nu_t)(da)\right]	\,dt \\
& + \left[ (\nabla_x \Gamma_t)(X_t, \nu_t)V_t + \int \tfrac{\delta \Gamma_t}{\delta m}(X_t,\nu_t, a)(\mu_t-\nu_t)(da) \right]\,dW_t\,.	
\end{split}
\end{equation}
We observe that this is a linear equation and so, under Assumption~\ref{as coefficients new} its solution is unique and has all the moments i.e. for any $p'\geq 1$ we have $\mathbb E^W \sup_{t'\leq t} |V_{t'}|^{p'} < \infty$. 
In the following lemma we will prove that the process $(V_t)_{t\in [0,T]}$ given by~\eqref{eq V proc linear} is the ``variation process'' for~\eqref{eq process} in that it is an $L^2$-directional derivative of $(X_t)_{t\in[0,T]}$ in the direction $\mu-\nu$.  
\begin{lemma}[Variation process is a directional derivative of the forward process]
\label{lemma V process linear}
Assume that $\nabla_x \Phi$ and $\nabla_x \Gamma$ exist and that there is $K>0$ such that for all $(t,m)\in [0,T] \times \mathcal P(\mathbb R^p)$ and for all $x,x' \in \mathbb R^d$ we have 
\[
\begin{split}
& |\nabla_x \Phi_t(x,m)| + |\nabla_x \Gamma_t(x,m)| \leq K\,,\\
& |\nabla_x \Phi_t(x,m)-\nabla_x \Phi_t(x',m)| + |\nabla_x \Gamma_t(x,m) - \nabla_x \Gamma_t(x',m)| \leq K|x-x'|\,.\\	
\end{split}
\]
Further assume that $\tfrac{\delta \Phi_t}{\delta m}$, $\tfrac{\delta^2 \Phi_t}{\delta m^2}$, $\tfrac{\delta \Gamma_t }{\delta m}$ and $\tfrac{\delta^2 \Gamma_t }{\delta m^2}$ exists and moreover there is $K>0$ such that for all $(t,x,m,a,a')\in [0,T] \times \mathbb R^d \times \mathcal P(\mathbb R^p) \times \mathbb R^p \times \mathbb R^p$ we have $\left|\tfrac{\delta^2 \Phi_t}{\delta m^2}(x,m,a,a')\right| + \left|\tfrac{\delta^2 \Gamma_t}{\delta m^2}(x,m,a,a')\right| \leq K$.
Then
\[
\lim_{\varepsilon\searrow 0} \mathbb E^W\left[\sup_{t\leq T} \left|\tfrac{X^{\varepsilon}_t - X_t}{\varepsilon} - V_t\right|^2 \right]= 0\,.
\]	
\end{lemma}
\begin{proof}
Let
\[
V^{\varepsilon}_t := \tfrac{X^{\varepsilon}_t - X_t}{\varepsilon} - V_t
\,\,\, \text{i.e.}\,\,\,
X^{\varepsilon}_t = X_t + \varepsilon(V^{\varepsilon}_t + V_t)\,.
\]
We wish to show that $\mathbb E \sup_{s\leq t}|V^\varepsilon_s|^2  \,   \leq c_T\varepsilon^2 \to 0$ as $\varepsilon \to \infty$. 
To that end we start calculating the terms appearing in the difference quotient. 
First we note that
\[
\begin{split}
& \Phi(X^\varepsilon_t, \nu^\varepsilon_t) - \Phi(X_t,\nu_t) 
= \Phi(X^\varepsilon_t, \nu^\varepsilon_t) - \Phi(X^\varepsilon_t, \nu_t) + \Phi(X^\varepsilon_t,\nu_t) - \Phi(X_t,\nu_t) \\
& = \varepsilon \int_0^1 (\nabla_x \Phi)(X_t + \lambda \varepsilon ( V^\varepsilon_t + V_t), \nu_t) (V^\varepsilon_t + V_t)\,d\lambda 
+ \varepsilon\int_0^1 \int \tfrac{\delta \Phi}{\delta m}(X^\varepsilon_t, (1-\lambda)\nu^\varepsilon_t + \lambda \nu_t,a)(\mu_t-\nu_t)(da)\,d\lambda\,. 	
\end{split}
\]
Hence, rearranging, for the drift we get,
\[
\begin{split}
& \frac1\varepsilon \left[\Phi(X^\varepsilon_t, \nu^\varepsilon_t) - \Phi(X_t,\nu_t) 
- \varepsilon (\nabla_x \Phi)(X_t, \nu_t)V_t - \varepsilon \int \tfrac{\delta \Phi}{\delta m}(X_t,\nu_t, a)(\mu_t - \nu_t)(da)  \right]	\\
& =  \int_0^1 (\nabla_x \Phi)(X_t + \lambda \varepsilon ( V^\varepsilon_t + V_t), \nu_t) V^\varepsilon_t \,d\lambda 
+  \int_0^1 \left[(\nabla_x \Phi)(X_t + \lambda \varepsilon ( V^\varepsilon_t + V_t), \nu_t) - (\nabla_x \Phi)(X_t,\nu_t) \right] V_t\,d\lambda \\
& + \int_0^1 \int \left[ \tfrac{\delta \Phi}{\delta m}(X^\varepsilon_t, (1-\lambda)\nu^\varepsilon_t + \lambda \nu_t,a) - \tfrac{\delta \Phi}{\delta m}(X_t,\nu_t,a) \right](\mu_t-\nu_t)(da)\,d\lambda =: I^{(0)}_t + I^{(1)}_t + I^{(2)}_t =: I_t.
\end{split}
\]
Similarly, for the diffusion coefficient we get,  
\[
\begin{split}
& \frac1\varepsilon \left[\Gamma(X^\varepsilon_t, \nu^\varepsilon_t) - \Gamma(X_t,\nu_t) 
- \varepsilon (\nabla_x \Gamma)(X_t, \nu_t)V_t - \varepsilon \int \tfrac{\delta \Gamma}{\delta m}(X_t,\nu_t, a)(\mu_t - \nu_t)(da)  \right]	\\
& =  \int_0^1 (\nabla_x \Gamma)(X_t + \lambda \varepsilon ( V^\varepsilon_t + V_t), \nu_t) V^\varepsilon_t \,d\lambda 
+  \int_0^1 \left[(\nabla_x \Gamma)(X_t + \lambda \varepsilon ( V^\varepsilon_t + V_t), \nu_t) - (\nabla_x \Gamma)(X_t,\nu_t) \right] V_t\,d\lambda \\
& + \int_0^1 \int \left[ \tfrac{\delta \Gamma}{\delta m}(X^\varepsilon_t, (1-\lambda)\nu^\varepsilon_t + \lambda \nu_t,a) - \tfrac{\delta \Gamma}{\delta m}(X_t,\nu_t,a) \right](\mu_t-\nu_t)(da)\,d\lambda =: J^{(0)}_t + J^{(1)}_t + J^{(2)}_t =: J_t.
\end{split}
\]
Note that
\[
dV^\varepsilon_t = \varepsilon^{-1}[dX^\varepsilon_t - dX_t] - dV_t
\]
and so we then see that for any $p'\geq 1$ we have
\begin{equation*}
|V^\varepsilon_t|^{p'} \leq c_p\bigg|\int_0^t I_r \,dr\bigg|^{p'} + c_p\bigg|\int_0^t J_r\,dW_r\bigg|^{p'}\,.	
\end{equation*}
Hence, due to Burkholder--Davis--Gundy inequality we have, with a constant depending also on $d'$, that
\begin{equation}
\label{eq V proof 1}
\mathbb E \left[ \sup_{s\leq t}|V^\varepsilon_s|^{p'} \right]  \leq  c_{p,T} \mathbb E\left[\int_0^t |I_r|^{p'}\,dr + \bigg(\int_0^t |J_r|^2\,dr\bigg)^{p'/2} \,\right]\,.
\end{equation}
Let $\nu_t^{\lambda,\lambda'}:= (1-\lambda') ((1-\lambda)\nu^\varepsilon_t + \lambda \nu_t) + \lambda' \nu_t $. 
Due to the differentiability assumptions
\[
\begin{split}
& \int_0^1 \int   \tfrac{\delta \Phi}{\delta m}(X^\varepsilon_t, (1-\lambda)\nu^\varepsilon_t + \lambda \nu_t,a) - \tfrac{\delta \Phi}{\delta m}(X^\varepsilon_t,\nu_t,a) (\mu_t-\nu_t)(da)\,d\lambda \\
& =  \int_0^1 \int_0^1 (1-\lambda) \int \int   \tfrac{\delta^2 \Phi}{\delta m^2}(X^\varepsilon_t, \nu_t^{\lambda,\lambda'},a,a')(\nu^\varepsilon_t - \nu_t)(da') (\mu_t-\nu_t)(da)\,d\lambda\,d\lambda' 	 \\
& =  \varepsilon \int_0^1 \int_0^1 (1-\lambda) \int \int   \tfrac{\delta^2 \Phi}{\delta m^2}(X^\varepsilon_t, \nu_t^{\lambda,\lambda'},a,a')(\mu_t - \nu_t)(da') (\mu_t-\nu_t)(da)\,d\lambda\,d\lambda'\,.
\end{split}
\]
Hence, using the assumption of uniform bound on $\tfrac{\delta^2 \Phi}{\delta m^2}$, we have
\[
\begin{split}
\bigg| \int_0^1 \int   \tfrac{\delta \Phi}{\delta m}(X^\varepsilon_t, (1-\lambda)\nu^\varepsilon_t + \lambda \nu_t,a) - \tfrac{\delta \Phi}{\delta m}(X^\varepsilon_t,\nu_t,a) (\mu_t-\nu_t)(da)\,d\lambda \bigg|
\leq  c\epsilon\,.
\end{split}
\]
Using this and our assumptions again yields
\[
\begin{split}
|I^{(2)}_t|^{p'} & = \bigg|\int_0^1 \int \left[ \tfrac{\delta \Phi}{\delta m}(X^\varepsilon_t, (1-\lambda)\nu^\varepsilon_t + \lambda \nu_t,a) - \tfrac{\delta \Phi}{\delta m}(X_t,\nu_t,a) \right](\mu_t-\nu_t)(da)\,d\lambda\bigg|^{p'} \\
& \leq  c_{p'}\bigg|\int_0^1 \int \left[ \tfrac{\delta \Phi}{\delta m}(X^\varepsilon_t, (1-\lambda)\nu^\varepsilon_t + \lambda \nu_t,a) - \tfrac{\delta \Phi}{\delta m}(X_t^\varepsilon,\nu_t,a) \right](\mu_t-\nu_t)(da)\,d\lambda\bigg|^{p'} \\
& \qquad + c_{p'} \bigg(\int_0^1 \int \left| \tfrac{\delta \Phi}{\delta m}(X^\varepsilon_t, \nu_t ,a) - \tfrac{\delta \Phi}{\delta m}(X_t,\nu_t,a) \right||\mu_t-\nu_t|(da)\,d\lambda\bigg)^{p'} 
\\
 & \leq   \, c_{p'}\bigg|\int_0^1 \int   \tfrac{\delta \Phi}{\delta m}(X^\varepsilon_t, (1-\lambda)\nu^\varepsilon_t + \lambda \nu_t,a) - \tfrac{\delta \Phi}{\delta m}(X^\varepsilon_t,\nu_t,a) (\mu_t-\nu_t)(da)\,d\lambda\bigg|^{p'}  \\
 & \qquad + c_{p'} \bigg(\int_0^1 \int | X^\varepsilon_t - X_t| \, |\mu_t-\nu_t|(da)\,d\lambda\bigg)^{p'} 
 \leq c_{p'} \varepsilon^{p'}  + c_{p'} |X_t^\varepsilon - X_t|^{p'}\,. 
\end{split}
\]
Hence
\begin{equation}
\label{eq V proof I2 est}
 |I^{(2)}_t|^{p'}  \leq c_{p'} \varepsilon^{p'}  + c_{p'}  |X^\varepsilon_t  - X_t |^{p'}  =  c_{p'} \varepsilon^{p'}+ c_{p'} \varepsilon^{p'}  |V^\varepsilon_t + V_t |^{p'}   
\end{equation}
and similarly
\begin{equation}
\label{eq V proof J2 est}
 |J^{(2)}_t|^{p'} \leq c_{p'} \varepsilon^{p'}  + c_{p'}  |X^\varepsilon_t  - X_t |^{p'}  =  c_{p'} \varepsilon^{p'} + c_{p'} \varepsilon^{p'}  |V^\varepsilon_t + V_t |^{p'}   \,.
\end{equation}
By our hypothesis $\nabla_x \Phi$ and $\nabla_x \Gamma$ are bounded uniformly in $(t,x,m)$ which implies that 
\[
\begin{split}
|I^{(0)}_t|^{p'} \leq c_{p'} |V_t^\varepsilon|^{p'}\,,\quad 	|I^{(1)}_t|^{p'} \leq c_{p'} |V_t|^{p'}\,,\quad |J^{(0)}_t|^{p'} \leq c_{p'} |V_t^\varepsilon|^{p'}\,,\quad 	|J^{(1)}_t|^{p'} \leq c_{p'} |V_t|^{p'}\,.
\end{split}
\]
Applying H\"older's inequality in~\eqref{eq V proof 1} for ${p'}\geq 2$ and $\varepsilon \leq 1$ we get that
\begin{equation}
\label{eq V proof 2}
\mathbb E \left[ \sup_{t'\leq t}|V^\varepsilon_{t'}|^{p'} \right]  \leq  c_{p',T} \mathbb E\left[\int_0^t |I_r|^{p'}\,dr + \int_0^t |J_r|^{p'}\,dr \,\right]\,.
\end{equation}
Thus, for $\varepsilon \leq 1$ and $p'\geq 2$, we see that
\begin{equation*}
\begin{split}
\mathbb E \left[ \sup_{t'\leq t}|V^\varepsilon_{t'}|^{p'} \right] & \leq  c_{p',T} \mathbb E\left[\int_0^t |I_r|^{p'}\,dr + \int_0^t |J_r|^{p'}\,dr \,\right]\\
& \leq c_{p',T} + c_{p',T} \mathbb E \int_0^t |V^\varepsilon_r|^{p'}\,dr + c_{p',T} \mathbb E \int_0^t |V_r|^{p'}\,dr 
\leq c_{p',T}  \int_0^t \mathbb E \sup_{r'\leq r}|V^\varepsilon_{r'}|^{p'}\,dr + c_{p',T}\,.
\end{split}
\end{equation*}
Gronwall's lemma yields, for any $p'\geq 2$, that
\begin{equation}
\label{eq V proof 3}
\sup_{\varepsilon \leq 1} \mathbb E \left[ \sup_{t'\leq t}|V^\varepsilon_{t'}|^{p'} \right] < \infty\,.	
\end{equation}
By our hypothesis we have that $\nabla_x \Phi$ and $\nabla_x \Gamma$ are Lipschitz continuous is $x$ uniformly in $t$ and $m$.
Consequently, with Young's inequality, one sees that
\[
\begin{split}
\mathbb E \int_0^T \Big[|I^{(1)}_t|^2+|J^{(1)}_t|^2\Big]\,dt & \leq  \varepsilon^2 c \int_0^T \mathbb E \Big[   |V_t^\varepsilon + V_t|^2|V_t|^2 \Big]\,dt 
\leq 
\varepsilon^2 c \int_0^T \mathbb E \Big[ |V_t^\varepsilon + V_t|^{4} + |V_t|^{4} \Big]\,dt \\
& \leq \varepsilon^2 c \int_0^T \mathbb E \Big[ |V_t^\varepsilon|^{4} + |V_t|^{4} + |V_t|^{4} \Big]\,dt
\leq \varepsilon^2 c \int_0^T \mathbb E \Big[ |V_t^\varepsilon|^{4} +  |V_t|^{4} \Big]\,dt \,.	
\end{split}
\]
With~\eqref{eq V proof 3} used with $p'=4$ and with the observation made earlier that $V_t$ has all the moments bounded we get that 
\begin{equation}
\label{eq V proof 4}
\mathbb E \int_0^T \Big[|I^{(1)}_t|^2+|J^{(1)}_t|^2\Big]\,dt \leq c \varepsilon^2\,.	
\end{equation}
Since we already established~\eqref{eq V proof I2 est}-\eqref{eq V proof J2 est} 
and since we may apply~\eqref{eq V proof 3} with $p'=2$ and since $V_t$ has all the moments bounded   
\begin{equation}
\label{eq V proof 5}
\mathbb E \int_0^T \Big[ |I^{(2)}_t|^2 + |J^{(2)}_t|^2\Big]\,dt \leq c\varepsilon^2 + c\varepsilon^2 \mathbb E \int_0^T |V_t^\varepsilon + V_t|^2\,dt \leq c \varepsilon^2\,.	
\end{equation}
From~\eqref{eq V proof 4} and~\eqref{eq V proof 5}, together with~\eqref{eq V proof 2}  it follows that
\[
\begin{split}
\mathbb E  \sup_{r\leq t}|V^\varepsilon_r|^2    & \leq  c \left( \int_0^t \mathbb E \sup_{r'\leq r}|V^\varepsilon_{r'}|^2\,dr + \varepsilon^2 \, \right)\,.	
\end{split}
\]
and by Gronwall's lemma $\mathbb E \sup_{s\leq t}|V^\varepsilon_s|^2  \,   \leq c_T\varepsilon^2 \to 0$ as $\varepsilon \to \infty$. 
\end{proof}

\begin{lemma}[Directional derivative of $J^0$ in terms of the variation process]
\label{lemma diff of J flat}
Let the hypothesis of Lemma~\ref{lemma V process linear} hold and additionally that $\nabla_x F$ and $\nabla_x g$, $\tfrac{\delta F}{\delta m}$, $\tfrac{\delta^2 F}{\delta m^2}$ exist and are jointly continuous in $(t,x,m)$ or $(t,x,m,a)$ respectively and there is  $K>0$ such that for all $(t,x,m,a,a')$ we have 
\[
\left|\tfrac{\delta^2 F_t}{\delta m^2}(x,m,a,a') \right| + \left|\nabla_x \tfrac{\delta F_t}{\delta m}(x,m,a) \right| + |\nabla_x^2 g(x)|\leq K\,\,\,\text{and}\,\,\, |\nabla_x g(x)| \leq K(1+|x|)\,.
\]
Then for any $\nu,\mu \in \mathcal V^W_{2}$ the mapping $\nu \mapsto J^0(\nu)$ defined by~\eqref{eq objective bar J} satisfies
\[
\begin{split}
\tfrac{d}{d\varepsilon} J^0 & \left((\nu_t + \varepsilon(\mu_t-\nu_t)_{t\in[0,T]}, \xi\right)\big|_{\varepsilon=0} \\
 = &  \mathbb E\Bigg[\int_0^T \left[\int \tfrac{\delta F_t}{\delta m}(X_t,\nu_t,a)\,(\mu_t-\nu_t)(da) + (\nabla_x F)(X_t,\nu_t) V_t  \right]\,dt  + (\nabla_x g)(X_T)V_T \Bigg]\,.	
\end{split}
\]	
\end{lemma}
\begin{proof}
We need to show that $I_\varepsilon :=  I^{(1)}_\varepsilon + I^{(2)}_\varepsilon \to 0$ as $\varepsilon\to 0$, where
\[
I^{(1)}_\varepsilon := \mathbb E \bigg| \int_0^T \Big[\varepsilon^{-1}\Big(F(X^\varepsilon_t,\nu^\varepsilon_t) - F(X_t,\nu_t)\Big) - \int \tfrac{\delta F_t}{\delta m}(X_t,\nu_t,a)\,(\mu_t-\nu_t)(da) - (\nabla_x F)(X_t,\nu_t) V_t \Big] \,dt \bigg|\,
\]
and
\[
I^{(2)}_\varepsilon := \mathbb E \Big|\varepsilon^{-1}\big(g(X^\varepsilon_T)-g(X_T)\big) - (\nabla_x g)(X_T)V_T \Big|\,. 
\]

We will start with the simpler case of $I^{(2)}_\varepsilon$.
First we note that
\[
\begin{split}
\varepsilon^{-1}\left[g(X^\varepsilon_T) - g(X_T)\right] & = \varepsilon^{-1}\int_0^1 \frac{d}{d\lambda}g(X_T + \lambda (X^\varepsilon_T-X_T)\,d\lambda 
= \varepsilon^{-1}\int_0^1 \frac{d}{d\lambda}g(X_T + \lambda \varepsilon (V^\varepsilon_T+V_T)\,d\lambda \\
& = \int_0^1 \nabla_x g(X_T + \lambda \varepsilon (V^\varepsilon_T+V_T))(V^\varepsilon_T + V_T) \,d\lambda
\end{split}
\]
and hence 
\[
\begin{split}
I^{(2)}_\varepsilon & = \mathbb E \left|\int_0^1 [\nabla_x g(X_T + \lambda \varepsilon (V^\varepsilon_T+V_T)) - \nabla_x g(X_T)](V^\varepsilon_T + V_T)\,d\lambda + \nabla_x g(X_T)V^\varepsilon_T \right| \\	
& \leq K \varepsilon \mathbb E |V^\varepsilon_T + V_T|^2  + \left(\mathbb E[c(1+|X_T^2|)\right)^{1/2} \left(\mathbb E|V^\varepsilon_T|^2\right)^{1/2}\,,
\end{split}
\]
where we used the assumption that $\|\nabla_x^2g\|\leq K$, that $\forall x\in \mathbb R^d$ we have $|\nabla_x g(x)| \leq K(1+|x|)$ and H\"older's inequality. 
Due to Lemma~\ref{lemma V process linear} we know that $\mathbb E |V^\varepsilon_T + V_T|^2$ is bounded uniformly in $\varepsilon$ and moreover $\left(\mathbb E|V^\varepsilon_T|^2\right)^{1/2} \to 0$ as $\varepsilon \to 0$. 
Hence $I^{(2)}_\varepsilon \to 0$.

To handle $I^{(1)}_\varepsilon$ first note that
\[
\begin{split}
&  \varepsilon^{-1} \left[F(X^{\varepsilon}_t, \nu^\varepsilon_t) - F(X^{\varepsilon}_t, \nu_t) + F(X^{\varepsilon}_t, \nu_t) - F(X_t, \nu_t) \right]\\
& =  \int_0^1 \int \tfrac{\delta F_t}{\delta m}(X^\varepsilon_t,  (1-\lambda)\nu^\varepsilon_t + \lambda \nu_t, a )(\mu_t - \nu_t)(da)	\,d\lambda 
+ \int_0^1(\nabla_x F_t)(X_t + \lambda \varepsilon (V^{\varepsilon}_t + V_t),\nu_t)(V^{\varepsilon}_t - V_t)\,d\lambda\,.
\end{split}
\]
Hence we will write $I^{(1)}_\varepsilon = I^{(1,1)}_\varepsilon + I^{(1,2)}_\varepsilon + I^{(1,3)}_\varepsilon$, where
\[
I^{(1,1)}_\varepsilon := \mathbb E \bigg| \int_0^T \bigg[\int_0^1 \int \tfrac{\delta F_t}{\delta m}(X^\varepsilon_t,  (1-\lambda)\nu^\varepsilon_t + \lambda \nu_t, a )(\mu_t - \nu_t)(da)	\,d\lambda - \int \tfrac{\delta F_t}{\delta m}(X^\varepsilon_t,\nu_t,a)\,(\mu_t-\nu_t)(da)\bigg]\,dt\bigg|\,,
\]
\[
I^{(1,2)}_\varepsilon := \mathbb E \bigg| \int_0^T \bigg[\int_0^1 \int \tfrac{\delta F_t}{\delta m}(X^\varepsilon_t,  \nu_t, a )(\mu_t - \nu_t)(da)	\,d\lambda - \int \tfrac{\delta F_t}{\delta m}(X_t,\nu_t,a)\,(\mu_t-\nu_t)(da)\bigg]\,dt\bigg|
\]
and
\[
I^{(1,3)}_\varepsilon := \mathbb E \bigg| \int_0^T \bigg[\int_0^1(\nabla_x F_t)(X_t + \lambda \varepsilon (V^{\varepsilon}_t + V_t),\nu_t)(V^{\varepsilon}_t - V_t)\,d\lambda - (\nabla_x F)(X_t,\nu_t) V_t \bigg]\,dt\bigg|\,.
\]
First, for $\lambda, \lambda' \in [0,1]$ and $\varepsilon > 0$ fixed, let $\nu_t^{\lambda,\lambda'}:= (1-\lambda') ((1-\lambda)\nu^\varepsilon_t + \lambda \nu_t) + \lambda' \nu_t $.
Using the assumption that $\tfrac{\delta^2 F}{\delta m^2}$ exists and is uniformly bounded we get
\[
\begin{split}
I^{(1,2)}_\varepsilon & \leq 
\mathbb E \int_0^T \bigg|\int_0^1 \int_0^1 (1-\lambda) \int \int \tfrac{\delta^2 F_t}{\delta m^2}(X_t, \nu^{\lambda,\lambda'}_t,a,a')(\nu^\varepsilon_t - \nu_t)(da')(\mu_t-\nu_t)(da)\,d\lambda\,d\lambda'\,dt\bigg|\\
& \leq \varepsilon \mathbb E \int_0^T \bigg|\int_0^1 \int_0^1 (1-\lambda) \int \int \tfrac{\delta^2 F_t}{\delta m^2}(X_t, \nu^{\lambda,\lambda'}_t,a,a')(\mu_t - \nu_t)(da')(\mu_t-\nu_t)(da)\,d\lambda\,d\lambda'\,dt\bigg|\leq \varepsilon 4 KT\,.
\end{split}
\] 
Next, using the assumption that $\nabla_x \tfrac{\delta F}{\delta m}$ exists and is uniformly bounded together with H\"older's inequality we get
\[
I^{(1,3)}_\varepsilon \leq  \mathbb E\int_0^T \bigg[ \int_0^1\int K\lambda\varepsilon |V^\varepsilon_t + V_t||\mu_t-\nu_t|(da)\,d\lambda
\leq \varepsilon 2KT \left(\mathbb  E\sup_{t\in[0,T]} |V^\varepsilon_t + V_t|^2\right)^{1/2}
\]
Thus due to Lemma~\ref{lemma V process linear} we get that $I^{(1)}_\varepsilon\to 0$ as $\varepsilon \to 0$ which concludes the proof.
\end{proof}

\begin{lemma}
\label{lemma der of J as hamiltonian flat}[Directional derivative of $J^0$ in terms of the linear functional derivative of the Hamiltonian]
Under the hypothesis of Lemma~\ref{lemma diff of J flat} for any $\nu,\mu \in \mathcal V^W_{2}$ we have that
\begin{equation*}
\tfrac{d}{d\varepsilon} J^0\left((\nu_t + \varepsilon(\mu_t-\nu_t)_{t\in[0,T]},\xi\right)\big|_{\varepsilon=0} 
= \mathbb E \bigg[ \int_0^T \left[\int \tfrac{\delta H^0}{\delta m}(X_t,Y_t,Z_t,\nu_t,a) (\mu_t-\nu_t)(da)   \right]\,dt \bigg]\,.	
\end{equation*}	
\end{lemma}

\begin{proof}[Proof of Lemma~\ref{lemma der of J as hamiltonian flat}]
We first observe that due to~\eqref{eq adjoint proc}, ~\eqref{eq V proc linear} and the fact that $V_0 = 0$, we have
\[
\begin{split}
& Y_T V_T  =  \int_0^T Y_t \, dV_t + \int_0^T V_t \,dY_t + \int_0^T \,d\langle V_t, Y_t\rangle \\	
& =  \int_0^T Y_t \left[ (\nabla_x \Phi_t)(X_t,\nu_t) V_t + \int \tfrac{\delta \Phi_t}{\delta m}(X_t,\nu_t,a)\,(\mu_t  - \nu_t)(da)\right]\,dt - \int_0^T V_t  (\nabla_x H^0_t)(X_t,Y_t,Z_t,\nu_t)\,dt \\
& \quad + \int_0^T \text{tr}\left(Z_t^\top  (\nabla_x \Gamma)(X_t,\nu_t)V_t + Z_t^\top\int \tfrac{\delta \Gamma_t}{\delta m}(X_t,\nu_t,a)(\mu_t-\nu_t)(da) \right)\,dt + M_T\,,
\end{split}
\]
where $M$ is a local martingale. 
Using the definition of the Hamiltonian~\eqref{eq hamiltonian} we get 
\[
\begin{split}
 Y_T V_T  = & \int_0^T Y_t \left[(\nabla_x \Phi)(X_t,\nu_t) V_t + \int \tfrac{\delta \Phi_t}{\delta m}\phi(X_t,\nu_t,a)\,(\mu_t-\nu_t) (da)\right]\,dt \\
 & + \int_0^T \text{tr}\left(Z_t^\top  (\nabla_x \Gamma)(X_t,\nu_t)V_t + Z_t^\top\int \tfrac{\delta \Gamma_t}{\delta m}(X_t,\nu_t,a)(\mu_t-\nu_t)(da) \right)\,dt \\
& - \int_0^T    \left[  Y_t (\nabla_x \Phi_t)(X_t,\nu_t)  V_t  + \text{tr}\big[Z_t^\top (\nabla_x \Gamma_t)(X_t,\nu_t) V_t \big] +  V_t (\nabla_x F_t)(X_t,\nu_t)  \right]\,dt + M_T
\end{split}
\]
and so
\[
\begin{split}
Y_T V_T =  & \int_0^T  \bigg[\int Y_t \tfrac{\delta \Phi_t}{\delta m}(X_t,\nu_t,a)\,(\mu_t-\nu_t)(da) 
+ \text{tr}\left[\int Z_t^\top \tfrac{\delta \Gamma_t}{\delta m}(X_t,\nu_t,a)(\mu_t-\nu_t)(da)\right]\\
& - \int V_t (\nabla_x F)(X_t,\nu_t)\,\bigg]\,dt + M_T\,.	
\end{split}
\]
From this, Lemma~\ref{lemma diff of J flat}, noting that $(\nabla_x g)(X_T)V_T =Y_T V_T$ and a usual stopping time argument we see that
\[
\begin{split}
& \tfrac{d}{d\varepsilon} J^0\left((\nu_t + \varepsilon(\mu_t-\nu_t)_{t\in[0,T]},\xi\right)\big|_{\varepsilon=0} 
=  \mathbb E  \int_0^T \bigg[\int \tfrac{\delta F_t}{\delta m}(X_t,\nu_t,a)\,(\mu_t-\nu_t)(da) +  (\nabla_x F_t)(X_t,\nu_t) V_t   \\
& + \int Y_t \tfrac{\delta \Phi_t}{\delta m}(X_t,\nu_t,a)\,(\mu_t-\nu_t)(da) + \text{tr}\left( Z_t^\top \int   \tfrac{\delta \Gamma_t}{\delta m}(X_t,\nu_t,a)(\mu_t-\nu_t)(da) \right) - \int V_t (\nabla_x F)(X_t,\nu_t)\bigg]\,dt  \\
& =  \mathbb E \Bigg[ \int_0^T \int \left[  \tfrac{\delta H^0_t}{\delta m}(X_t,Y_t,Z_t,\nu_t,a) 
\right](\mu_t-\nu_t)(da)\,dt\,\Bigg]\,.
\end{split}
\]
This concludes the proof.
\end{proof}

\begin{proof}[Proof of Theorem~\ref{th with measure flow}] 

Let us write $\nu_t^{\varepsilon} := \nu_t+\varepsilon(\mu_t-\nu_t)$
and
$\mu_t^\varepsilon := \mu_t + \varepsilon(\mu_t - \nu_t)$.
Note that $\mu_t^\varepsilon = \mu_t - \nu_t + \nu_t^\varepsilon$ and so $\mu_t^\varepsilon - \nu_t^\varepsilon = \mu_t-\nu_t$.
From the Fundamental Theorem of Calculus we get
\[
\begin{split}
J^0(\mu) - J^0(\nu) & = \int_0^1 \lim_{\delta\to 0} \frac{1}{\delta}\Big( J^0\big(\nu + (\varepsilon + \delta)(\mu-\nu),\xi\big) - J^0\big(\nu + \varepsilon(\mu-\nu),\xi\big)  \Big) d\varepsilon\\
& = \int_0^1 \lim_{\delta\to 0} \frac{1}{\delta}\Big( J^0\big(\nu^\varepsilon + \delta(\mu^\varepsilon-\nu^\varepsilon),\xi\big) - J^0\big(\nu^\varepsilon,\xi\big)  \Big) d\varepsilon\,.	
\end{split}
\]
Due to Lemma~\ref{lemma der of J as hamiltonian flat} and using the notation introduced in~\eqref{eq fat h} we have
\[
\lim_{\delta\to 0} \frac{1}{\delta}\Big( J^0\big(\nu^\varepsilon + \delta(\mu^\varepsilon-\nu^\varepsilon),\xi\big) - J^0\big(\nu^\varepsilon,\xi\big)  \Big)
= \mathbb E \Bigg[ \int_0^T \left[\int \tfrac{\delta \mathbf H_t^0}{\delta m}(\nu_t^\varepsilon,a) (\mu_t^\varepsilon-\nu_t^\varepsilon)(da)   \right]\,dt\Bigg]\,.	 
\]	
Hence
\begin{equation*}
J^0(\mu) - J^0(\nu) = \int_0^1 \mathbb E \Bigg[ \int_0^T \left[\int \tfrac{\delta \mathbf H_t^0}{\delta m}(\nu_t^\varepsilon,a) (\mu_t-\nu_t)(da)   \right]\,dt\Bigg]\,d\varepsilon
\end{equation*}	
Take $\mu_t = \nu_{s+h,t}$ and $\nu_t = \nu_{s,t}$  and write $\nu_{s,t}^{\varepsilon,h}= \nu_{s,t}+\varepsilon(\nu_{s+h,t}-\nu_{s,t})$. Note that $\nu_t^{\varepsilon,h}\rightarrow \nu_t$ as $h\downarrow0$.  
Hence, using~\eqref{eq fpe}, we get
\[
\begin{split}
\tfrac{d}{ds} J^0(\nu_{s,\cdot}) & = \lim_{h\downarrow0} h^{-1}\left(J^0(\nu_{s+h,t}) - J^0(\nu_{s,t})\right) 
 = \int_0^1 \!\!
\mathbb E \Bigg[ \int_0^T \lim_{h\downarrow0}  \bigg[\int\tfrac{\delta \mathbf H_t^0}{\delta m}(\nu_t^{\varepsilon,h},a)\frac1h(\nu_{s+h,t}-\nu_{s,t})(da) \bigg]\,dt\Bigg]\,d\varepsilon\\
& =
\mathbb E \Bigg[ \int_0^T   \bigg[\int\tfrac{\delta \mathbf H_t^0}{\delta m}(\nu_t,\cdot)\nabla_a \cdot \bigg(\Big(b_{s,t}+\tfrac{\sigma^2}{2} \nabla_a \log \nu_{s,t}\Big)\, \nu_{s,t} \bigg)(da) \bigg]\,dt\Bigg]\,.    
\end{split}
\]
Due to Lemma~\ref{lemma prop grad flow} we have that $\int |\nabla_a \log \nu_{s,t}(a)|^2 \nu_{s,t}(a)\,da < \infty$ and due to standard moment estimates for SDEs that $\int |a|^m \nu_{s,t}(a)\,da <\infty$ for any $m\geq 2$ and for all $s>0$, $t\in [0,T]$ and $\omega^W \in \Omega^W$.
The continuity $a\mapsto \nu_{s,t}(a)$ then implies that $\nabla_a |\log \nu_{s,t}(a)|^2\nu_{s,t}(a) \to 0$ and $|a|^2\nu_{s,t}(a) \to 0$ as $|a|\to \infty$. 
Thus, upon integration by parts there are no boundary terms left, and we get that
\begin{equation}
\label{eq pf of flow thm 1}
\begin{split}
\tfrac{d}{ds} J^0(\nu_{s,\cdot}) = & 
\mathbb E^W \Bigg[ \int_0^T \bigg[-\int \left(\nabla_a\tfrac{\delta \mathbf H_t^0}{\delta m}\right)(\nu_{s,t},\cdot) \Big(b_{s,t}+\tfrac{\sigma^2}{2} \nabla_a \log \nu_{s,t}\Big) \,\nu_{s,t}(da) \bigg]\,dt  \Bigg] \,.	
\end{split}	
\end{equation}
Recalling the definition of $J^\sigma$ and combining~\eqref{eq pf of flow thm 1} and Lemma~\ref{lemma derivative of entropy along flow} we get
\[
\tfrac{d}{ds}J^\sigma(\nu_{s,\cdot}) = - \mathbb E^W \Bigg[\int_0^T \int \bigg[ \left(\nabla_a\tfrac{\delta \mathbf H_t^0}{\delta m}\right)(\nu_{s,t},\cdot) + \tfrac{\sigma^2}{2}\frac{\nabla_a \nu_{s,t}}{\nu_{s,t}} + \tfrac{\sigma^2}{2}\nabla_a U\bigg]\cdot \Big(b_{s,t}+\tfrac{\sigma^2}{2} \nabla_a \log \nu_{s,t}\Big) \nu_{s,t}(da) \,dt \Bigg]\,.
\]
This completes the proof.
\end{proof}

\begin{proof}[Proof of Theorem~\ref{thm necessary cond linear}.]
Let $(\mu_t)_{t\in[0,T]}$  be an arbitrary relaxed control.
Since $(\nu_t)_{t\in[0,T]}$ is optimal we know that 	
$J^\sigma\left(\nu_t + \varepsilon(\mu_t - \nu_t))_{t\in[0,T]}\right) \geq J^\sigma(\nu)$ for any $\varepsilon > 0$.
From this, Lemma~\ref{lemma der of J as hamiltonian flat} and~\ref{lemma diff of Ent flat} we get that
\[
\begin{split}
0 & \leq \limsup_{\varepsilon \to 0} \tfrac1\varepsilon\left(J^\sigma(\nu_t + \varepsilon(\mu_t - \nu_t))_{t\in[0,T]} - J^\sigma(\nu)\right)\\
& \leq \mathbb E\int_0^T  \int \left[\tfrac{\delta H^0}{\delta m}(X_t, Y_t, Z_t,\nu_t,a) + \tfrac{\sigma^2}2(\log \nu_t(a)  - \log \gamma(a))\right]\, (\mu_t-\nu_t)(da)  \,dt\,.	
\end{split}
\]
Now assume there is $S \in \mathcal F \otimes \mathcal B([0,T])$ with strictly positive $\mathbb P \otimes \lambda$ (with $\lambda$ denoting the Lebesgue on $\mathcal B([0,T])$) measure such that 
\[
\mathbb E\int_0^T \mathds{1}_S \int \left[\tfrac{\delta H^0}{\delta m}(X_t, Y_t, Z_t,\nu_t,a) + \tfrac{\sigma^2}2(\log \nu_t(a)  - \log \gamma(a))\right]\, (\mu_t-\nu_t)(da)\,dt < 0
\]
Define $\tilde \mu_t := \mu_t \mathds{1}_S + \nu_t \mathds{1}_{S^c}$. 
Then by the same argument as above 
\[
\begin{split}
0 & \leq \mathbb E \int_0^T \int \left[\tfrac{\delta H^0}{\delta m}(X_t, Y_t, Z_t,\nu_t,a) + \tfrac{\sigma^2}2(\log \nu_t(a)  - \log \gamma(a))\right]\, (\tilde\mu_t-\nu_t)(da)\,dt \\
& = \mathbb E\int_0^T \mathds{1}_S \int \left[\tfrac{\delta H^0}{\delta m}(X_t, Y_t, Z_t,\nu_t,a) + \tfrac{\sigma^2}2(\log \nu_t(a)  - \log \gamma(a))\right] \, (\mu_t-\nu_t)(da)\,dt < 0	
\end{split}
\]
leading to a contradiction. 
\end{proof}

\begin{proof}[Proof of Theorem \ref{thm unique minimiser}]

Fix $t\in [0,T]$ and $\omega^W \in \Omega^W$. 
Let $b_{s,t}(a) := (\nabla_a \tfrac{\delta \mathbf H_t^{0}}{\delta m})(\mu_s, a) + \tfrac{\sigma^2}{2}(\nabla_a U)(a)$
and $\mu_{s,t} = \mathcal L(\theta_{s,t} | \mathcal F_t^W)$.
As in the proof of Lemma~\ref{lem pde} we see that $\mu_{s,t}$ is a solution to
\begin{equation}
\label{eq fpe2}
\partial_s \mu_{s,t} = \nabla_a \cdot \left( b_{s,t} \mu_{s,t} + \tfrac{\sigma^2}{2}\nabla_a \mu_{s,t}\right)\,,\,\,\,s\geq0\,, \,\,\, \mu_{s,0} = \mu^0_t := \mathcal L(\theta_{0,t}| \mathcal F_t^W)\,.	
\end{equation}
Due to Lemma~\ref{lem pde} we know that the solution is unique and moreover for each $t\in [0,T]$ and $\omega^W \in \Omega^W$ fixed we have $\mu_{\,\cdot,t} \in C^{1,\infty}((0,\infty)\times \mathbb R^p; \mathbb R)$.
Also, $P_s \mu^0 = \mu_{s,\cdot}$ so $P_s$ is the solution operator for~\eqref{eq fpe2}.

Since $\mu^\ast$ is an invariant measure $0 = \rho(P_s \mu^\ast, \mu^\ast) = \big(\mathbb E^W\big[\mathcal W_2^T(P_s \mu^\ast, \mu^\ast)^2\big)^{1/2}$ we get that for almost all  $t\in [0,T]$ and $\omega^W \in \Omega^W$ we have $(P_s \mu^\ast)_t = \mu^\ast_t$ and so $\partial_s \mu^\ast_{s,t} = 0$.
Hence for almost all  $t\in [0,T]$ and $\omega^W \in \Omega^W$ we have that $\mu^\ast_t$ is a solution to the stationary Kolmogorov--Fokker--Planck equation
\begin{equation}
\label{eq stationary fpe}
0 = \nabla_a \cdot \bigg( \Big( (\nabla_a \tfrac{\delta \mathbf H_t^{0}}{\delta m})(\mu^\ast,\cdot) + \tfrac{\sigma^2}{2}(\nabla_a U) \Big) \mu^\ast_{t} + \tfrac{\sigma^2}{2}\nabla_a \mu^\ast_{t}\bigg)\,.	
\end{equation}
This implies that $\mu^\ast \in \mathcal I^\sigma$.
This proves item i).

Next we will show that if the invariant measure is unique then  $\{\mu^\ast\} = \mathcal I^\sigma$.
To that end we will first show by contradiction that $\mu^\ast$ is at least (locally) optimal.
Assume that $\mu^\ast$ is not  the (locally) optimal control for $J^\sigma$ defined in~\eqref{eq objective bar J}. 
Then for some $\mu^0 \in \mathcal V_2^W$ it holds that $J^\sigma(\mu^0) < J^\sigma(\mu^\ast)$.
We have by assumption that $\lim_{s\to\infty} P_s \mu^0 = \mu^\ast$.
From this, Theorem~\ref{th with measure flow} and from lower semi-continuity of $J^\sigma$ we get 
\begin{equation}
\label{eq proof of uniq min}
\begin{split}
 J^\sigma(\mu^\ast) - J^\sigma(\mu^0) & \leq \liminf_{s\to \infty} J^\sigma(P_s \mu^0) - J^\sigma(\mu^0) \\
& = - \liminf_{s\to \infty}\int_0^s \mathbb E^W \int_0^T  \left[\int  \left| \left(\nabla_a\tfrac{\delta \mathbf H^\sigma}{\delta m}\right)((P_s\mu^0)_t,a) \right|^2 \,(P_s\mu^0)_t(da)   \right]\,dt\,ds \leq 0	\,.
\end{split}
\end{equation}
This is a contradiction and so $\mu^\ast$ must be (locally) optimal.

On the other hand for any (locally) optimal control $\nu^\ast \in \mathcal V_2^W$ we have for any $\nu \in \mathcal V_2^W$, due to Theorem~\ref{thm necessary cond linear} that 
\[
0 \leq \mathbb E^W \bigg[ \int_0^T \int  \tfrac{\delta \mathbf H^\sigma_t}{\delta m}(\nu^\ast,a)(\nu_t - \nu^\ast_t)(da)\,dt \bigg] \,.
\]
This together Lemma~\ref{lem constant} implies that $\nu^\ast$ is a constant function of $a$ and so $\nu^\ast \in \mathcal I^\sigma$, where $\mathcal I^\sigma$ is defined in~\eqref{eq foc}.
From the uniqueness of $\mu^\ast$ as the invariant measure we get that the set of local minimizers is a singleton. 
Consider now any $\mu^0 \in \mathcal V^W_2$ s.t. $\mu^0\neq \mu^\ast$. 
We know $\mu^0$ is not a local minimizer. 
Moreover from~\eqref{eq proof of uniq min} we know that $J^\sigma(\mu^\ast) \leq J^\sigma(\mu^0)$.
This completes the proof.
\end{proof}

\section{Existence and Uniqueness of Solutions to the Mean-field system and exponential convergence to the invariant measure}
\label{sec ex uniq and conve}

\begin{lemma}[System dissipativity]
\label{lem sys diss}
Let Assumptions~\ref{ass dissipativity of F},~\ref{ass for bsde estimates} and \ref{ass for gradient system} hold. 
Let $(\theta_{s,t})_{s\geq 0, t\in [0,T]}$ be a solution to~\eqref{eq mfsgd}-\eqref{eq mfsgd 2} and let $\mu_{s,t} := \mathcal L(\theta_{s,t}|\mathcal F^W_t)$.
If $\mathbb E\int_0^T |\theta_{0,t}|^2 \,dt < \infty$
then there is $\hat c>0$ such that for any $s\geq 0$ we have
\[
-\mathbb E \int_0^T (\nabla_a \tfrac{\delta \mathbf H_t^{0}}{\delta m})(\mu_{s,\cdot},\theta_{s,t}) \theta_{s,t}\,dt
\leq -\frac{\kappa_f}{2} \mathbb E \int_0^T |\theta_{s,t}|^2\,dt  + \hat c\left(1 + |\xi|^2 + \mathbb E\int_0^T |\theta_{s,t}|^2\,dt\right)\,.
\]
The constant $\hat c$ is independent of $\kappa_f$, $\kappa_u$ and $\sigma$ but depends on $T$, $K$ and the constants arising in Lemmas~\ref{bsde bounds} and~\ref{lemma existence for fixed control}.	
\end{lemma}
\begin{proof}
Recall that
\[
\begin{split}
(\nabla_a \tfrac{\delta \mathbf H_t^{0}}{\delta m})(\mu_{s,\cdot},\theta_{s,t}) 
= &  (\nabla_a \tfrac{\delta \Phi_t}{\delta m})(X_t(\mu_{s,\cdot}), \mu_{s,t}, \theta_{s,t}(\mu)) \,Y_t(\mu_{s,\cdot})
+ (\nabla_a \tfrac{\delta \Gamma_t}{\delta m})(X_t(\mu_{s,\cdot}), \mu_{s,t}, \theta_{s,t}(\mu)) \,Z_t(\mu_{s,\cdot})\\
& + (\nabla_a \tfrac{\delta F^L_t}{\delta m})(X_t(\mu_{s,\cdot}), \mu_{s,t}, \theta_{s,t}(\mu)) 
+ (\nabla_a \tfrac{\delta F^C_t}{\delta m})(X_t(\mu_{s,\cdot}), \mu_{s,t}, \theta_{s,t}(\mu))\,. 	
\end{split}
\]
Then due to Assumption~\ref{ass dissipativity of F} we have
\begin{align}
\label{eq proof system diss 1}
- \mathbb E \int_0^T (\nabla_a \tfrac{\delta \mathbf H_t^{0}}{\delta m})(\mu_{s,\cdot},\theta_{s,t}) \theta_{s,t}\,dt
\leq  
\mathbb E\int_0^T|\theta_{s,t}|\Big(|I^{(1)}_{s,t}|+|I^{(2)}_{s,t}|+|I^{(3)}_{s,t}|\Big)\,dt - \frac{\kappa_f}{2}\mathbb E\int_0^T |\theta_{s,t}|^2\,dt  \,,
\end{align}
where
\[
I^{(1)}_{s,t} = (\nabla_a \tfrac{\delta \Phi_t}{\delta m})(X_t(\mu_{s,\cdot}), \mu_{s,t}, \theta_{s,t}(\mu)) \,Y_t(\mu_{s,\cdot})\,,\,\,\,
I^{(2)}_{s,t} = (\nabla_a \tfrac{\delta \Gamma_t}{\delta m})(X_t(\mu_{s,\cdot}), \mu_{s,t}, \theta_{s,t}(\mu)) \,Z_t(\mu_{s,\cdot})\,,
\]
\[
I^{(3)}_{s,t} = (\nabla_a \tfrac{\delta F^L_t}{\delta m})(X_t(\mu_{s,\cdot}), \mu_{s,t}, \theta_{s,t}(\mu)) \,.
\]
We now observe that Assumption~\ref{ass for gradient system} ii) yields
\[
|I^{(1)}_{s,t}| \leq K\left(1 + |X_t(\mu_{s,\cdot})|+\mathcal W_2(\mu_{s,t},\delta_0) + |\theta_{s,t}| \right)|Y_t(\mu_{s,\cdot})|\,.
\] 
With Lemmas~\ref{bsde bounds} and~\ref{lemma existence for fixed control} and Young's inequality we observe that
\begin{equation}
\label{eq proof system diss 2}
\begin{split}
\mathbb E\int_0^T|\theta_{s,t}||I^{(1)}_{s,t}|\,dt & \leq K\|Y(\mu(\nu_{s,\cdot})\|_{\mathbb H^\infty}\mathbb E\int_0^T \left[\tfrac12 + \tfrac12|X_t(\mu_{s,\cdot})|^2 + \tfrac12 \mathcal W_2^2(\mu_{s,t},\delta_0) + \tfrac52 |\theta_{s,t}|^2\right]\,dt \\
& \leq cT + cT(1+|\xi|^2) + c\mathbb E\int_0^T |\theta_{s,t}|^2\,dt\,.	
\end{split}
\end{equation}
With Assumption~\ref{ass for gradient system} i), Corollary~\ref{corollary to fefferman} and Lemma~\ref{bsde bounds} we see that 
\begin{equation}
\label{eq proof system diss 3}
\begin{split}
\mathbb E\int_0^T|\theta_{s,t}||I^{(2)}_{s,t}|\,dt 
& \leq K\mathbb E\int_0^T|\theta_{s,t}||Z_{t}(\mu_{s,\cdot})|\,dt 
\leq K \sqrt{2} \|Z(\mu_{s,\cdot})\cdot W\|_{\text{BMO}} \,\mathbb E\bigg[\bigg(\int_0^T |\theta_{s,t}|^2\,dt\bigg)^{1/2}\bigg] \\
& \leq c\bigg(1+\mathbb E \int_0^T |\theta_{s,t}|^2\,dt\bigg)\,.
\end{split}
\end{equation}
Finally, again with Lemma~\ref{lemma existence for fixed control} and Young's inequality, we obtain 
\[
\mathbb E\int_0^T|\theta_{s,t}||I^{(3)}_{s,t}|\,dt \leq cT + cT(1+|\xi|^2) + c\mathbb E\int_0^T |\theta_{s,t}|^2\,dt\,.
\]
Combining this, with~\eqref{eq proof system diss 1}, \eqref{eq proof system diss 2} and~\eqref{eq proof system diss 3} we conclude that
\begin{equation*}
-\mathbb E \int_0^T (\nabla_a \tfrac{\delta \mathbf H_t^{0}}{\delta m})(\mu_{s,\cdot},\theta_{s,t}) \theta_{s,t}\,dt
\leq -\frac{\kappa_f}{2} \mathbb E \int_0^T |\theta_{s,t}|^2\,dt  + c + cT + cT(1+|\xi|^2) + c\mathbb E\int_0^T |\theta_{s,t}|^2\,dt\,.
\end{equation*}
Choosing $\hat c>0$ appropriately concludes the proof.	
\end{proof}

\begin{lemma}[A priori estimate]
\label{lem apriori}
Let Assumptions~\ref{ass dissipativity of U}, \ref{ass dissipativity of F}, \ref{ass for bsde estimates} and~\ref{ass for gradient system} hold. 
Let 
$(\theta_{s,t})_{s\geq 0, t\in [0,T]}$ be a solution to~\eqref{eq mfsgd}-\eqref{eq mfsgd 2}.
Let $\mathbb E\int_0^T |\theta_{0,t}|^2 \,dt < \infty$.
If $\lambda:=\sigma^2\kappa_u + \kappa_f - \hat c>0$
 then for any $s\geq 0$ we have
\begin{equation}
\int_{0}^{T}\mathbb E|\theta_{s,t}|^2\,dt
\leq e^{-\lambda s} \int_0^T \mathbb E |\theta_{0,t}|^2 \,dt + \frac1\lambda(p\sigma^2 T + \hat c(1+|\xi|^2)\,.
\end{equation}	
\end{lemma}
\begin{proof}
Let $\mu_{s,t} := \mathcal L(\theta_{s,t}|\mathcal F^W_t)$.
Applying It\^o's formula in~\eqref{eq mfsgd}, first with $t\in [0,T]$ fixed, then integrating over $[0,T]$, leads to
\begin{align*}
& e^{\lambda s} \int_{0}^{T}|\theta_{s,t}|^2\,dt
=\int_0^T |\theta_{0,t}|^2 \,dt  + \lambda \int_{0}^{s} e^{\lambda v} \int_{0}^{T} |\theta_{v,t}|^2\,dt\,dv+p\sigma^2 T \int_{0}^{s}e^{\lambda v}  \,dv\\
&\quad -2\int_{0}^{s}e^{\lambda v} \int_{0}^{T} \bigg(\tfrac{\sigma^2}{2}\nabla_a U_t(\theta_{v,t})
+(\nabla_a \tfrac{\delta \mathbf H_t^{0}}{\delta m})(\mu_{v,\cdot},\theta_{v,t})\bigg) \theta_{v,t} \,dt\,dv + 2 \sigma  \int_0^T \int_0^s e^{\lambda v}  \theta_{v,s}\,  dB_v\,dt\,.\\ 
\end{align*}
Taking expectation (after employing the usual stopping time arguments) yields that
\begin{align*}
& e^{\lambda s} \int_{0}^{T}\mathbb E|\theta_{s,t}|^2\,dt
=\int_0^T \mathbb E |\theta_{0,t}|^2 \,dt + p \sigma^2 T \int_{0}^{s}e^{\lambda v}  \,dv\\
&\quad +\int_{0}^{s}e^{\lambda v} \int_{0}^{T} \mathbb E  \left[ \lambda |\theta_{v,t}|^2-2\bigg(\tfrac{\sigma^2}{2}\nabla_a U_t(\theta_{v,t})
+(\nabla_a \tfrac{\delta \mathbf H_t^{0}}{\delta m})(\mu_{v,\cdot},\theta_{v,t})\bigg) \theta_{v,t} \right]\,dt\,dv \,.\\ 
\end{align*}
With Assumptions~\ref{ass dissipativity of U} and Lemma~\ref{lem sys diss} we can obtain that
\begin{align*}
& e^{\lambda s} \int_{0}^{T}\mathbb E|\theta_{s,t}|^2\,dt
\leq \int_0^T \mathbb E |\theta_{0,t}|^2 \,dt + p\sigma^2 T \int_{0}^{s}e^{\lambda v}  \,dv\\
&\quad +\int_{0}^{s}e^{\lambda v} \int_{0}^{T} \mathbb E  \left[ (\lambda - \sigma^2\kappa_u - \kappa_f + \hat c)|\theta_{v,t}|^2 \right]\,dt\,dv + \int_0^s e^{\lambda v} \hat{c}(1+|\xi|^2)\,dv \,.
\end{align*}
Taking $\lambda = \sigma^2\kappa_u + \kappa_f - \hat c>0$ we then have 
\[
\int_{0}^{T}\mathbb E|\theta_{s,t}|^2\,dt
\leq e^{-\lambda s} \int_0^T \mathbb E |\theta_{0,t}|^2 \,dt + (p\sigma^2 T + \hat c(1+|\xi|^2)\frac1\lambda(1-e^{-\lambda s})\,.
\]
This concludes the proof.
\end{proof}

Fix $S > 0$, $I:=[0,S]$,
$
C(I; \mathcal V_2^W) := \left\{ \nu = (\nu_{s,\cdot})_{s\in I} : \nu_{s,\cdot} \in  \mathcal V_2^W\,\,\text{and}\,\,\lim_{s'\to s} \rho_2(\nu_{s',\cdot},\nu_{s,\cdot}) = 0 \,\,\,\forall s\in I \right\}$.
Consider $\mu \in C(I,\mathcal V_2^W)$. 
For each $\mu_{s,\cdot} \in \mathcal V_2^W$, $s\geq 0$ we obtain unique solution to~\eqref{eq process} and \eqref{eq adjoint proc} which we denote $(X_{s,\cdot}, Y_{s,\cdot},Z_{s,\cdot})$.
Moreover, for each $t\in [0,T]$ the SDE
\begin{equation}
\label{eq lang in proof}
d\theta_{s,t}(\mu) = - \bigg((\nabla_a \tfrac{\delta \mathbf H_t^{0}}{\delta m})(\mu_{s,t},\theta_{s,t}(\mu)) + \tfrac{\sigma^2}{2}(\nabla_a U)(\theta_{s,t}(\mu))\bigg)\,ds + \sigma\,dB_s	\,,\,\,\,s\geq 0
\end{equation}
has a unique strong solution, see e.g.~\cite[Theorem 3.1]{krylov1981stochastic}.
We denote the measure in
$\mathcal V_2^W$ induced by 
$\theta_{s,t}$ conditioned on $\mathcal F_t^W$ for each $t\in[0,T]$ as $\mathcal L(\theta(\mu)_{s,\cdot}|W)$. 
Our aim will be to obtain a contraction based on the map $\Psi$ given by $\mathcal C(I,\mathcal V_2^W) \ni \mu \mapsto  \{ \mathcal L(\theta_{s,\cdot}(\mu)\,|\,W(\omega^W)) : \omega^W \in \Omega^W, s\in I\}$. 
Before we do that we need the following estimate.

\begin{lemma}[Linearised flow monotonicity]
\label{lem sys mono}
Let Assumptions~\ref{ass dissipativity of F},~\ref{ass for bsde estimates} and \ref{ass for gradient system} hold. 
Let $\mu,\mu' \in C(I,\mathcal V_2^W)$ and let $\theta(\mu)$, $\theta(\mu')$ be the two solutions to~\eqref{eq lang in proof} arising from $\mu$ and $\mu'$.  
Then there is $L>0$ such that for any $s\geq 0$,
\[
\begin{split}
 - 2 & \mathbb E^W \int_0^T \left(\nabla_a \tfrac{\delta \mathbf H_t^{0}}{\delta m}(\mu_{s,\cdot},\theta_{s,t}(\mu)) - \nabla_a \tfrac{\delta \mathbf H_t^{0}}{\delta m}(\mu'_{s,\cdot},(\theta_{s,t}(\mu'))\right)\Big(\theta_{s,t}(\mu)-\theta_{s,t}(\mu')\Big)\,dt\\
&\leq -\kappa_f \mathbb E^W \int_0^T |\theta_{s,t}(\mu) - \theta_{s,t}(\mu')|^2\,dt + L \mathbb E^W \int_0^T |\theta_{s,t}(\mu) - \theta_{s,t}(\mu')|^2\,dt  + L\rho_2(\mu_{s,\cdot},\mu'_{s,\cdot})^2 \,.	
\end{split}
\]
The constant $L$ is independent of $\kappa_f$ and $\sigma$ but depends on $T$, $K$ and the constants arising in Lemmas~\ref{lemma sde stability},~\ref{bsde bounds} and \ref{lemma:bsde-diff-est}.	
\end{lemma}

\begin{proof}
We see that due to Assumption~\ref{ass dissipativity of F} we have
\[
\begin{split}
 - 2 & \mathbb E^W \int_0^T \left(\nabla_a \tfrac{\delta \mathbf H_t^{0}}{\delta m}(\mu_{s,\cdot},\theta_{s,t}(\mu)) - \nabla_a \tfrac{\delta \mathbf H_t^{0}}{\delta m}(\mu'_{s,\cdot},\theta_{s,t}(\mu'))\right)\Big(\theta_{s,t}(\mu)-\theta_{s,t}(\mu')\Big)\,dt\\
= & -2\mathbb E^W \int_0^T \left(\nabla_a \tfrac{\delta F^C_t}{\delta m}(X_{s,t}(\mu'), \mu'_{s,t},\theta_{s,t}(\mu)) -\nabla_a \tfrac{\delta F^C_t}{\delta m}(X_{s,t}(\mu'), \mu'_{s,t},\theta_{s,t}(\mu'))\right)\Big(\theta_{s,t}(\mu)-\theta_{s,t}(\mu')\Big)\,dt\\
& -2 \mathbb E^W \int_0^T 	(\theta_{s,t}(\mu)-\theta_{s,t}(\mu'))\Big(I^{(1)}_{s,t} + I^{(2)}_{s,t} + I^{(3)}_{s,t} + I^{(4)}_{s,t}\Big)\\
\leq & -\kappa_f \mathbb E^W \int_0^T |\theta_{s,t}(\mu)-\theta_{s,t}(\mu')|^2\,dt +  2 \mathbb E \int_0^T 	|\theta_{s,t}(\mu)-\theta_{s,t}(\mu')|\Big|I^{(1)}_{s,t} + I^{(2)}_{s,t} + I^{(3)}_{s,t} + I^{(4)}_{s,t}\Big|\,dt\,,
\end{split}
\]
where
\[
I^{(1)}_{s,t} := (\nabla_a \tfrac{\delta \Phi_t}{\delta m})(X_t(\mu_{s,\cdot}), \mu_{s,t}, \theta_{s,t}(\mu)) \,Y_t(\mu_{s,\cdot}) - (\nabla_a \tfrac{\delta \Phi_t}{\delta m})(X_t(\mu'_{s,\cdot}), \mu'_{s,t}, \theta_{s,t}(\mu')) \,Y_t(\mu'_{s,\cdot})\,, 
\]
\[
I^{(2)}_{s,t} := \text{tr}\bigg[(\nabla_a \tfrac{\delta \Gamma_t}{\delta m})(X_t(\mu_{s,\cdot}), \mu_{s,t}, \theta_{s,t}(\mu)) \,Z_t(\mu_{s,\cdot}) - (\nabla_a \tfrac{\delta \Gamma_t}{\delta m})(X_t(\mu'_{s,\cdot}), \mu'_{s,t}, \theta_{s,t}(\mu')) \,Z_t(\mu'_{s,\cdot})\bigg]\,, 
\]
\[
I^{(3)}_{s,t} :=(\nabla_a \tfrac{\delta F^L_t}{\delta m})(X_t(\mu_{s,\cdot}), \mu_{s,t}, \theta_{s,t}(\mu))  - (\nabla_a \tfrac{\delta F^L_t}{\delta m})(X_t(\mu'_{s,\cdot}), \mu'_{s,t}, \theta_{s,t}(\mu'))\,, 
\]
and
\[
I^{(4)}_{s,t} :=(\nabla_a \tfrac{\delta F^C_t}{\delta m})(X_t(\mu_{s,\cdot}), \mu_{s,t}, \theta_{s,t}(\mu))  - (\nabla_a \tfrac{\delta F^C_t}{\delta m})(X_t(\mu'_{s,\cdot}), \mu'_{s,t}, \theta_{s,t}(\mu))\,. 
\]
Assumption~\ref{ass for gradient system} ii) and Lemma~\ref{bsde bounds} yield
\[
\begin{split}
& \mathbb E^W \int_0^T |\theta_{s,t}(\mu) - \theta_{s,t}(\mu')| |I^{(1)}_{s,t} |\,dt 
\leq  C \mathbb E^W \int_0^T \bigg[ |\theta_{s,t}(\mu) - \theta_{s,t}(\mu')| |X_t(\mu_{s,\cdot}) - X_t(\mu'_{s,\cdot})| \\
& \qquad \qquad\qquad\qquad\qquad\qquad + |\theta_{s,t}(\mu) - \theta_{s,t}(\mu')|\mathcal W_2(\mu_{s,t},\mu'_{s,t}) + |\theta_{s,t}(\mu) - \theta_{s,t}(\mu')|^2
\bigg]\,dt \\
& \leq C \mathbb E^W \int_0^T \bigg[ |\theta_{s,t}(\mu) - \theta_{s,t}(\mu')|^2+ |X_t(\mu_{s,\cdot}) - X_t(\mu'_{s,\cdot})|^2 
  + \mathcal W_2(\mu_{s,t},\mu'_{s,t})^2\
\bigg]\,dt\,.   	
\end{split}
\]
With Assumption~\ref{ass for gradient system} i) and Lemma~\ref{lemma sde stability} we can observe that
\[
\mathbb E^W \int_0^T |\theta_{s,t}(\mu) - \theta_{s,t}(\mu')| |I^{(1)}_{s,t} |\,dt \leq C^{(1)} \mathbb E^W \int_0^T |\theta_{s,t}(\mu) - \theta_{s,t}(\mu')|^2 \,dt + C^{(1)}\rho_2(\mu_{s,\cdot},\mu'_{s,\cdot})^2 \,.
\]
Due to Assumption~\ref{ass for gradient system} and Lemma~\ref{lemma:bsde-diff-est} we have
\[
\mathbb E^W \int_0^T |\theta_{s,t}(\mu) - \theta_{s,t}(\mu')| |I^{(2)}_{s,t} |\,dt \leq C^{(2)}\mathbb E^W \int_0^T |\theta_{s,t}(\mu) - \theta_{s,t}(\mu')|^2 \,dt  + C^{(2)}\rho_2(\mu_{s,\cdot},\mu'_{s,\cdot})^2\,.
\]
Assumption~\ref{ass for gradient system} ii) and Lemma~\ref{lemma sde stability} allow us to conclude that
\[
\mathbb E^W \int_0^T |\theta_{s,t}(\mu) - \theta_{s,t}(\mu')| |I^{(3)}_{s,t} + I^{(4)}_{s,t} |\,dt \leq C^{(3)} \mathbb E^W \int_0^T |\theta_{s,t}(\mu) - \theta_{s,t}(\mu')|^2 \,dt + C^{(3)}\rho_2(\mu_{s,\cdot},\mu'_{s,\cdot})^2\,.
\]
Letting $L:=C^{(1)}+C^{(2)}+C^{(3)}$ completes the proof.
\end{proof}

\begin{proof}[Proof of Lemma~\ref{thm:WellposednessMeanField}]

{\bf Step 1.} 
We need to show that $ \{ \mathcal L(\theta(\mu)_{s,\cdot}|W(\omega^W)) : \omega^W \in \Omega^W, s\in I\} \in \mathcal C(I, \mathcal V_2^W)$. 
This amounts to showing that we have the appropriate integrability and continuity.
Integrability follows from the same argument as in the proof Lemma~\ref{lem apriori} with $\theta_{s,t}$ replaced by $\theta_{s,t}(\mu)$.
To establish the continuity property note that for $s'\geq s$ we have 
\[
\theta_{s',t}(\mu) - \theta_{s,t}(\mu) = -\int_s^{s'} \left( \nabla_a \tfrac{\delta \mathbf H^0}{\delta m}(\mu_{r,\cdot}, \theta_{r,t}(\mu)) - \tfrac{\sigma^2}{2}\nabla_a U(\theta_{r,t}(\mu))\right)\,dr + \sigma (B_{s'} - B_s)\,.
\]
From Lemma~\ref{lem apriori} we know that there is $c>0$ independent of $r\in [s',s]$ such that 
$
\mathbb E^W \int_0^T |\theta_{r,t}(\mu)|^2\,dt < c 
$ 
and hence $\mathbb E\int_s^{s'} \int_0^T (\theta_{r',t}(\mu) - \theta_{s,t}(\mu))\,dB_r = 0$. 
Thus using It\^o's formula and integrating we get
\[
\begin{split}
& \mathbb E\int_0^T |\theta_{s',t}(\mu) - \theta_{s,t}(\mu)|^2\,dt \\ 
& = \int_s^{s'} \mathbb E \int_0^T (\theta_{r,t}(\mu) - \theta_{s,t}(\mu)) \left(\nabla_a \tfrac{\delta \mathbf H^0}{\delta m}(\mu_{r,\cdot},\theta_{r,t}(\mu)) - \tfrac{\sigma^2}{2}\nabla_a U(\theta_{r,t}(\mu))\right)\,dt\,dr + T\sigma^2 \int_s^{s'}\,dr\\
& = T\sigma^2 (s'-s) + \int_s^{s'} \mathbb E \int_0^T \theta_{r,t}(\mu)\nabla_a \tfrac{\delta \mathbf H^0}{\delta m}(\mu_{r,\cdot},\theta_{r,t}(\mu)) \,dt\,dr - \int_s^{s'} \mathbb E \int_0^T  \theta_{s,t}(\mu) \nabla_a \tfrac{\delta \mathbf H^0}{\delta m}(\mu_{r,\cdot},\theta_{r,t}(\mu)) \,dt\,dr 	\\
& \qquad - \int_s^{s'} \mathbb E \int_0^T \theta_{r,t}(\mu) \tfrac{\sigma^2}{2}\nabla_a U(\theta_{r,t}(\mu))\,dt\,dr + \int_s^{s'} \mathbb E \int_0^T \theta_{s,t}(\mu) \tfrac{\sigma^2}{2}\nabla_a U(\theta_{r,t}(\mu))\,dt\,dr\,. 	
\end{split}
\]
Same estimates as for the proof of Lemma~\ref{lem sys diss} yield that there is a constant $c > 0$ that's independent of $r\in [s,s']$ such that
$
\mathbb E\int_0^T |\theta_{s',t}(\mu) - \theta_{s,t}(\mu)|^2\,dt \leq T\sigma^2 (s'-s) + \int_s^{s'} c \,dr \leq c (s'-s)\,.
$
This establishes the required continuity property.

{\bf Step 2.}
Let $\lambda\in \mathbb R$ to be fixed later. 
Fix $t\in[0,T]$. 
We will integrate over $t$ in a moment.
It\^o's formula together with Assumption~\ref{ass dissipativity of U} yields
\[
\begin{split}
& d\Big(e^{\lambda s}  |\theta_{s,t}(\mu) - \theta_{s,t}(\mu')|^2  \Big) \\
& =  \,  e^{\lambda s} 
\biggl[ \lambda |\theta_{s,t}(\mu) - \theta_{s,t}(\mu')|	^2  - \sigma^2(\theta_{s,t}(\mu) - \theta_{s,t}(\mu'))  \Big( (\nabla_a U)(\theta_{s,t}(\mu)) -  (\nabla_a U)(\theta_{s,t}(\mu')) \Big) \\
& \qquad \qquad
- 2 (\theta_{s,t}(\mu) - \theta_{s,t}(\mu'))\bigg( (\nabla_a \tfrac{\delta \mathbf H_t^{0}}{\delta m})(\mu_{s,\cdot},\theta_{s,t}(\mu)) - (\nabla_a \tfrac{\delta \mathbf H_t^{0}}{\delta m})(\mu'_{s,\cdot},\theta_{s,t}(\mu))\bigg)\bigg] ds \\ 
& \leq  \,  e^{\lambda s} 
\biggl[ (\lambda-\sigma^2\kappa_U) |\theta_{s,t}(\mu) - \theta_{s,t}(\mu')|	^2   \\
& \qquad \qquad
- 2 (\theta_{s,t}(\mu) - \theta_{s,t}(\mu'))\bigg( (\nabla_a \tfrac{\delta \mathbf H_t^{0}}{\delta m})(\mu_{s,\cdot}, \theta_{s,t}(\mu)) - (\nabla_a \tfrac{\delta \mathbf H_t^{0}}{\delta m})(\mu'_{s,\cdot}, \theta_{s,t}(\mu'))\bigg)\bigg] ds \,.
\end{split}
\]
Notice that the difference has no martingale term since the same process $(B_s)_{s\in I}$ is used regardless of the initial condition since it appears in~\eqref{eq lang in proof} as an additive term.
Integrating over $[0,T]\times \Omega^W$ we get 
\[
\begin{split}
& d\bigg(e^{\lambda s} \, \mathbb E^W\int_0^T  |\theta_{s,t}(\mu) - \theta_{s,t}(\mu')|^2\,dt  \bigg) 
\leq  e^{\lambda s} \mathbb E^W\int_0^T  \bigg[ (\lambda-\sigma^2\kappa_U) |\theta_{s,t}(\mu) - \theta_{s,t}(\mu')|	^2\,dt\\
& -  e^{\lambda s} 2\mathbb E^W \int_0^T \left(\nabla_a \tfrac{\delta \mathbf H_t^{0}}{\delta m}(\mu_{s,\cdot}, \theta_{s,t}(\mu)) - \nabla_a \tfrac{\delta \mathbf H_t^{0}}{\delta m}(\mu'_{s,\cdot}, \theta_{s,t}(\mu'))\right)\Big(\theta_{s,t}(\mu)-\theta_{s,t}(\mu')\Big)\,dt\,ds\,.
\end{split}
\]
Lemma~\ref{lem sys mono} leads us to
\begin{equation}
\label{eq for wellposednessMeanField}
\begin{split}
& d\bigg(e^{\lambda s} \, \mathbb E^W\int_0^T  |\theta_{s,t}(\mu) - \theta_{s,t}(\mu')|^2\,dt  \bigg) \\
& \leq  e^{\lambda s} \mathbb E^W\int_0^T (\lambda-\sigma^2\kappa_U - \kappa_f +L) |\theta_{s,t}(\mu) - \theta_{s,t}(\mu')|	^2\,dt\,ds
+ L e^{\lambda s} \rho_2(\mu_s,\mu'_s)^2\,ds\,.
\end{split}
\end{equation}
Taking $\lambda := \sigma^2\kappa_U + \kappa_F - L \geq 0$, using the definition of 2-Wasserstein distance and integrating over from $0$ to $s$ find that
\[
e^{\lambda s} \rho_2(\Psi(\mu)_s,\Psi(\mu')_s)^2 \leq L \int_0^s  e^{\lambda v} \rho_2(\mu_{v},\mu'_{v})^2\,dv\,,\,\,\,s\geq 0\,.
\]
{\bf Step 3.} 
Thus
\[
\rho_2(\Psi(\mu)_s,\Psi(\mu')_s)^2 \leq L \int_0^s  \rho_2(\mu_{v},\mu'_{v})^2\,dv\,,\,\,\,s\geq 0\,.
\]
So for any $s \geq 0$ we see that
\[
\rho_2(\Psi(\Psi(\mu))_s,\Psi(\Psi(\mu'))_s)^2 
\leq L \int_0^{s}  \rho_2(\Psi(\mu)_{s_1},\Psi(\mu')_{s_1})^2\,d{s_1}
\leq L^2 \int_0^{s} \int_0^{s_1} \rho_2(\mu_{s_2},\mu'_{s_2})^2\,ds_2\,ds_1\,.
\]
Let $\Psi^k$ denote the $k$-th composition of the mapping $\Psi$ with itself. 
Then for any $s\geq 0$, 
\[
\begin{split}
\rho_2(\Psi^k(\mu)_s,\Psi^k(\mu')_s)^2 
& \leq L^k \int_0^{s} \int_0^{s_1} \cdots\int_0^{k-1} \rho_2(\mu_{s_k},\mu'_{s_k})^2\,ds_k\ldots ds_2\,ds_1\\
& \leq L^k \sup_{0\leq v\leq s}\rho_2(\mu_v,\mu'_v)^2 \int_0^s\int_0^{s_1} \cdots\int_0^{k-1} ds_k\ldots ds_2\,ds_1
\,.	
\end{split}
\]
Hence
\[
\rho_2(\Psi^k(\mu)_s,\Psi^k(\mu')_s)^2 \leq L^k \frac{s^k}{k!}\sup_{0\leq v\leq s}\rho_2(\mu_v,\mu'_v)^2 \,,\,\,\,s\geq 0\,.
\]
So 
\[
\sup_{0\leq v\leq s} \rho_2(\Psi^k(\mu)_v,\Psi^k(\mu')_v) \leq \left(\frac{L^k s^k}{k!}\right)^{1/2}\sup_{0\leq v\leq s}\rho_2(\mu_v,\mu'_v) \,,\,\,\,s\geq 0\,.
\]
Thus for any $s\in I$ there is $k\geq 1$ such that $\Psi^k$ is a contraction. 

{\bf Step 4.} Since $\Psi^k$ is a contraction we get existence and uniqueness of solution to the system from Banach's fixed point theorem. 
The estimate~\eqref{eq:UnifBound} follows from Lemma~\ref{lem apriori}. 
\end{proof}

Recall that $P_{s,t} \mu^0 := \mathcal L\big( \theta_{s,t} | \mathcal F_t^W \big)$,
where $(\theta_{s,t})_{s\geq 0, t\in [0,T]}$ is the unique solution (cf. Lemma~\ref{thm:WellposednessMeanField}) to the  system~\eqref{eq mfsgd}-\eqref{eq mfsgd 2} started with the initial condition $(\theta^0_t)_{t\in [0,T]}$ such that $\mathcal L(\theta^0_t(\omega^W)) = \mathbb P^B \circ \theta^0_t(\omega^W)^{-1} = \mu^0(\omega^W)$ for a.e. $\omega^W \in \Omega^W$. 
Let $P_s \mu^0 := (P_{s,t} \mu^0)_{t\in [0,T]}$ and note that $P_s \mu^0 \in \mathcal V_{q}^W$ for any $s\geq 0$. 
Moreover note that due to uniqueness of solutions to~\eqref{eq mfsgd}-\eqref{eq mfsgd 2} we have $P_{s+s'} \mu^0 = P_s\big(P_{s'}\mu^0\big)$.

\begin{lemma}
\label{lemma exp conv est}
Let Assumptions~\ref{ass dissipativity of U}, \ref{ass dissipativity of F}, \ref{ass for bsde estimates} and~\ref{ass for gradient system} hold.
If $\lambda = 2L - \sigma^2\kappa_U - \kappa_f \geq 0$ and if $\mu^0, \bar \mu^0 \in \mathcal V_{q}^W$, then for all $s\geq 0$ we have
\begin{equation}
\label{eq inv meas conv rate}
\rho_2(P_s \mu^0, P_s \bar \mu^0)\leq e^{-\frac12\lambda s}  \rho_2(\mu^0, \bar \mu^0)\,.
\end{equation}
\end{lemma}

\begin{proof}
Let $\varepsilon > 0$ be fixed. 
We claim that there are $(\theta^0_t)_{t\in [0,T]}$ and $(\bar \theta^0_t)_{t\in [0,T]}$ such that $\mathcal L(\theta^0_t(\omega^W)) = \mu^0(\omega^W)$
and 
$(\bar \theta^0_t)_{t\in [0,T]}$
such that $\mathcal L(\bar \theta^0_t(\omega^W)) = \bar \mu^0(\omega^W)$
for a.e. $\omega^W \in \Omega^W$ and such that $\mathbb E^W\int_0^T  |\theta^0_{t} - \bar\theta^0_{t}|^2\,dt - \varepsilon \leq \rho_2(\mu^0, \bar \mu^0)$. 
This follows from the definition of Wasserstein distance and from the definition of $\rho_2$, see~\eqref{eq def of rho}.
Let $(\theta_{s,\cdot})_{s\geq 0}$ and $( \bar\theta_{s,\cdot})_{s\geq 0}$
denote two solutions to~\eqref{eq mfsgd}-\eqref{eq mfsgd 2} with initial conditions 
$(\theta^0_t)_{t\in [0,T]}$ and $(\bar \theta^0_t)_{t\in [0,T]}$. 
For $\lambda = 2L - \sigma^2\kappa_U - \kappa_f \geq 0$, the same computation as in the  Step 2 of the proof of Lemma~\ref{thm:WellposednessMeanField}, see~\eqref{eq for wellposednessMeanField}, gives for all $s\geq 0$ that,
\begin{equation*}
\begin{split}
& e^{\lambda s} \, \mathbb E^W\int_0^T  |\theta_{s,t} - \bar\theta_{s,t}|^2\,dt  \leq  \mathbb E^W\int_0^T  |\theta_{0,t} - \bar\theta_{0,t}|^2\,dt \\
&  +   \int_0^s e^{\lambda r} \mathbb E^W\int_0^T (\lambda-\sigma^2\kappa_U - \kappa_f +L) |\theta_{r,t} - \bar\theta_{r,t}|	^2\,dt\,dr
+ L \int_0^s e^{\lambda r} \rho_2(P_r\mu^0,P_r \bar \mu^0)^2\,dr\,.
\end{split}
\end{equation*}
From the properties of Wasserstein norm and the fact that $\lambda = 2L - \sigma^2\kappa_U - \kappa_f \geq 0$ we get for all $s\geq 0$ that
\[
 \mathbb E^W\int_0^T  |\theta_{s,t} - \bar\theta_{s,t}|^2\,dt  \leq  e^{-\lambda s} \mathbb E^W\int_0^T  |\theta_{0,t} - \bar\theta_{0,t}|^2\,dt\,.
\]
Hence 
\[
\rho_2(P_s \mu^0, P_s \bar \mu^0)^2 \leq e^{-\lambda s}\varepsilon + e^{-\lambda s}\rho_2(\mu^0, \bar \mu^0)^2\,,\,\,\, s\geq 0\,.
\]
Since $\varepsilon >0$ was arbitrary we get the desired conclusion. 
\end{proof}

\begin{proof}[Proof of Theorem~\ref{thm conv to inv meas rate}]
We follow a similar argument as in the proof of \cite[Corollary 1.8]{majka2017coupling} or in \cite[Section 3]{komorowski2012central}.
Choose $s_0 > 0$ such that $e^{-\frac1q \lambda s_0} < 1$. 
Then $P_{s_0} : \mathcal V_q^W \to \mathcal V_q^W$ 
is a contraction due to Lemma~\ref{lemma exp conv est}. 
By Banach's fixed point theorem and Lemma~\ref{lemma L2overmeasurescomplete} there is a (unique) $\tilde \mu \in \mathcal V_q^W$ such that $P_{s_0}\tilde \mu = \tilde \mu$.
Note that $\tilde \mu$ depends on the above choice of $s_0$.

Let $\mu^\ast := \int_0^{s_0} P_s\tilde \mu \,ds$.
Now take an arbitrary $r \leq s_0$. 
Then
\[
P_r \mu^\ast = P_r \int_0^{s_0} P_s\tilde \mu\,ds = \int_0^{s_0} P_{r+s} \tilde \mu \,ds 	= \int_r^{r+s_0} P_s\tilde \mu\,ds 
= \int_r^{s_0} P_s\tilde \mu\,ds + \int_{s_0}^{r+s_0} P_s\tilde \mu\,ds\,. 
\]
Since $\tilde \mu = P_{s_0}\tilde \mu$ we get that
\[
P_r \mu^\ast = \int_r^{s_0} P_s\tilde \mu\,ds + \int_{s_0}^{r+s_0} P_{s+s_0}\tilde \mu\,ds = \int_r^{s_0} P_s\tilde \mu\,ds + \int_{0}^{r} P_{s}\tilde \mu\,ds = \mu^\ast\,.
\]
For $r > s_0$ choose $k\in \mathbb N$ so that $ks_0 < r$ and $(k+1)s_0 \geq r$. 
Then $P_r \mu^\ast = (P_{s_0})^k P_{(r-ks_0)} \mu^\ast = \mu^\ast$.
To see that $\mu^\ast$ is unique consider $\nu^\ast \neq \mu^\ast$ such that $P_r \nu^\ast = \nu^\ast$ for any $r\geq 0$. Then from Lemma~\ref{lemma exp conv est} we have, for any $r>s_0$, that
\[
\rho_q(\mu^\ast, \nu^\ast) = \rho_q(P_r \mu^\ast, P_r \nu^\ast) \leq e^{-\frac1q\lambda r} \rho_q(\mu^\ast,\nu^\ast)
\] 
which is a contradiction as $e^{-\frac1q \lambda r} < 1$.
From Lemma~\ref{lemma exp conv est} we immediately get~\eqref{eq exp conv to inv meas}.
\end{proof}

\section{BSDE estimates}

Let $\mathbb E_t^W[\cdot] := \mathbb E[\cdot | \mathcal F_t^W]$.
Let $\|X\|_\infty := \text{ess sup}_\omega^W |X(\omega^W)|$ 
and $\|Z\|_{\mathbb H^\infty} := \text{ess sup}_{t,\omega^W}|Z_t(\omega^W)|$ for any r.v. $X$ and for any progressively measurable process $Z$ on $(\Omega^W, \mathcal F^W, (\mathcal F^W_t)_{t\in [0,T]}, \mathbb P^W)$. 
For a uniformly integrable martingale $M$ such that $M_0 = 0$ let $\|M\|_{\text{BMO}} := \sup_\tau \|(\mathbb E_\tau^W[\langle M\rangle_T - \langle M \rangle_\tau ] )^{1/2} \|_\infty$,
where the supremum is taken over all $(\mathcal F_t^W)_{t\in [0,T]}$ stopping times.
For a process $Z$ adapted to $(\mathcal F_t^W)_{t\in [0,T]}$ let $(Z\cdot W)_{t'} := \int_0^{t'} Z_t\,dW_t$.
We start by recalling a well known lemma about controlled SDEs with Lipschitz coefficients.
\begin{lemma}
\label{lemma sde stability}
Assume there exists $K>0$ such that for all $t\in[0,T]$, for all $x,x' \in \mathbb R^d$ and for all $m,m'\in \mathcal P_2(\mathbb R^p)$ we have
\[
|\Phi_t(x,m) - \Phi_t(x',m')| + |\Gamma_t(x,m) - \Gamma_t(x',m')| \leq K |x-x'| + \mathcal W_2(m,m')\,.
\]
Then there is $C=C_{T,K}$ such that for any $\mu,\mu' \in \mathcal V_2^W$ we have 
\[
\mathbb E^W \sup_{0\leq t \leq T} |X_t(\mu)-X_t(\mu')|^2 \leq C \rho_2(\mu,\mu')^2\,. 
\]
\end{lemma}

The following lemma is proved in~\cite{kerimkulov2021regularized} which uses results proved in~\cite{harter2019stability}.
It is the first key stability result needed to show that the gradient flow system~\eqref{eq mfsgd}-\eqref{eq mfsgd 2} has a solution and that it converges to invariant measure.

\begin{lemma}
\label{bsde bounds}
Assume there exists $K>0$ such that for all $x\in\mathbb R^d$, all $m\in \mathcal P_2(\mathbb R^p)$ and all $t\in [0,T]$ we have
$
|\nabla_x \Phi_t(x,m)| + \sum_{i=1}^d\sum_{j=1}^{d'}|\nabla_x \Gamma_t^{ij}(x,m)| + |\nabla_x g(x)| + |\nabla_x F_t(x,m)|\leq K\,.
$
Then 
\[
\sup_{\mu \in \mathcal V_q^W} \|Y(\mu)\|_{\mathbb H^\infty} < \infty \,\,\,\text{and}\,\,\,
\sup_{\mu \in \mathcal V_q^W}\|Z(\mu)\cdot W\|_{\text{BMO}} < \infty\,.
\]
\end{lemma}

To proceed we need to recall the following deep result about BMO martingales that is proved in~\cite[Theorem 2.5]{Kazamaki}.
\begin{theorem}
\label{thm fefferman}
If $M\in \text{BMO}(\Omega^W)$ and $N$ is an $(\mathcal F_t^W)_{t\in[0,T]}$-martingale such that $\sup_{0\leq t\leq T}|N_t| \in L_1(\Omega^W)$, then
$\mathbb E^W \left[\int_0^T |d\langle M, N\rangle_t| \right] \leq \sqrt{2} \|M\|_{\text{BMO}}\mathbb E\left[\langle N \rangle_T^{1/2}\right]$.	
\end{theorem}

We will mainly use Theorem~\ref{thm fefferman} to give us the following corollary. 

\begin{corollary}
\label{corollary to fefferman}
Let $\varphi, \psi$ be progressively measurable such that $\varphi, \psi \in L^2((0,T)\times \Omega^W; \mathbb R^k)$ and $\varphi \cdot W\in \text{BMO}(\Omega^W)$. 
Then $\mathbb E^W \int_0^T |\varphi_t||\psi_t|\,dt \leq \sqrt{2}\|\varphi \cdot W\|_{\text{BMO}} \left(\mathbb E^W\int_0^T |\psi_t|^2\,dt\right)^{1/2}$.  	
\end{corollary}
\begin{proof}
Observe that $\mathbb E^W \int_0^T |\varphi_t||\psi_t|\,dt = \mathbb E^W\int_0^T |d\langle |\varphi| \cdot W, |\psi| \cdot W\rangle_t|$. 
Using Theorem~\ref{thm fefferman} and H\"older's inequality we get
\[
\begin{split}
& \mathbb E^W\int_0^T |d\langle |\varphi| \cdot W, |\psi| \cdot W\rangle_t| \\
& \leq \sqrt{2} \|\varphi\cdot W\|_{\text{BMO}} \mathbb E^W \bigg[\bigg(\int_0^T |\psi_t|^2\,dt\bigg)^{1/2}\bigg] \leq \sqrt{2}\|\varphi \cdot W\|_{\text{BMO}} \bigg(\mathbb E^W\int_0^T |\psi_t|^2\,dt\bigg)^{1/2}\,.	
\end{split}
\]
\end{proof}

We now will now state and prove the second key stability result on the backward equation that is needed to show that the gradient flow system~\eqref{eq mfsgd}-\eqref{eq mfsgd 2} has a solution and that it converges to invariant measure.

\begin{lemma}
\label{lemma:bsde-diff-est}
Assume that there is $K>0$ such that for all $t\in[0,T$], for all $x,x' \in \mathbb R^d$ and for all $m,m' \in \mathcal P_2(\mathbb R^p)$ we have:
\begin{enumerate}[i)]
\item 
$
|\Phi_t(x,m) - \Phi_t(x',m')| + |\Gamma_t(x,m) - \Gamma_t(x',m')| \leq K |x-x'| + \mathcal W_2(m,m')\,.
$
\item 
$
|\nabla_x \Phi_t(x,m)| + \sum_{i=1}^d\sum_{j=1}^{d'}|\nabla_x \Gamma_t^{ij}(x,m)| + |\nabla_x g(x)| + |\nabla_x F_t(x,m)|\leq K\,.
$
\item
$
|\nabla_x \Phi_t(x,m) - \nabla_x \Phi_t(x',m')| \leq K |x-x'| + \mathcal W_2(m,m')$ and $|\nabla_x \Gamma_t(x,m) - \nabla_x \Gamma_t(x',m')| \leq K |x-x'|\,.
$
\item 
$
|\nabla_x F_t(x,m) - \nabla_x F_t(x',m')| \leq K |x-x'| + \mathcal W_2(m,m')
\,\,\,\text{and}\,\,\,
|\nabla_x g(x) - \nabla_x g(x')| \leq K|x-x'|\,.
$
\end{enumerate}	
Then there is a constant $C=C_{d,d',K,T}>0$ such that if $\mu,\mu' \in \mathcal V_2^W$ then 
\[
\mathbb E^W \sup_{0\leq t\leq T} |Y_t(\mu) - Y_t(\mu')|^{2} + \mathbb E^W \int_0^T |Z_t(\mu) - Z_t(\mu')|^2\,dt\leq C \rho_{2}(\mu,\mu')^{2}\,.
\] 
\end{lemma}

\begin{proof}
Given a constant $\beta\geq 1$ to be fixed later, we compute by It\^o's formula that:
\begin{equation}
\label{eq diff of Y and Z ito}	
\begin{split}
& e^{\beta t'}  |Y_{t'}(\mu) - Y_{t'}(\mu')|^2  =  e^{\beta T}|\left(\nabla_x g\right)(X_T(\pi))-\left(\nabla_x g\right)(X_T(\pi'))|^2 \\
& + 2\int_{t'}^T e^{\beta t} \big(Y_t(\mu)-Y_t(\mu')\big)^\top \big(Z_t(\mu)-Z_t(\mu')\big)\,dW_t \\ 
& +\int_{t'}^T e^{\beta t}\bigg[2\big(Y_t(\mu)-Y_t(\mu')\big)\Big(-\left(\nabla_x H^{0}_t\right)(X_t(\mu),Y_t(\mu),Z_t(\mu), \mu_t)+\left(\nabla_x H^{0}_t\right)(X_t(\mu'), Y_t(\mu'), Z_t(\mu'), \mu'_t)\Big)\\
&\qquad\qquad-|Z_t(\mu)-Z_t(\mu')|^2-\beta|Y_t(\mu)-Y_t(\mu')|^{q'}\bigg]\,dt\,.
\end{split}
\end{equation}
By taking expectation and using the fact that $Z(\mu),Z(\mu')\in L^2((0,T)\times \Omega^W)$ and Lemma~\ref{bsde bounds} we get
\begin{equation}\label{eq diff of Y and Z}
\begin{split}
&   \mathbb E^W\int_{0}^T e^{\beta t} \Big(\beta |Y_t(\mu)-Y_t(\mu')|^2 + |Z_t(\mu)-Z_t(\mu')|^2\Big)\,dt 
\leq  e^{\beta T} \mathbb E^W|\left(\nabla_x g\right)(X_T(\mu))-\left(\nabla_x g\right)(X_T(\mu'))|^2  \\
&  + \mathbb E^W\int_0^T 2e^{\beta t}\big(Y_t(\mu)-Y_t(\mu')\big)\bigg(-\left(\nabla_x H^{0}_t\right)(X_t(\mu),Y_t(\mu),Z_t(\mu), \mu_t) \\
&  \qquad \qquad \qquad \qquad \qquad \qquad \qquad \qquad \qquad  
  +\left(\nabla_x H^{0}_t\right)(X_t(\mu'), Y_t(\mu'), Z_t(\mu'), \mu'_t)\bigg)\,dt\,.
\end{split}
\end{equation}
We now wish to decompose the term involving the Hamiltonian as follows
\begin{equation*}
\begin{split}
- &\left(\nabla_x H^{0}_t\right)(X_t(\mu), Y_t(\mu), Z_t(\mu), \mu_t)+\left(\nabla_x H^{0}_t\right)(X_t(\mu'), Y_t(\mu'), Z_t(\mu'), \mu'_t)\\
= &-\left(\nabla_x \Phi_t\right)(X_t(\mu), \mu_t)\Big(Y_t(\mu)-Y_t(\mu')\Big)+Y_t(\mu')\Big(\left(\nabla_x \Phi_t\right)(X_t(\mu'), \mu'_t)-\left(\nabla_x \Phi_t\right)(X_t(\mu), \mu_t)\Big)\\
&-\left(\nabla_x \Gamma_t\right)(X_t(\mu), \mu_t)\Big(Z_t(\mu)-Z_t(\mu')\Big)+Z_t(\mu')\Big(\left(\nabla_x \Gamma_t\right)(X_t(\mu'), \mu'_t)-\left(\nabla_x \Gamma_t\right)(X_t(\mu), \mu_t)\Big)\\
&-\left(\nabla_x F_t\right)(X_t(\mu), \mu_t)+\left(\nabla_x F_t\right)(X_t(\mu'), \mu'_t)\,.
\end{split}
\end{equation*}
Carrying this through we obtain
\[
\begin{split}
&\mathbb E^W\int_0^T e^{\beta t}\big(Y_t(\mu)-Y_t(\mu')\big)\Big(\left(\nabla_x H^{0}_t\right)(X_t(\mu'), Y_t(\mu'), Z_t(\mu'), \mu'_t)-\left(\nabla_x H^{0}_t\right)(X_t(\mu),Y_t(\mu),Z_t(\mu), \mu_t)\Big)\,dt\\
& \leq I_1 + I_2 + I_3 + I_4 + I_5\,,
\end{split}
\]
where, estimating the first term using assumption ii),
\[
\begin{split}
I_1 & := \mathbb E^W\int_0^T e^{\beta t}|Y_t(\mu)-Y_t(\mu')|^2|\nabla_x \Phi_t(X_t(\mu),\mu_t)|\,dt \leq K \mathbb E^W\int_0^T e^{\beta t}|Y_t(\mu)-Y_t(\mu')|^2\,dt\,,\\
I_2 & := \mathbb E^W\int_0^T e^{\beta t}|Y_t(\mu)-Y_t(\mu')||Z_t(\mu)-Z_t(\mu')||\nabla_x \Gamma_t(X_t(\mu),\mu_t)|\,dt\,,\\
I_3 & := \mathbb E^W\int_0^T e^{\beta t}|Y_t(\mu)-Y_t(\mu')||\nabla_x F_t(X_t(\mu),\mu_t) - \nabla_x F_t(X_t(\mu'),\mu'_t)|\,dt\,,\\
I_4 &:= \mathbb E^W\int_0^T e^{\beta t}|Y_t(\mu)-Y_t(\mu')||Z_t(\mu')||\nabla_x \Gamma_t(X_t(\mu'),\mu'_t) - \nabla_x \Gamma_t(X_t(\mu),\mu_t)|\,dt\,,\\
I_5 &:= \mathbb E^W\int_0^T e^{\beta t}|Y_t(\mu)-Y_t(\mu')||Y_t(\mu')||\nabla_x \Phi_t(X_t(\mu'),\mu'_t) - \nabla_x \Phi_t(X_t(\mu),\mu_t)|\,dt\,.
\end{split}
\]
Using assumption ii) and Young's inequality
\[
I_2 \leq 4K\mathbb E^W\int_0^T e^{\beta t} |Y_t(\mu)-Y_t(\mu')|^2\,dt + \frac14\mathbb E^W\int_0^T e^{\beta t} |Z_t(\mu)-Z_t(\mu')|^2\,dt\,.
\]
Next we note that due to assumption iii) and iv) and due to Lemma~\ref{lemma sde stability} we obtain, with $G$ standing in for any of $\Phi$, $\Gamma$ or $F$, that
\begin{equation}
\label{eq bsde stability proof 1}
\begin{split}
& \mathbb E^W\int_0^T e^{\beta t} |\nabla_x G_t(X_t(\mu),\mu_t) - \nabla_x G_t(X_t(\mu'), \mu'_t)|^2\,dt \\	
& \leq C_{K} \mathbb E^W\int_0^T e^{\beta t} |X_t(\mu) - X_t(\mu')|^2\,dt + C_{K} \mathbb E^W\int_0^T e^{\beta t} \mathcal W_{2}(\mu_t,\mu'_t)^2\,dt \leq C \rho_2(\mu,\mu')^2\,. 
\end{split}
\end{equation}
With this in mind and using Young's inequality we get 
\[
\begin{split}
I_3 & \leq \frac12\mathbb E^W\int_0^T e^{\beta t}|Y_t(\mu)-Y_t(\mu')|^2\,dt 
+ \frac12 \mathbb E^W\int_0^T e^{\beta t}|\nabla_x F_t(X_t(\mu),\mu_t) - \nabla_x F_t(X_t(\mu'),\mu'_t)|^2\,dt\\
& \leq \frac12\mathbb E^W\int_0^T e^{\beta t}|Y_t(\mu)-Y_t(\mu')|^2\,dt + C e^{\beta T} \rho_2(\mu,\mu')^2\,.
\end{split}
\]
To estimate the next term we will employ Theorem~\ref{thm fefferman} as follows:
\[
\begin{split}
I_4 & = \mathbb E^W\int_{0}^T \bigg| d\Big\langle |Z(\mu')|\cdot W, e^{\beta t}|Y_t(\mu)-Y_t(\mu')||\nabla_x \Gamma_t(X_t(\mu'),\mu'_t) - \nabla_x \Gamma_t(X_t(\mu),\mu_t)| \cdot W \Big\rangle_t\bigg|\\
& \leq \sqrt{2}\|Z(\mu')\cdot W\|_{\text{BMO}} \mathbb E^W\Big[\Big\langle e^{\beta t}|Y_t(\mu)-Y_t(\mu')||\nabla_x \Gamma_t(X_t(\mu'),\mu'_t) - \nabla_x \Gamma_t(X_t(\mu),\mu_t)| \cdot W  \Big\rangle_T^{1/2} \Big]	\,.
\end{split}
\]
Lemma~\ref{bsde bounds} yields
\[
\begin{split}
I_4 & \leq C \mathbb E^W\bigg[  \bigg(\int_0^T e^{2\beta t} |Y_t(\mu)-Y_t(\mu')|^2|\nabla_x \Gamma_t(X_t(\mu'),\mu'_t) - \nabla_x \Gamma_t(X_t(\mu),\mu_t)|^2\,dt\bigg)^{1/2} \bigg]\\
& \leq C \mathbb E^W\bigg[ \sup_{0\leq t \leq T} e^{\sqrt{\beta t}}|\nabla_x \Gamma_t(X_t(\mu'),\mu'_t) - \nabla_x \Gamma_t(X_t(\mu),\mu_t)|  \bigg(\int_0^T e^{\beta t} |Y_t(\mu)-Y_t(\mu')|^2\,dt\bigg)^{1/2} \bigg]\,.
\end{split}
\]
Thus with Young's inequality, assumption iii) and Lemma~\ref{lemma sde stability} we have
\[
I_4 \leq  \mathbb E^W\int_0^Te^{\beta t}|Y_t(\mu)-Y_t(\mu')|^2\,dt + C\rho_2(\mu,\mu')^2\,.
\]
To estimate the final term we use Lemma~\ref{bsde bounds}, Young's and H\"older's inequalities to see that 
\[
I_5 \leq  \mathbb E^W\int_0^T |Y_t(\mu)-Y_t(\mu')|^2 \,dt + C e^{2\beta T}\rho_2(\mu,\mu')^2\,.
\]
Note that assumption iv) and Lemma~\ref{lemma sde stability} allow us to see that
\[
e^{\beta T} \mathbb E^W|\left(\nabla_x g\right)(X_T(\pi))-\left(\nabla_x g\right)(X_T(\pi'))|^2 
\leq e^{\beta T} C\rho_2(\mu,\mu')^2\,.	
\]
Thus~\eqref{eq diff of Y and Z} and  can be estimated as 
\[
\begin{split}
& \mathbb E^W\int_{0}^T e^{\beta t} \Big(\beta |Y_t(\mu)-Y_t(\mu')|^2 + |Z_t(\mu)-Z_t(\mu')|^2\Big)\,dt  \\
& \leq (10K+4)\mathbb E^W\int_0^T e^{\beta t} |Y_t(\mu)-Y_t(\mu')|^2\,dt + \frac12\mathbb E^W\int_0^T e^{\beta t} |Z_t(\mu)-Z_t(\mu')|^2\,dt 
+ C e^{2\beta T}\rho_{2}(\mu,\mu')^{2}\,.
\end{split}
\]
Let us take $\beta = 10K + 5$.
Then 
\begin{equation}
\label{eq bsde stability proof 2}
\begin{split}
& \mathbb E^W\int_{0}^T e^{\beta t} \Big(|Y_t(\mu)-Y_t(\mu')|^2 + \tfrac12|Z_t(\mu)-Z_t(\mu')|^2\Big)\,dt \leq  C e^{2\beta T}\rho_{2}(\mu,\mu')^{2}\,.
\end{split}
\end{equation}
Now let us return to~\eqref{eq diff of Y and Z ito}, take $\beta=0$, supremum over $t\in [0,T]$ and expectation:
\[
\begin{split}
& \mathbb E^W \sup_{0\leq t\leq T} |Y_t(\mu) - Y_t(\mu')|^2 
\leq e^{\beta T} \mathbb E^W |\left(\nabla_x g\right)(X_T(\pi))-\left(\nabla_x g\right)(X_T(\pi'))|^2 \\
& + 2\mathbb E^W \bigg[\sup_{0\leq t\leq T} 	\int_t^T \big(Y_{t'}(\mu)-Y_{t'}(\mu')\big)^\top \big(Z_{t'}(\mu)-Z_{t'}(\mu')\big)\,dW_{t'} \bigg] \\
& + 2\mathbb E^W \int_0^T |Y_t(\mu)-Y_t(\mu')|\Big|\left(\nabla_x H^{0}_t\right)(X_t(\mu),Y_t(\mu),Z_t(\mu), \mu_t)-\left(\nabla_x H^{0}_t\right)(X_t(\mu'), Y_t(\mu'), Z_t(\mu'), \mu'_t)\Big|\,dt\,.
\end{split}
\]
Using the estimates we obtained for $I_1$,\ldots,$I_5$ together with Lemma~\ref{lemma sde stability} and~\eqref{eq bsde stability proof 2} thus yields
\[
\begin{split}
& \mathbb E^W \sup_{0\leq t\leq T} |Y_t(\mu) - Y_t(\mu')|^2 
\leq  2\mathbb E^W \bigg[\sup_{0\leq t\leq T} 	\int_t^T \big(Y_{t'}(\mu)-Y_{t'}(\mu')\big)^\top \big(Z_{t'}(\mu)-Z_{t'}(\mu')\big)\,dW_{t'} \bigg]  + C \rho_2(\mu,\mu')^2\,.
\end{split}
\]
Applying the inequality of Burkholder--Davis--Gundy and finally Young's inequality we obtain that
\[
\begin{split}
&\mathbb E^W \bigg[\sup_{0\leq t\leq T} 	\int_t^T \big(Y_{t'}(\mu)-Y_{t'}(\mu')\big)^\top \big(Z_{t'}(\mu)-Z_{t'}(\mu')\big)\,dW_{t'} \bigg]\\
&\leq C\mathbb E^W \bigg[\bigg(\int_0^T |Y_{t'}(\mu)-Y_{t'}(\mu')|^2|Z_{t'}(\mu)-Z_{t'}(\mu')|^2 \,dt\bigg)^{1/2}\bigg]\\
&\leq C\mathbb E^W \bigg[\sup_{0\leq t' \leq T}|Y_{t'}(\mu)-Y_{t'}(\mu')|\bigg(\int_0^T |Z_{t'}(\mu)-Z_{t'}(\mu')|^2 \,dt\bigg)^{1/2}\bigg]\\
&\leq C\gamma\mathbb E^W \sup_{0\leq t' \leq T}|Y_{t'}(\mu)-Y_{t'}(\mu')|^2+CC_\gamma\mathbb E^W\int_0^T |Z_{t'}(\mu)-Z_{t'}(\mu')|^2 \,dt\,.\\
\end{split}
\]
Hence, taking $\gamma > 0$ sufficiently small and recalling~\eqref{eq bsde stability proof 2} we get
\[
\begin{split}
& \mathbb E^W \sup_{0\leq t\leq T} |Y_t(\mu) - Y_t(\mu')|^2 
\leq C \rho_2(\mu,\mu')^2\,.
\end{split}
\]
This concludes the proof.
\end{proof}

\section*{Acknowledgements}
We would like to thank the anonymous referees and  William Hammersley (Université Côte d'Azur) for their insightful comments which undoubtedly helped us to improve the paper.

\newif\ifprintappendix
\printappendixtrue

\ifprintappendix

\appendix

\section{Measure derivatives}

We first define flat derivative on $\mathcal P_2(\mathbb R^p)$. See e.g.~\cite[Section 5.4.1]{carmona2018probabilistic} for more details.

\begin{definition}
\label{def:flatDerivative}
A functional $U:\mathcal P_2(\mathbb R^p) \to \mathbb R$ is said to admit a linear derivative
if there is a (continuous on $\mathcal P_2(\mathbb R^p)$) map $\tfrac{\delta U}{\delta m} : \mathcal P(\mathbb R^p) \times \mathbb R^d \to \mathbb R$, such that $|\tfrac{\delta U}{\delta m}(a,\mu)|\leq C(1+|a|^2)$ and, for all
$m, m' \in\mathcal P_2(\mathbb R^p)$, it holds that
\[
U(m) - U(m') = \int_0^1 \int \tfrac{\delta U}{\delta m}(m + \lambda(m' - m),a) \, (m'
-m)(da)\,d\lambda\,.
\]
Since $\tfrac{\delta U}{\delta m}$ is only defined up to a constant we make a choice by demanding $\int \tfrac{\delta U}{\delta m}(m,a)\,m(da) = 0$.
	
\end{definition}

We will also need the linear functional derivative on 
\begin{equation*}
\mathcal M_2 :=\Big\{ \nu \in \mathscr M([0,T]\times \mathbb R^p): \nu(dt,da)=\nu_t(da)dt, \, \nu_t \in \mathcal P_2(\mathbb R^p),\,  \int_{0}^{T}\!\!\!\int |a|^2\nu_t(da)dt < \infty \Big\}\,,
\end{equation*}
which provides a slight extension of the one introduced above in Definition~\ref{def:flatDerivative}.

\begin{definition}
\label{def:ExtendedDerivative}
A functional $F:\mathcal V_2^W \to \mathbb R^d$, is said to admit a first order linear derivative, if there exists a functional $\tfrac{\delta F}{\delta {{\nu}}}:\mathcal V_2^W\times \Omega^W \times (0,T)\times\mathbb R^p\rightarrow \mathbb R^d$, such that
\begin{enumerate}[i)]
\item For all $(\omega^W, t,a)\in \Omega^W\times (0,T)\times\mathbb R^p$, $\mathcal V_2^W \ni {{\nu}}\mapsto \tfrac{\delta F}{\delta \nu}({\nu},\omega^W,t,a,)$ is continuous (for $\mathcal V_2^W$ endowed with the weak topology of $\mathscr M^+_b(\Omega^W\times (0,T)\times\mathbb R^p)$).
\item For any $\nu \in \mathcal V_2^W$ there exists $C=C_{\nu,T,d,p} >0$ such that for all $a\in \mathbb R^p$ we have that
\[\left|\tfrac{\delta F}{\delta \nu}({\nu},\omega^W,t,a)\right|\leq C(1+|a|^2)\,.
\]
\item For all ${{\nu}},{\nu'}\in\mathcal V_2^W$,
\begin{equation}\label{def:FlatDerivativeOnVa}
F({\nu'})-F({{\nu}})=\int_{0}^{1} \mathbb E^W\int_0^T\int\tfrac{\delta F}{\delta {{\nu}}}((1-\lambda){{\nu}}+\lambda {\nu'},t,a)\left({\nu'_t}-{{\nu_t}}\right)(da)\,dt\,d\lambda.
\end{equation}
\end{enumerate}
The functional $\tfrac{\delta F}{\delta {{\nu}}}$ is then called the linear (functional) derivative of $F$ on $\mathcal V_2^W$.
\end{definition}
The linear derivative $\tfrac{\delta F}{\delta \nu}$ is here also defined up to the additive constant $\mathbb E^W\int_0^T\int\tfrac{\delta F}{\delta \nu}({\nu},t,a){{\nu_t}}(da)\,dt$. 
By a centering argument, $\tfrac{\delta F}{\delta\nu}$ can be generically defined under the assumption that
$\mathbb E^W\int_0^T\int\tfrac{\delta F}{\delta \nu}({\nu},t,a){{\nu_t}}(da)\,dt=0$.
Note that if $\tfrac{\delta F}{\delta \nu}$ exists according to Definition~\ref{def:ExtendedDerivative} then
\begin{equation}\label{def:FlatDerivativeOnVb}
\forall\,\nu,{\nu'}\in\mathcal V_2^W,\,\lim_{\epsilon\rightarrow 0^+}\frac{F(\nu+\epsilon(\nu'-\nu))-F(\nu)}{\epsilon}=\mathbb E^W\int_0^T\int\tfrac{\delta F}{\delta \nu}({\nu},t,a)\left(\nu'_t-\nu_t\right)(da)\,dt.
\end{equation}
Indeed~\eqref{def:FlatDerivativeOnVa} immediately implies~\eqref{def:FlatDerivativeOnVb}.
To see the implication in the other direction take $v^\lambda := \nu + \lambda(\nu'-\nu)$
and $\nu'^\lambda := \nu' - \nu + \nu^\lambda$ and notice that~\eqref{def:FlatDerivativeOnVb} ensures for all $\lambda \in [0,1]$ that
\[
\begin{split}
& \lim_{\varepsilon\rightarrow 0^+}\frac{F(\nu^\lambda+\varepsilon(\nu'-\nu))-F(\nu^\lambda)}{\varepsilon}
= \lim_{\varepsilon\rightarrow 0^+}\frac{F(\nu^\lambda+\varepsilon(\nu'^\lambda-\nu^\lambda))-F(\nu^\lambda)}{\varepsilon} \\
& =\mathbb E^W\int_0^T\int\tfrac{\delta F}{\delta \nu}(\nu^\lambda,t,a)\big({\nu'^\lambda_t}-{{\nu^\lambda_t}}\big)(da)\,dt
= \mathbb E^W\int_0^T\int\tfrac{\delta F}{\delta {{\nu}}}(\nu^\lambda,t,a)\left({\nu'_t}-{{\nu_t}}\right)(da)\,dt
\,.	
\end{split}
\]
By the fundamental theorem of calculus
\[
F({\nu'})-F({{\nu}})=\int_0^1\lim_{\varepsilon\rightarrow 0^+}\frac{F(\nu^{\lambda+\varepsilon})-F(\nu^{\lambda})}{\varepsilon}\,d\lambda
=\int_0^1\mathbb E^W\int_{0}^{T}\int\tfrac{\delta F}{\delta\nu}(\nu^\lambda,t,a)({\nu'_t}-\nu_t)(da)dt\,d\lambda\,.
\]

\begin{lemma} \label{lem constant}
Fix $m \in \mathcal P(\mathbb R^p)$.
Let $u: \mathbb R^p \to \mathbb R$ be such that for all $m' \in \mathcal P(\mathbb R^p)$ we have that
\[
0 \leq \int u(a)\,(m' - m)(da)\,.
\]
Then for all $a\in \mathbb R^p$ we have $u(a) = \int u(a')m(da')$.
\end{lemma}

\begin{proof}[Proof of Lemma~\ref{lem constant}]
Let $M := \int u(a) \,m(da)$.
Fix $\varepsilon > 0$.
Assume that $m(\{a:u(a) - M \leq -\varepsilon\}) > 0$.
Take $dm' := \frac{1}{m(\{u - M \leq -\varepsilon\})} \mathds{1}_{\{u - M \leq -\varepsilon\}}\,dm$.
Then
\[
\begin{split}
0 & \leq \int u(a)\,(m'-m)(da) = \int [ u(a) - M ]\,m'(da) 	\\
& = \int \mathds{1}_{\{u - M \leq -\varepsilon\}} [u(a)-M]\,m'(da)
+ \int \mathds{1}_{\{u - M > -\varepsilon\}} [u(a)-M]\,m'(da)\\
& = \int \mathds{1}_{\{u - M \leq -\varepsilon\}} [u(a)-M] \frac1{m(\{u - M \leq -\varepsilon\})} m(da) \leq -\varepsilon\,.
\end{split}
\] 	
As this is a contradiction we get $m(\{u - M \leq -\varepsilon\}) = 0$ and taking $\varepsilon \to 0$ we get $m(\{u - M < 0\}) = 0$.
On the other hand assume that $m(\{u-M\geq \varepsilon\}) > 0$.
Then, since $u - M \geq 0$ holds $m$-a.s., we have
\[
0 = \int [u(a) - M]\,m(da) \geq \int_{\{u-M\geq \varepsilon\}} [u(a) - M]\,m(da) \geq \varepsilon m(u-M\geq \varepsilon)>0
\]
which is again a contradiction meaning that for all $\varepsilon > 0$ we have
$m(u-M\geq \varepsilon) = 0$ i.e. $u=M$ $m$-a.s..
\end{proof}

\begin{lemma}
\label{lem flat der foc}
Let $F: \mathcal P_2(\mathbb R^p) \to \mathbb R$.
Let $m^\ast \in \argmin_{m \in \mathcal P_2(\mathbb R^p)} F(m)$. 
Assume that $\tfrac{\delta F}{\delta m}$ exists and for all $m$ and all $a$ we have $\tfrac{\delta F}{\delta m}(m,a) \leq |a|^2$.
Then $\tfrac{\delta F}{\delta m}(m^\ast,\cdot)$ is a constant function.
\end{lemma}
\begin{proof}
Let $m\in \mathcal P_2(\mathbb R^p)$ be arbitrary. 
Let $m^\varepsilon := (1-\varepsilon)m^\ast + \varepsilon m$.
Clearly $0\leq F(m^\varepsilon) - F(m^\ast)$. 
Hence
\[
0 \leq \frac1\varepsilon\Big(F(m^\varepsilon) - F(m^\ast)\Big) = \int_0^1 \int \tfrac{\delta F}{\delta m}((1-\lambda)m^\ast + \lambda m^\varepsilon,a)(m-m^\ast)(da)\,d\lambda\,.
\]
By reverse Fatou's lemma 
\[
0 \leq \limsup_{\varepsilon \to 0} \int_0^1 \int \tfrac{\delta F}{\delta m}((1-\lambda)m^\ast + \lambda m^\varepsilon,a)(m-m^\ast)(da)\,d\lambda \leq \int \tfrac{\delta F}{\delta m}(m^\ast,a)(m-m^\ast)(da)\,.
\]
Using Lemma~\ref{lem constant} we conclude $\tfrac{\delta F}{\delta m}(m^\ast,\cdot)$ is a constant function.
\end{proof}

\begin{lemma}
\label{lemma L2overmeasurescomplete}
Let $(X,d)$ be a complete metric space and let $(\Omega,\mathcal F, \mathcal P)$ be a probability space. Let $p\geq 1$. Let $\|\cdot\|_p$ denote the norm in $L^p(\Omega;\mathbb R)$.
Let $\mu_0 : \Omega \to X$ be a random variable. Let
\[
S:= \{ \mu : \Omega \to X \,\,\text{r.v.}\,\,: \|d(\mu,\mu_0)\|_p < \infty\}\,.
\]	
Let $\rho(\mu,\mu') := \|d(\mu,\mu')\|_p$.
Then $(S,\rho)$ is a complete metric space.
\end{lemma}

\begin{proof}[Proof of Lemma~\ref{lemma L2overmeasurescomplete}]
Consider a Cauchy sequence $(\mu_n)_{n\in \mathbb N} \subset S$. 
Then there exists a subsequence $(\mu_{n(k)})_{k\in \mathbb N}$ such that 
$\rho(\mu_{n(k)}, \mu_{n(k+1)}) \leq 4^{-k/p}$. 
Hence 
$\mathbb E \big[d(\mu_{n(k)}, \mu_{n(k+1)})^p\big] \leq 4^{-k}$.

Let $A_k := \{\omega : d(\mu_{n(k)}, \mu_{n(k+1)})^p \geq 2^{-k} \}$. 
Then, due to Chebychev's inequality, 
\[
\mathbb P(A_k) \leq 2^k \mathbb E \big[  d(\mu_{n(k)}, \mu_{n(k+1)})^p \big] \leq 2^{-k}\,.
\] 
By the Borel-Cantelli lemma $\mathbb P(\limsup_{k\to \infty} A_k) = 0$.
This means that for almost all $\omega \in \Omega$ there is $K(\omega)$ such that for all $k\geq K(\omega)$ we have $d(\mu_{n(k)}(\omega), \mu_{n(k+1)}(\omega)) \leq 2^{-k}$.
Hence for any $k\geq K(\omega)$ and any $j\geq 1$, by the triangle inequality, 
\[
d(\mu_{n(k)}(\omega), \mu_{n(k+j)}(\omega)) \leq \sum_{j'=0}^{j-1} d(\mu_{n(k+j')}(\omega), \mu_{n(k+j'+1)}(\omega)) \leq 2^{-k}\sum_{j'=0}^{j-1} 2^{-j'} \leq 2\cdot2^{-k}\,.
\]
This means that almost surely $(\mu_{n(k)})_{k\in \mathbb N}$ is a Cauchy sequence in the complete space $(X,d)$ and thus it has a limit $\mu$.

Recall that $(\mu_n)_{n\in \mathbb N} \subset S$ is Cauchy in $(S,\rho)$. 
Hence for any $\varepsilon > 0$ there is $N\in \mathbb N$ so that for all $m,n > N$ we have $\rho(\mu_m,\mu_n) < \varepsilon$.
Moreover for $n > N$ we have, due to Fatou's Lemma, that
\[
\mathbb E\big[ d(\mu,\mu_n)^p \big] = \mathbb E\big[ \liminf_{k\to \infty} d(\mu_{n(k)},\mu_n)^p\big] \leq \liminf_{k\to \infty} \mathbb E\big[d(\mu_{n(k)},\mu_n)^p\big] < \varepsilon^p\,.
\]
Thus for all $n> N$ and so $\rho(\mu,\mu_n) = \|d(\mu,\mu_n)\|_p < \varepsilon$.
In other words $\rho(\mu,\mu_n) \to 0$ as $n\to \infty$.

Finally, we will show that $\mu \in S$. 
Indeed for sufficiently large $N$ we have $\|d(\mu, \mu_N)\|_p \leq 1$ and so 
\[
\|d(\mu_0,\mu)\|_p \leq \|d(\mu_0,\mu_N)\|_p + \|d(\mu_N,\mu)\|_p \leq \|d(\mu_0,\mu_N)\|_p + 1 < \infty\,.
\]
Hence $(S,\rho)$ is complete.
\end{proof}

\section{Sufficient condition for optimality}
\label{sec pont sufficient cond}
The main results of the article do not use the following Pontryagin sufficient condition for optimality but we include it for completeness.

\begin{theorem}[Sufficient condition for optimality]
\label{thm sufficient condition}
Fix $\sigma \geq 0$. 
Assume that $g$ and $H^0$ are continuously differentiable in the $x$ variable. 
Assume that $\nu\in \mathcal V_2^W$, $X$, $Y$, $Z$, are a solution to~\eqref{eq process}-\eqref{eq adjoint proc}
such that
\[
\nu_t \in \argmin_{m \in \mathcal P_2(\mathbb R^d)} H^\sigma(X_t, Y_t, Z_t, m, \gamma_t)\,.
\]
Finally assume that 
\begin{enumerate}[i)]
\item the map $x\mapsto g(x)$ is convex and
\item \label{item ii in suff pont}the map $(x,m) \mapsto H^\sigma(x, Y_t, Z_t, m,\gamma_t)$ is convex for a.e. $(t,\omega)$, in the sense that for all $x,x'\in \mathbb R^d$ and all $m,m' \in \mathcal P(\mathbb R^p)$ (absolutely continuous w.r.t. the Lebesgue measure if $\sigma >0$) it holds that 
\[
\begin{split}
& H^\sigma(x,Y_t,Z_t,m,\gamma_t) - H^\sigma(x',Y_t,Z_t,m',\gamma_t) \\
& \leq (\nabla_x H^\sigma)(x,Y_t,Z_t,m)(x-x') + \int \tfrac{\delta H^0}{\delta m}(x, Y_t, Z_t, m, a) (m - m')(da)+\tfrac{\sigma^2}{2}\int (\log m(a) - \log \gamma_t(a)) (m-m')(da)\,.	
\end{split}
\]  	
\end{enumerate}
Then the control $\nu\in \mathcal V_2^W$ is an optimal control (and if $g$ or $H^\sigma$ are strictly convex then it is the optimal control).
\end{theorem}

\begin{proof}[Proof of Theorem~\ref{thm sufficient condition}.]
Let $(\tilde \nu_t)_{t\in [0,T]}$ be another control with the associated family of
forward and backward processes $\tilde X$, $\tilde Y$, $\tilde Z$, $\xi \in \mathbb R^d$.
Of course $X_0 = \tilde X_0$.
First, we note that due to convexity of $x\mapsto g(x)$  we have
\[
\begin{split}
\mathbb E^W & \left[ g(X - g(\tilde X_T) \right] \, 
\leq \mathbb E^W \left[ (\nabla_x g)(X_T)(X_T - \tilde X_T) \right] \\
& = \mathbb E^W\left[ Y_T (X_T - \tilde X_T) \right] \\
  & = \mathbb E^W \left[ \int_0^T (X_t - \tilde X_t)\,dY_t + \int_0^T Y_t(dX_t - d\tilde X_t) + \int_0^T d\langle Y_t, X_t - \tilde X_t\rangle \right]\\
  & = - \mathbb E^W \int_0^T (X_t - \tilde X_t)(\nabla_x H)(X_t,Y_t,\nu_t)\,dt 
+ \mathbb E^W \int_0^T Y_t\left(\Phi(X_t,\nu_t) - \Phi(\tilde X_t,\tilde \nu_t)\right)\,dt\\
& + \mathbb E^W \int_0^T \text{tr}\big[ \big(\Gamma(X_t, \nu_t) - \Gamma(\tilde X_t, \tilde \nu_t)\big)^\top Z_t \big]\,dt\,. 
\end{split}
\]
Moreover, since $F(x,\nu) + \tfrac{\sigma^2}{2}R(\nu|\gamma) = H^\sigma(x,y,z,\nu) - \Phi(x,\nu)\,y - \text{tr}[\Gamma(x,\nu)^\top z]$  we have
\[
\begin{split}
 \int_0^T &  \left[F(X_t,\nu_t) - F(\tilde X_t,\tilde \nu_t) +\tfrac{\sigma^2}{2}\gamma_t - \tfrac{\sigma^2}{2}R(\tilde \nu_t|\gamma_t) \right]\,dt 	\\
 = &  \int_0^T \Big[H^\sigma(X_t,Y_t,Z_t,\nu_t) - \Phi(X_t,\nu_t)Y_t - \text{tr}[\Gamma(X_t,\nu_t)^\top Z_t] \\
& - H^\sigma(\tilde X_t,Y_t,Z_t,\tilde \nu_t) + \Phi(\tilde X_t,\tilde \nu_t)Y_t + \text{tr}[\Gamma(\tilde X_t,\tilde \nu_t)^\top Z_t] \Big]\,dt \,. 	\\
\end{split}
\]
Hence
\begin{equation}
\label{eq pont suf proof}
\begin{split}
&J^\sigma(\nu) - J^\sigma(\tilde \nu) \\ 
&\leq -\mathbb E^W\int_0^T (X_t - \tilde X_t)(\nabla_x H^\sigma)(X_t,Y_t,\nu_t)\,dt + \mathbb E^W \int_0^T \left[H^\sigma(X_t,Y_t,\nu_t) - H^\sigma(\tilde X_t,Y_t,\tilde \nu_t)\right]\,dt	\,.
\end{split}	
\end{equation}

We are assuming that $(x,\mu) \mapsto H^\sigma(x,Y_t,Z_t,\mu)$ is jointly convex in the sense of flat derivatives and so we have
\[
\begin{split}
& H^\sigma(X_t,Y_t,Z_t,\nu_t) - H^\sigma(\tilde X_t,Y_t,Z_t,\tilde \nu_t) \\
& \leq (\nabla_x H^0)(X_t,Y_t,Z_t,\nu_t)(X_t-\tilde X_t) + \int \tfrac{\delta H^\sigma}{\delta m}(X_t, Y_t, Z_t, \nu_t, a) (\nu_t - \tilde \nu_t)(da)+\tfrac{\sigma^2}{2}\int (\log \nu_t(a) - \log \gamma_t(a)) (\nu_t-\tilde \nu_t')(da)\,.	
\end{split}
\]
We thus have 
\[
\begin{split}
J^\sigma(\nu) - J^\sigma(\tilde \nu) \leq  \mathbb E^W \int_0^T \int  \bigg( \tfrac{\delta H^0}{\delta m}(X_t, Y_t, Z_t, \nu_t, a) +\tfrac{\sigma^2}{2}(\log\nu_t(a) -\log \gamma_t(a))\bigg)\,  (\nu_t - \tilde \nu_t)(da)	\,dt\,.
\end{split}
\]
The assumption $\nu_t = \argmin_{m} H^\sigma(X_t,Y_t,Z_t,m,\gamma_t)$ together with Lemma~\ref{lem flat der foc} implies that 
\[
a \mapsto \left(\tfrac{\delta H^0}{\delta m}(X_t, Y_t, Z_t, \nu_t, a)+\tfrac{\sigma^2}{2}(\log\nu_t(a) -\log \gamma_t(a))\right)
\] is a constant function and hence $\int   \left(\tfrac{\delta H^0}{\delta m}(X_t, Y_t, Z_t, \nu_t, a)+\tfrac{\sigma^2}{2}(\log\nu_t(a) -\log \gamma_t(a))\right) (\nu_t - \tilde \nu_t)(da)	 = 0$. 
This implies that $J^\sigma(\nu) - J^\sigma(\tilde \nu) \leq 0$ and so $\nu$ is an optimal control.
We note that if either $g$ or $H$ are strictly convex then $\nu$ is the optimal control.	
\end{proof}

\else
\fi

\bibliographystyle{abbrv}
\bibliography{Bibliography.bib}

\begin{thebibliography}{10}

\bibitem{bellman1966dynamic}
R.~Bellman.
\newblock Dynamic programming.
\newblock {\em Science}, 153(3731):34--37, 1966.

\bibitem{benamou2000computational}
J.-D. Benamou and Y.~Brenier.
\newblock A computational fluid mechanics solution to the
  {M}onge--{K}antorovich mass transfer problem.
\newblock {\em Numerische Mathematik}, 84(3):375--393, 2000.

\bibitem{bensoussan2004stochastic}
A.~Bensoussan.
\newblock {\em Stochastic control of partially observable systems}.
\newblock Cambridge University Press, 2004.

\bibitem{bensoussan2011applications}
A.~Bensoussan and J.-L. Lions.
\newblock {\em Applications of variational inequalities in stochastic control}.
\newblock Elsevier, 2011.

\bibitem{bertsekas1995dynamic}
D.~P. Bertsekas.
\newblock {\em Dynamic programming and optimal control}.
\newblock Athena scientific Belmont, MA, 1995.

\bibitem{bertsekas2004stochastic}
D.~P. Bertsekas and S.~Shreve.
\newblock {\em Stochastic optimal control: the discrete-time case}.
\newblock 2004.

\bibitem{carmona2018probabilistic}
R.~Carmona and F.~Delarue.
\newblock {\em Probabilistic Theory of Mean Field Games with Applications
  I-II}.
\newblock Springer, 2018.

\bibitem{chassagneux2021learning}
J.-F. Chassagneux, J.~Chen, N.~Frikha, and C.~Zhou.
\newblock A learning scheme by sparse grids and {P}icard approximations for
  semilinear parabolic {PDE}s.
\newblock {\em IMA Journal of Numerical Analysis}, 43(5):3109--3168, 2023.

\bibitem{chassagneux2019weak}
J.-F. Chassagneux, L.~Szpruch, and A.~Tse.
\newblock Weak quantitative propagation of chaos via differential calculus on
  the space of measures.
\newblock {\em The Annals of Applied Probability}, 32(3):1929--1969, 2022.

\bibitem{delarue2021uniform}
F.~Delarue and A.~Tse.
\newblock Uniform in time weak propagation of chaos on the torus.
\newblock {\em arXiv preprint arXiv:2104.14973}, 2021.

\bibitem{doya2000reinforcement}
K.~Doya.
\newblock Reinforcement learning in continuous time and space.
\newblock {\em Neural Computation}, 12:219--245, 2000.

\bibitem{emmrich2017nonlinear}
E.~Emmrich and D.~{\v{S}}i{\v{s}}ka.
\newblock Nonlinear stochastic evolution equations of second order with
  damping.
\newblock {\em Stoch PDE: Anal Comp}, 5(1), 2017.

\bibitem{fleming2006controlled}
W.~H. Fleming and H.~M. Soner.
\newblock {\em Controlled Markov processes and viscosity solutions}.
\newblock Springer, 2006.

\bibitem{geist2019theory}
M.~Geist, B.~Scherrer, and O.~Pietquin.
\newblock A theory of regularized {M}arkov decision processes.
\newblock In {\em International Conference on Machine Learning}, pages
  2160--2169. PMLR, 2019.

\bibitem{gobet2021newton}
E.~Gobet and M.~Grangereau.
\newblock Newton method for stochastic control problems.
\newblock 2021.

\bibitem{harter2019stability}
J.~Harter and A.~Richou.
\newblock A stability approach for solving multidimensional quadratic {BSDE}s.
\newblock {\em Electronic Journal of Probability}, 4(24):1--51, 2019.

\bibitem{hu2019meanode}
K.~Hu, A.~Kazeykina, and Z.~Ren.
\newblock Mean-field {L}angevin system, optimal control and deep neural
  networks.
\newblock {\em arXiv:1909.07278}, 2019.

\bibitem{hu2019Mean}
K.~Hu, Z.~Ren, D.~{\v{S}}i{\v{s}}ka, and {\L}.~Szpruch.
\newblock Mean-field {L}angevin dynamics and energy landscape of neural
  networks.
\newblock {\em Annales de l'Institut Henri Poincare (B) Probabilites et
  statistiques}, 57(4):2043--2065, 2021.

\bibitem{huang2022convergence}
Y.-J. Huang, Z.~Wang, and Z.~Zhou.
\newblock Convergence of policy improvement for entropy-regularized stochastic
  control problems.
\newblock {\em arXiv preprint arXiv:2209.07059}, 2022.

\bibitem{ito2021neural}
K.~Ito, C.~Reisinger, and Y.~Zhang.
\newblock A neural network-based policy iteration algorithm with global
  {$H^2$}-superlinear convergence for stochastic games on domains.
\newblock {\em Foundations of Computational Mathematics}, 21:331--374, 2021.

\bibitem{jabir2019mean}
J.-F. Jabir, D.~{\v{S}}i{\v{s}}ka, and {\L}.~Szpruch.
\newblock Mean-field neural {ODE}s via relaxed optimal control.
\newblock {\em arXiv:1912.05475}, 2019.

\bibitem{Kazamaki}
N.~{Kazamaki}.
\newblock {\em Continuous Exponential Martingales and BMO}.
\newblock Springer-Verlag Berlin Heidelberg, 1994.

\bibitem{kerimkulov2021modified}
B.~Kerimkulov, D.~{\v{S}}i{\v{s}}ka, and L.~Szpruch.
\newblock A modified {MSA} for stochastic control problems.
\newblock {\em Applied Mathematics \& Optimization}, pages 1--20, 2021.

\bibitem{kerimkulov2021regularized}
B.~Kerimkulov, D.~{\v{S}}i{\v{s}}ka, {\L}.~Szpruch, and Y.~Zhang.
\newblock Mirror descent for stochastic control problems with measure-valued
  controls.
\newblock {\em arXiv preprint arXiv:2401.01198}, 2024.

\bibitem{kerimkulov2020exponential}
B.~Kerimkulov, D.~\v{S}i\v{s}ka, and L.~Szpruch.
\newblock Exponential convergence and stability of {H}oward's policy
  improvement algorithm for controlled diffusions.
\newblock {\em SIAM Journal on Control and Optimization}, 58(3):1314--1340,
  2020.

\bibitem{komorowski2012central}
T.~Komorowski and A.~Walczuk.
\newblock Central limit theorem for {M}arkov processes with spectral gap in the
  {W}asserstein metric.
\newblock {\em Stochastic Processes and their Applications}, 122(5):2155--2184,
  2012.

\bibitem{krylov1980controlled}
N.~V. Krylov.
\newblock {\em Controlled diffusion processes}.
\newblock Springer, 1980.
\newblock Translated from the Russian by A. B. Aries.

\bibitem{krylov1981stochastic}
N.~V. Krylov and B.~L. Rozovskii.
\newblock Stochastic evolution equations.
\newblock {\em Journal of Soviet Mathematics}, (14):1233--1277, 1981.

\bibitem{LSU68}
O.~A. Ladyzenskaja, V.~A. Solonnikov, and N.~N. Ural'ceva.
\newblock {\em Linear and quasi-linear equations of parabolic type}.
\newblock Translations of Mathematical Monographs. AMS, 1968.

\bibitem{majka2017coupling}
M.~B. Majka.
\newblock Coupling and exponential ergodicity for stochastic differential
  equations driven by {L}\'evy processes.
\newblock {\em Stochastic Processes and their Applications},
  127(12):4083--4125, 2017.

\bibitem{OV00}
F.~Otto and C.~Villani.
\newblock Generalization of an inequality by {T}alagrand and links with the
  logarithmic {S}obolev inequality.
\newblock {\em Journal of Functional Analysis}, 173:361--400, 2000.

\bibitem{reisinger2020regularity}
C.~Reisinger and Y.~Zhang.
\newblock Regularity and stability of feedback relaxed controls.
\newblock {\em SIAM Journal on Control and Optimization}, 59(5):3118--3151,
  2021.

\bibitem{sutton2018reinforcement}
R.~S. Sutton and A.~G. Barto.
\newblock {\em Reinforcement learning: An introduction}.
\newblock MIT press, 2018.

\bibitem{szpruch2019antithetic}
{\L}.~Szpruch and A.~Tse.
\newblock Antithetic multilevel particle system sampling method for
  {M}c{K}ean--{V}lasov sdes.
\newblock {\em arXiv preprint arXiv:1903.07063}, 2019.

\bibitem{tang2022exploratory}
W.~Tang, Y.~P. Zhang, and X.~Y. Zhou.
\newblock Exploratory {HJB} equations and their convergence.
\newblock {\em SIAM Journal on Control and Optimization}, 60(6):3191--3216,
  2022.

\bibitem{villani2008optimal}
C.~Villani.
\newblock {\em Optimal transport: old and new}.
\newblock Springer, 2008.

\bibitem{wang2019exploration}
H.~Wang, T.~Zariphopoulou, and X.~Y. Zhou.
\newblock Exploration versus exploitation in reinforcement learning: a
  stochastic control approach.
\newblock {\em Available at SSRN 3316387}, 2019.

\bibitem{young2000lectures}
L.~C. Young.
\newblock {\em Lectures on the calculus of variations and optimal control
  theory}, volume 304.
\newblock American Mathematical Soc., 2000.

\bibitem{zhang2017backward}
J.~Zhang.
\newblock {\em Backward stochastic differential equations}.
\newblock Springer, 2017.

\bibitem{ziebart2010modeling}
B.~D. Ziebart.
\newblock {\em Modeling purposeful adaptive behavior with the principle of
  maximum causal entropy}.
\newblock Carnegie Mellon University, 2010.

\end{thebibliography}

\end{document}